\documentclass{amsart}

\usepackage[all]{xy}
\usepackage{amsmath}
\usepackage{hyperref}
\usepackage{amsfonts,graphics,amsthm,amsfonts,amscd,latexsym}
\usepackage{epsfig}
\usepackage{flafter}
\usepackage{mathtools}
\usepackage{comment}

\hypersetup{
    colorlinks=true,    
    linkcolor=blue,          
    citecolor=blue,      
    filecolor=blue,      
    urlcolor=blue           
}
\usepackage{tikz}
\usetikzlibrary{graphs,positioning,arrows,shapes.misc,decorations.pathmorphing}

\tikzset{
    >=stealth,
    every picture/.style={thick},
    graphs/every graph/.style={empty nodes},
}

\tikzstyle{vertex}=[
    draw,
    circle,
    fill=black,
    inner sep=1pt,
    minimum width=5pt,
]
\usepackage[position=top]{subfig}
\usepackage[alphabetic,backrefs]{amsrefs}
\usepackage{amssymb}
\usepackage{color}

\calclayout

\usetikzlibrary{decorations.pathmorphing}
\tikzstyle{printersafe}=[decoration={snake,amplitude=0pt}]

\newcommand{\codim}{\operatorname{codim}}

\newcommand{\mult}{\operatorname{mult}}

\newcommand{\Spec}{\operatorname{Spec}}

\newcommand{\pp}{\mathbb{P}}

\renewcommand{\qq}{\mathbb{Q}}

\newcommand{\nn}{\mathbb{N}}

\newcommand{\cc}{\mathbb{C}}

\def\O#1.{\mathcal {O}_{#1}}			
\def\pr #1.{\mathbb P^{#1}}				
\def\af #1.{\mathbb A^{#1}}			
\def\ses#1.#2.#3.{0\to #1\to #2\to #3 \to 0}	
\def\xrar#1.{\xrightarrow{#1}}			
\def\K#1.{K_{#1}}						
\def\bA#1.{\mathbf{A}_{#1}}			
\def\bM#1.{\mathbf{M}_{#1}}				
\def\bL#1.{\mathbf{L}_{#1}}				
\def\bB#1.{\mathbf{B}_{#1}}				
\def\bK#1.{\mathbf{K}_{#1}}			
\def\subs#1.{_{#1}}					
\def\sups#1.{^{#1}}

\DeclareMathOperator{\Supp}{Supp}

\DeclareMathOperator{\lct}{lct}

\DeclareMathOperator{\Diff}{Diff}

\newcommand{\rar}{\rightarrow}	
\newcommand{\drar}{\dashrightarrow}	
\newcommand{\rddown}[1]{\left\lfloor{#1}\right\rfloor}

\usepackage{tikz}
\usetikzlibrary{matrix,arrows,decorations.pathmorphing}

  \newtheorem{introthm}{Theorem}
  
  \newtheorem{introcor}{Corollary}
  
  \newtheorem{intrormk}{Remark}
  
  \newtheorem{theorem}{Theorem}[section]
  \newtheorem{lemma}[theorem]{Lemma}
  \newtheorem{proposition}[theorem]{Proposition}
  \newtheorem{corollary}[theorem]{Corollary}
  
  \newtheorem{conjecture}[theorem]{Conjecture}

\theoremstyle{definition}
  
  \newtheorem{notation}[theorem]{Notation}
  \newtheorem{definition}[theorem]{Definition}
  \newtheorem{example}[theorem]{Example}

\newtheorem{remark}[theorem]{Remark}

\theoremstyle{remark}

\numberwithin{equation}{section}

\begin{document}

\title[Log canonical $3$-fold complements]{Log canonical $3$-fold complements}

\author[S.~Filipazzi]{Stefano Filipazzi}
\address{
EPFL, SB MATH CAG, MA C3 625 (B\^{a}timent MA), Station 8, CH-1015 Lausanne, Switzerland
}
\email{stefano.filipazzi@epfl.ch}

\author[J.~Moraga]{Joaqu\'in Moraga}
\address{Department of Mathematics, Princeton University, Fine Hall, Washington Road, Princeton, NJ 08544-1000, USA
}
\email{jmoraga@princeton.edu}

\author[Y.~Xu]{Yanning Xu}
\address{
DPMMS, center for Mathematical Sciences, University of Cambridge, Wilberforce Road, Cambridge, UK}
\email{yx264@cam.ac.uk}

%%% to fix ams new release of subject classification
\makeatletter
\@namedef{subjclassname@2020}{
  \textup{2020} Mathematics Subject Classification}
\makeatother

\maketitle

\begin{abstract}
We expand the theory of log canonical $3$-fold complements.
We prove that if $X\rightarrow T$ is a $3$-dimensional contraction of log Calabi-Yau type, then we can find $B\geq 0$ on $X$ for which $(X,B)$ is log canonical and
$n(K_X+B)\sim_T 0$, where $n$ is an uniform natural number. 
This means that every $3$-fold of log Calabi-Yau type can be turned into a log Calabi-Yau pair in an effective way.
\end{abstract}

\setcounter{tocdepth}{1}
\tableofcontents

\section{Introduction}

The idea of complements originates in Shokurov's work in the late 1970s~\cite{Sho79}.
In this paper, Shokurov proves that if $X$ is a Fano $3$-fold, then $|-K_X|$ has a smooth element $D$. 
This smooth element can be considered as a ``complement'' of the canonical divisor, as $K_X+D\sim 0$ holds.
The theory of complements was used to control the index of semistable degenerations of varieties~\cites{Sho97,Pro99}.
Given a contraction $X\rightarrow Z$ of Fano-type over $z\in Z$, 
the theory of complements predicts the existence of a positive integer $n$, 
such that $|-nK_X/Z|$ contains an element with good singularities around $z\in Z$.
This statement is usually known as boundedness of complements for Fano-type morphisms 
and it is expected that there exists such integer $n$ that only depends on the dimension of $X$. 
It is expected that we can weaken the Fano-type condition of the morphism $X\rightarrow Z$, 
to the existence of a log canonical $\qq$-complement.
This means that if for some positive integer $m$, we can find an element with good singularities in $|-mK_X/Z|$,
then we can also find an element with good singularities in $|-nK_X/Z|$,
where $n$ only depends on $\dim X$.
This statement is usually known as boundedness of complements for $\qq$-complemented pairs.
If we can find a log canonical complement, then we can do it effectively.

Complements were rigorously defined first in ~\cites{Sho96}. Prokhorov and Shokurov then developed the theory of log canonical surface complements in more detail~\cites{Sho97,Pro99}.
Some partial results are known in dimension three~\cites{Pro00,Pro01a,Pro01b,Fuj01}.
They also introduced some inductive schemes towards the existence of bounded complements for Fano-type varieties~\cites{PS01,PS09}.
The boundedness of complements for Fano-type morphisms was proved by Birkar~\cite{Bir16a}, 
and it was used to prove the BAB conjecture regarding the boundedness of Fano varieties~\cite{Bir16b}.
The existence of complements without the Fano-type assumption is expected to be considerably harder:
indeed, this condition guarantees that $X$ has klt singularities and $-K_X$ has positivity properties.
The third author proved some initial results towards the boundedness of complements for $\qq$-complemented pairs in dimension three~\cites{Xu19a,Xu19b,thesis}.
In particular, in ~\cite{thesis}, he proved the existence and boundedness of complements for projective log canonical Fano $3$-folds.
In this article, we expand the theory of log canonical $3$-fold complements.
The following theorem settles the existence of bounded complements for $\qq$-complemented $3$-folds with rational coefficients lying in a DCC set with rational accumulation points.

\begin{introthm}\label{main-theorem}
Let $\Lambda \subset \qq$ be a set satisfying the descending chain condition with rational accumulation points.
There exists a natural number $n$ only depending on $\Lambda$ that satisfies the following.
Let $X\rightarrow T$ be a contraction between normal quasi-projective varieties such that
\begin{itemize}
    \item $(X,B)$ is a log canonical $3$-fold;
    \item $(X,B)$ is $\qq$-complemented over $t\in T$; and 
    \item the coefficients of $B$ belong to $\Lambda$.
\end{itemize}
Then, up to shrinking $T$ around $t$, we can find
\[
0 \leq \Gamma \sim_{T} -n(K_X+B)
\]
such that $(X,B+\Gamma/n)$ is a log canonical pair.
\end{introthm}

\begin{intrormk} \label{intrormk}
{\em
    One relevant instance of Theorem \ref{main-theorem} is when $\Lambda$ is a set of hyper-standard coefficients.
    More precisely, we can consider $\Lambda = \Phi(\mathcal R)$, where $\mathcal R \subset [0,1]$ is a finite set of rational numbers.
    Indeed, the key step to prove Theorem \ref{main-theorem} is to prove the case when $\Lambda = \Phi(\mathcal R)$.
    This is the content of Theorem \ref{main-theorem-hyper}.
    Once Theorem \ref{main-theorem-hyper} is proved, Theorem \ref{main-theorem} follows by approximation techniques introduced in \cite{FM18}.
    }
\end{intrormk}

In \S\ref{sec:examples}, we give examples that show that no condition of the theorem can be weakened.
As a consequence of the main theorem, we prove that strictly log canonical $3$-folds
with hyper-standard coefficients have bounded index.

\begin{introcor}\label{introcor}
Let $\Lambda \subset \qq$ be a set satisfying the descending chain condition with rational accumulation points.
There exists a natural number $n$ only depending on $\Lambda$ that satisfies the following.
Let $X\rightarrow T$ be a contraction between normal quasi-projective varieties such that
\begin{itemize}
    \item $(X,B)$ is a log canonical $3$-fold;
    \item $(X,B)$ has a log canonical center in the fiber over $t\in T$;
    \item $K_X+B\sim_{\qq,T} 0$; and 
    \item the coefficients of $B$ belong to $\Lambda$.
\end{itemize}
Then $n(K_X+B)\sim_{T} 0$.
\end{introcor}

As another consequence of the above corollary, we prove the following result about the local index of strict log canonical $4$-fold singularities.

\begin{introcor}\label{introcor1}
	Let $\Lambda = \Phi(\lbrace 0,1 \rbrace) $ be the set of standard multiplicities. 
	There exists a natural number $n$ only depending on $\Lambda$ that satisfies the following.
	Let $(x\in X,B)$ be a germ of a projective pair such that
	\begin{itemize}
		\item $(x\in X,B)$ is a log canonical $4$-fold;
		\item the point $x$ is a log canonical center of $(x\in X,B)$; and  
		\item the coefficients of $B$ belong to $\Lambda$.
	\end{itemize}
	Then after possibly shrinking $X$ around $x$, we have $n(K_X+B)\sim 0$.
\end{introcor}

\subsection*{Acknowledgements} 
SF and JM were partially supported by NSF research grants no: DMS-1801851, DMS-1265285 and by a
grant from the
Simons Foundation; Award Number: 256202.
YX would like to thank Prof. Birkar for his insightful discussions about the topic. 
After we completed this work, Jingjun Han informed us that he and Chen Jiang obtained some partial results for complemented 3-folds.

\section{Sketch of the proof}\label{sketch}

In this section, we will give a brief sketch of the proof of Theorem~\ref{main-theorem}.
The proof consists of three steps.
In the two first steps, we will prove Theorem~\ref{main-theorem} in the case that
$\Lambda=\Phi(\mathcal{R})$. 
The last step is an approximation argument
that leads to the general case.
\begin{enumerate}
    \item In this step, we reduce to the case in which $(X,B)$ is dlt and $-(K_X+B)$ is semiample over the base. To achieve this, we use dlt modifications and the minimal model program. 
    \item In this step, we lift a complement for $(X,B)$.
    In some cases we lift a complement from a log canonical center of $(X,B)$.
    In the other cases, we consider the morphism defined by $-(K_X+B)$ and lift a complement from the target of this morphism.
    \item Finally, we approximate the set $\Lambda$ with a set of hyper-standard coefficients to deduce the general statement.
\end{enumerate}
We give more details of these steps below. 

\subsection{Dlt modification and MMP} Let $B+B'$ be a $\qq$-complement of the log canonical $3$-fold $(X,B)$ around the point $t\in T$,
i.e., up to shrinking $T$ around $t\in T$, we have that $0 \leq B' \sim_{T,\qq} -(K_X+B)$ and $(X,B+B')$ is log canonical. 
We denote by $\pi \colon  Y\rightarrow X$ a $\qq$-factorial dlt modification of $(X,B+B')$ over $T$.
During the sketch, we will adopt the following notation
\[
\pi^*(K_X+B+B')=K_Y+B_Y+B'_Y+E, 
\]
where $E$ is the reduced divisor that contains all the log canonical places of $(X,B+B')$.
Note that $E$ may be reducible.
We also have
\[
B_Y \coloneqq \pi^{-1}_{*}B -\pi^{-1}_{*}B \wedge E,
\]
and
\[
B'_Y \coloneqq \pi^{-1}_{*}B' -\pi^{-1}_{*}B'\wedge E.
\]
Observe that $B_Y$ (resp. $B'_Y$) is the strict transform of $B$ (resp. $B'$)
minus all the divisors contained in the support of $E$.
In order to produce an $n$-complement for $(X,B)$ over $t\in T$, 
it suffices to produce an $n$-complement for $(Y,B_Y+E)$ (which is $\qq$-complemented over $t\in T$) 
over $t\in T$, and then push it forward to $X$.
Notice that all the log canonical places of $(Y,B_Y+B'_Y+E)$ are contained in the support of $E$.
Therefore, for $\epsilon>0$ small enough, the pair 
\[
(Y,B_Y+(1+\epsilon)B'_Y+E)
\]
is a $\qq$-factorial dlt pair that is pseudo-effective over $T$.
We run a minimal model program for $K_Y+B_Y+(1+\epsilon)B'_Y+E$ 
over $T$, which terminates with a good minimal model
$(Z,B_Z+(1+\epsilon)B'_Z+E_Z)$ over $T$. 
Here, $B_Z$ (resp. $B'_Z$ and $E_Z)$ denotes the strict transform of $B_Y$ (resp. $B'_Y$ and $E$).
Observe that this minimal model program is also a minimal model program for 
\[
-\epsilon (K_Y+B_Y+E) \sim_{\qq,T} \epsilon B'_Y \sim_{\qq,T} K_Y+B_Y+(1+\epsilon)B'_Y+E.
\]
In particular, any $n$-complement over $t\in T$ for $(Z,B_Z+E_Z)$ pulls back to an
$n$-complement over $t\in T$ for $(Y,B_Y+E_Y)$.
Therefore, it suffices to produce an $n$-complement for the log canonical pair
$(Z,B_Z+E_Z)$, which is $\qq$-complemented over $t\in T$.
Since $Z$ is a relative good minimal model, we have that $-(K_Z+B_Z+E_Z)$ is semi-ample over $T$.
Hence, it induces a morphism $\phi \colon Z\rightarrow Z_0$ over $T$.
We obtain a diagram as follows
\[
 \xymatrix@C=50pt{
 Y\ar[d]_-{\pi} \ar[rdd]\ar@{-->}[rr] & & Z\ar[ldd]\ar[d]^-{\phi} \\
 X\ar[rd] & & Z_0\ar[ld]  \\
 & T. & 
 }
\]

\subsection{Lifting complements} 
Now, we distinguish two main cases: $\phi$ is a fibration and hence $\dim Z_0 < \dim Z$, and $\phi$ is a birational contraction.

First, assume that $\phi$ is a fibration.
In this case, we apply an effective canonical bundle formula to write
\[
I(K_Z+B_Z+E_Z)\sim_T I\phi^*(K_{Z_0}+B_{Z_0}),
\]
where $I$ only depends on $\mathcal{R}$ and $(Z,B_Z)$, 
and the coefficients of $B_{Z_0}$ belong to some hyper-standard set only depending on $\mathcal{R}$.
Observe that it suffices to find an $n$-complement $\Gamma$ for $(Z_0,B_{Z_0})$ over $t\in T$.
Indeed, we will have
\[
\Gamma \sim_T -n(K_{Z_0}+B_{Z_0}), 
\]
with $(Z_0,B_{Z_0}+\Gamma/n)$ log canonical, and hence 
\[
I\Gamma_Z \sim_T -nI(K_Z+B_Z+E_Z)
\]
holds, where $\Gamma_Z$ denotes the pull-back of $\Gamma$ to $Z$.
Moreover, by construction, we will have that $(Z,B_Z+E_Z+\Gamma_Z/n)$ is log canonical. 
Hence, we reduced the problem to produce a complement for an anti-ample log canonical divisor, 
i.e., a pair that is log canonical and log Fano over the point $t\in T$.
In \S\ref{log-canonical-surface-complements}, we prove that the statement of Theorem~\ref{main-theorem} holds for surfaces.
Notice that most of the work in this direction already appears in \cites{Sho97,Pro99}.
In \S\ref{lifting-complements-surfaces}, we prove Theorem \ref{main-theorem} in the case that $Z_0$ is a surface,
by lifting the complement constructed in \S\ref{log-canonical-surface-complements} as explained above.
In \S\ref{lifting-complements-curves}, we prove the statement of Theorem \ref{main-theorem} in the case that $Z_0$ has dimension at most one, 
by either reducing the statement to the case of \S\ref{lifting-complements-surfaces},
applying the boundedness of the index for log Calabi--Yau projective $3$-folds,
or lifting the complement from a curve.

Now, assume that $\phi$ is a birational contraction.
In this case, $(Z_0,B_{Z_0}+E_{Z_0})$ is log Fano over $T$ and $(Z,B_Z+E_Z)$ is weak log Fano over $T$.
In this setup, we follow the ideas of Birkar, who settled the klt case in \cite{Bir16a}:
that is, we try to construct a complement on the pair induced by adjunction $(E_Z,B_{E_Z})$ and then lift it to a complement of $(Z,B_Z+E_Z)$.
In doing so, we have to face three main technical difficulties, which can be avoided in the klt case.
First, $E_Z$ may have several components.
To circumvent this, we develop a refined version of Koll\'ar's gluing theory, to guarantee that we can work on each component of $E_Z$ separately and then glue the complement.
In order to construct a complement that glues to the whole $E_Z$, we then consider the curves $E \subs Z,i. \cap E \subs Z,j.$ obtained as intersection of different connected components of $E \subs Z.$.
We first construct a complement on these curves and then lift it to the various components of $E_Z$.
In doing so, we face the second main difficulty:
$-(K_Z+B_Z+E_Z)$ may not be big on certain components of $E$.
Indeed, when $-(K_Z+B_Z+E_Z)$ is big along some component $E \subs Z,i.$, we can lift any complement constructed on $E \subs Z,i. \cap E \subs Z,j.$, while it is no longer the case if  $-(K_Z+B_Z+E_Z)$ is not big along $E \subs Z,i.$.
Thus, in \S\ref{sec:complements-sldt} we perform a detailed analysis of these cases that allows us to choose suitable complements on $E \subs Z,i. \cap E \subs Z,j.$ that can be lifted to $E \subs Z,i.$ even in the absence of positivity conditions.
This is achieved by proving an effective version of Koll\'ar's gluing theory 
for semi-dlt surfaces in \S\ref{effective-Kollar-gluing}.
Once a complement on $E_Z$ is constructed, we have to lift it to $Z$.
This represents the last technical challenge, as $(Z,B_Z+E_Z)$ is log canonical but not klt, thus vanishing theorems may not apply.
To circumvent this issue, in \S\ref{log-canonical-Fano}, we show that we can replace $Z$ with another suitable dlt model of $(Z_0,B \subs Z_0.+E_{Z_0})$ so that $(Z,B_Z+E_Z)$ can be perturbed to a klt pair.

\subsection{DCC coefficients}
In the two previous steps, 
we proved Theorem~\ref{main-theorem}
for $\Lambda=\Phi(\mathcal{R})$ (this is the statement of Theorem~\ref{main-theorem-hyper}).
Once this is established, we argue that a DCC set $\Lambda \subset [0,1] \cap \qq$ can be effectively approximated (from the perspective of the theory of complements) by a finite set $\lbrace \frac{1}{m}, \ldots , \frac{m-1}{m}, 1 \rbrace$ for some positive integer $m \gg 0$.
This is the content of Lemma \ref{coefficients-perturbation}, which is a generalization of \cite{FM18}*{Lemma 3.2}.
This finishes the proof of Theorem~\ref{main-theorem}.

\section{Preliminaries}

Throughout this paper, we work over an algebraically closed field $k$ of characteristic zero.
All varieties considered in this paper are normal unless otherwise stated.
In this section, we will collect some definitions and preliminary results which will be used in this article.

\subsection{Contractions} In this paper a {\em contraction} is a projective morphism of quasi-projective varieties $f  \colon  X \rar Z$ with $f_* \O X. = \O Z.$. Notice that, if $X$ is normal, then so is $Z$.

\subsection{Hyper-standard sets}

Let $\mathcal R$ be a subset of $[0,1]$.
Then, we define the \emph{set of hyper-standard multiplicities} associated to $\mathcal R$ as
\[
\Phi (\mathcal R) \coloneqq \bigg \lbrace \left. 1- \frac{r}{m}  \right| r \in \mathcal R, m \in \nn \bigg \rbrace.
\]
When $\mathcal R = \lbrace 0,1 \rbrace$, we call it the set of \emph{standard multiplicities}.
Usually, with no mention, we assume $0,1 \in \mathcal R$, such that $\Phi (\lbrace 0,1 \rbrace ) \subset \Phi (\mathcal R)$.
Furthermore, if $1-r \in \mathcal R$ for every $r \in R$, we have that $\mathcal R \subset \Phi (\mathcal R)$.

Now, assume that $\mathcal{R} \subset [0,1]$ is a finite set of rational numbers.
Then, $\Phi(\mathcal R)$ is a set of rational numbers satisfying the \emph{descending chain condition} (\emph{DCC} in short) whose only accumulation point is 1.
We define $I (\mathcal R)$ to be the smallest positive integer such that $I(\mathcal R) \cdot \mathcal R \subset \nn$.
The following is a useful property of $I (\mathcal R)$.

\begin{proposition} \label{prop arithmetic}
Let $\mathcal R \subset [0,1]$ be a finite set of rational numbers, and let $n$ be a positive integer divisible by $I (\mathcal R)$.
Fix $t \in \Phi(\mathcal{R})$.
Then, we have $\lfloor (n+1) t \rfloor - nt \geq 0$.
\end{proposition}

\begin{proof}
By assumption, we have $t = 1- \frac{r}{m}$, where $nr,m \in \nn$.
Then, we may write
\begin{align}
\nonumber    \Big \lfloor (n+1) \left( 1 - \frac{r}{m} \right) \Big \rfloor - n \left( 1 - \frac{r}{m} \right) = & \Big \lfloor n+1 - \frac{r(n+1)}{m} \Big \rfloor - \left( n - \frac{rn}{m} \right) \\
\nonumber     =& \Big \lfloor  1 - \frac{r(n+1)}{m} \Big \rfloor + \frac{rn}{m} \\
\nonumber =& \Big \lfloor - \frac{r(n+1)}{m} \Big \rfloor + 1 + \frac{rn}{m}\\
\nonumber =& 1 + \frac{rn}{m} - \Big \lceil  \frac{r(n+1)}{m} \Big \rceil \\
\nonumber =& 1 + \frac{rn}{m} - \Big \lceil  \frac{rn}{m} \Big \rceil \geq 0,
\end{align}
where the last equality follows from the fact that $rn \in \nn$.
\end{proof}

\subsection{Divisors} Let $X$ be a normal quasi-projective variety. We say that $D$ is a divisor on $X$ if it is a $\qq$-Weil divisor, i.e., $D$ is a finite sum of prime divisors on $X$ with coefficients in $\qq$.
The \emph{support} of a divisor $D=\sum_{i=1}^n d_iP_i$ is the union of the prime divisors appearing in the formal sum, $\mathrm{Supp}(D)= \sum_{i=1}^n P_i$.
Given a prime divisor $P$ in the support of $D$, we will denote by $\mult_P (D)$ the coefficient of $P$ in $D$.
Given a divisor $D = \sum \subs i=1. ^n d_i P_i$, we define its {\it round down} $\lfloor D \rfloor  \coloneqq \sum \subs i=1.^n \lfloor d_i \rfloor P_i$.
The {\it round up} $\lceil D \rceil$ of $D$ is defined analogously.
The {\it fractional part} $\{D\}$ of $D$ is defined as $\{D\} \coloneqq D - \lfloor D \rfloor$.
Let $D_1 = \sum \subs i=1. ^n d_i P_i$ and $D_2= \sum \subs i=1. ^n e_i P_i$ be two divisors.
We define $D_1 \wedge D_2 \coloneqq \sum \subs i=1. ^n \min \lbrace d_i,e_i \rbrace P_i$.
Similarly, we set $D_1 \vee D_2 \coloneqq \sum \subs i=1. ^n \max \lbrace d_i,e_i \rbrace P_i$.
For a divisor $D= \sum d_i P_i$ and a real number $a$, we set $D \sups \leq a. \coloneqq \sum \subs d_i \leq a. d_i P_i$.
The divisors $D \sups < a.$, $D \sups = a.$, $D \sups \geq a.$, and $D \sups > a.$ are defined analogously.

Let $f  \colon  X \rar Z$ be a projective morphism of quasi-projective varieties.
Given a divisor $D = \sum d_i P_i$ on $X$, we define
\[
D^v \coloneqq  \sum_{f(D_i) \subsetneq Z} d_i P_i, \; 
D^h \coloneqq  \sum_{f(D_i) = Z} d_i P_i.
\]
We call $D^v$ and $D^h$ the \emph{vertical part} and \emph{horizontal part} of $D$, respectively.
Let $D_1$ and $D_2$ be divisors on $X$.
We write $D_1 \sim_Z D_2$ (respectively $D_1 \sim \subs \qq,Z . D_2$) if there is a Cartier (respectively $\qq$-Cartier) divisor $L$ on $Z$ such that $D_1 - D_2 \sim f^*L$ (respectively $D_1 - D_2 \sim_\qq f^*L$).
Equivalently, we may also write $D_1 \sim D_2$ over $Z$.
The case of $\qq$-linear equivalence is denoted similarly.
Let $z$ be a point in $Z$.
We write $D_1 \sim D_2$ over $z$ if $D_1 \sim_{Z} D_2$ holds after possibly shrinking $Z$ around $z$.
We also make use of the analogous notion for $\qq$-linear equivalence.

Let $X \rar Z$ be a projective morphism, and let $D$ be a $\qq$-Cartier divisor on $Z$.
When the context is clear, and there is no ambiguity arising from a possible adjunction to a log canonical center, we may write $D|_X$ for the pull-back of $D$ to $X$.
This notation is useful when we are dealing with several morphisms that may not be labeled by a letter.

\subsection{Pairs}

A \emph{sub-pair} $(X,B)$ is the datum of a normal quasi-projective variety and a divisor $B$ such that $\K X. + B$ is $\qq$-Cartier.
If $B \sups \leq 1. = B$, we say that $B$ is a \emph{sub-boundary}, and if in addition $B \geq 0$, we call it \emph{boundary}.
A sub-pair $(X,B)$ is called a \emph{pair} if $B \geq 0$.
A sub-pair $(X,B)$ is \emph{simple normal crossing} (or \emph{log smooth}) if $X$ is smooth, every irreducible component of $\Supp(B)$ is smooth, and \'etale locally $\Supp(B) \subset X$ is isomorphic to the intersection of $r \leq n$ coordinate hyperplanes in $\mathbb{A}^n$.
A \emph{log resolution} of a sub-pair $(X,B)$ is a birational contraction $\pi \colon X' \rar X$ such that $\mathrm{Ex}(\pi)$ is a divisor and $(X',\pi \sups -1._* \Supp(B) + \mathrm{Ex}(\pi))$ is log smooth.
Here $\mathrm{Ex}(\pi) \subset X'$ is the \emph{exceptional set} of $\pi$, i.e., the reduced subscheme of $X'$ consisting of the points where $\pi$ is not an isomoprhism.

Let $(X,B)$ be a sub-pair, and let $\pi \colon X' \rar X$ be a birational contraction from a normal variety $X'$.
Then, we can define a sub-pair $(X',B')$ on $X'$ via the identity
\[
\K X'. + B' = \pi^*(\K X. + B),
\]
where we assume that $\pi_* \K X'. = \K X.$.
We call $(X',B')$ the \emph{log pull-back} or \emph{trace} of $(X,B)$ on $X'$.
The \emph{log discrepancy} of a prime divisor $E$ on $X'$ with respect to $(X,B)$ is defined as $a_E(X,B) \coloneqq 1 - \mult_E (B')$.
We say that a sub-pair $(X,B)$ is \emph{sub-log canonical} (resp. \emph{sub-klt}) if $a_E(X,B) \geq 0$ (resp. $a_E(X,B) > 0$) for every $\pi$ and every $E$ as above.
When $(X,B)$ is a pair, we say that $(X,B)$ is \emph{log canonical} or \emph{klt}, respectively.
Notice that, if $(X,B)$ is log canonical (resp. klt), we have $0 \leq B = B \sups \leq 1.$ (resp. $0 \leq B = B \sups < 1.$).

Let $(X,B)$ be a sub-pair.
A \emph{non-klt place} is a prime divisor $E$ on a birational model of $X$ such that $a_E(X,B) \leq 0$.
A \emph{non-klt center} is the image of a non-klt place.
If $a_E(X,B)=0$, we say that $E$ is a \emph{log canonical place}, and the corresponding center is said to be a \emph{log canonical center}.
The \emph{non-klt locus} $\mathrm{Nklt}(X,B)$ is defined as the union of all the non-klt centers of $(X,B)$.
Similarly, the \emph{non-log canonical locus} $\mathrm{Nlc}(X,B)$ is defined as the union of all the non-klt centers of $(X,B)$ that are not log canonical centers.
Given a sub-log canonical sub-pair $(X,B)$ and an effective $\qq$-Cartier divisor $D$, we define the \emph{log canonical threshold} of $D$ with respect to $(X,B)$ as
\[
\lct (X,B;D) \coloneqq \sup \lbrace t \geq 0 | (X,B+tD) \text{ is sub-log canonical} \rbrace.
\]

In this paper, we make use of the standard results of the MMP.
We refer to \cites{BCHM,KM98} for the main results and the standard terminology.
Furthermore, as this work is mainly concerned with varieties of dimension 3, we refer to \cites{KMM94,Sho96} for the additional results that hold in dimension 3.

The minimal model program allows building suitable birational modifications of a pair.
A pair $(X,B)$ is called \emph{dlt} if there exists a closed subset $Z \subset X$ such that $(X \setminus Z, B | \subs X \setminus Z.)$ is log smooth and every divisor $E$ on a birational model whose center on $X$ is contained in $Z$ satisfies $a_E(X,B) > 0$.
Recall that a normal variety $X$ is called \emph{$\qq$-factorial} if every divisor is $\qq$-Cartier.
Given a pair $(X,B)$, a \emph{dlt model} is a birational model $\pi \colon X' \rar X$ with reduced exceptional divisor $E'$ such that $X'$ is $\qq$-factorial, $(X',\pi \sups -1._* (B \wedge \Supp(B)) + E')$ is dlt, and every $\pi$-exceptional divisor $D$ satisfies $a_D(X,B) \leq 0$.
The existence of these models is due to Hacon \cite{KK10}*{Theorem 3.1}.
Relying on the MMP for 3-folds \cite{Sho96}, we can prove the following refinement of \cite{KK10}*{Theorem 3.1}.

\begin{lemma} \label{lemma dlt model}
Let $(X,B)$ be a log canonical pair with $\dim X \leq 3$.
Fix a log resolution $f \colon X'' \rar X$ of $(X,B)$, and let $E''_1, \ldots , E''_k$ be the prime divisors on $X''$ with log discrepancy equal to 0.
Then, we can run a suitable MMP on $X''$ to obtain a dlt model $X'$ of $(X,B)$ such that $E'_1, \ldots ,E'_k$ are the only divisors extracted.
Here, $E'_i$ is the strict transform on $X'$ of $E''_i$.
In particular, the rational map $X'' \drar X'$ is an isomorphism along the generic point of each log canonical center of $(X'',B'')$.
Furthermore, we may choose $X'$, so that every log canonical center of $(X',E'_1+\dots+E'_k)$ contained in the exceptional locus of $X' \rar X$ is contained in
the support of some $E'_i$.
\end{lemma}

\begin{proof}
Define $\Delta \coloneqq \lbrace B \rbrace$, and set $\Delta'' \coloneqq f \sups -1._* \Delta$.
We define $B'$ via the identity $\K X''. + B'' = f^*(\K X. + B)$.
Then, we can decompose $B'$ as $B''= \Delta'' + E'' + F'' - G''$, where
\begin{itemize}
    \item $E'' = \sum \subs i=1. ^k E''_i$ denotes the (not necessarily $f$-exceptional) divisors with log discrepancy equal to 0;
    \item $F'' \geq 0$ is supported on the $f$-exceptional divisors with log discrepancy in $(0,1]$; and
    \item $G'' \geq 0$ is supported on the $f$-exceptional divisors with log discrepancy strictly greater than 1.
\end{itemize}
Furthermore, let $P''$ be the sum of the $f$-exceptional divisors with log discrepancy in $(0,1]$, and $Q''$ be the sum of the $f$-exceptional divisors with log discrepancy strictly greater than 1.
Fix $0 < \epsilon \ll 1$.
Then, we have
\[
\K X''. + \Delta'' + E'' + F'' + \epsilon P'' \sim_\qq \epsilon P'' + G''/X,
\]
and the pair $(X'',\Delta'' + E'' + (1+\epsilon)F'')$ is dlt.
Notice that $\Supp(\epsilon P'' + G'')=P''+Q''$.
Since $\dim X \leq 3$, we can run a $(\K X''. + \Delta'' + E'' + F'' + \epsilon P'')$-MMP relative to $X$, which terminates with a $\qq$-factorial minimal model $X'$ \cite{Sho96}.
Notice that this is an $(\epsilon P'' + G'')$-MMP.
In particular, any divisor contracted by this MMP is in $\Supp(P'' + Q'')$.
Denote by $\Gamma'$ the strict transform of any given divisor $\Gamma''$ on $X''$.
Then, by the negativity lemma \cite{KM98}*{Lemma 3.39}, we have $\epsilon P' + G'=0$.
In particular, all the divisors extracted by $X' \rar X$ are $E'_1, \ldots, E'_k$.
\newline
Now, we check that $X'' \drar X'$ is an isomorphism along the generic point of every stratum of $(B'') \sups =1.$.
Let $X_i \drar X \subs i+1.$ denote a step of the above MMP.
Let $\Gamma_i$ and $\Gamma \subs i+1.$ denote the strict transforms of any given divisor $\Gamma''$ on $X_i$ and $X \subs i+1.$, respectively.
Let $Z''$ be a log canonical center of $(X'',B'')$, and assume that $X'' \drar X_i$ is an isomorphism along the generic point of $Z''$.
Assume that $X_i \drar X_{i+1}$ is not an isomorphism along the generic point of $Z_i$, the image of $Z''$ on $X_i$.
Let $P$ be a log canonical place of $(X_i,B_i)$ corresponding to $Z_i$.
Then, as $\Delta_i + E_i + F_i + \epsilon P_i  \geq B_i$, $P$ is a log canonical place of $(X_i,\Delta_i + E_i + F_i + \epsilon P_i)$.
By \cite{KM98}*{Lemma 3.38}, $P$ is not a log canonical place of $(X_{i+1},\Delta_{i+1} + E_{i+1} + F_{i+1} + \epsilon P_{i+1})$.
This is a contradiction, as $X_i \drar X \subs i+1.$ is $(\K X_i. + B_i)$-trivial, and $P$ is a log canonical place for $(X \subs i+1.,B \subs i+1.)$.
\newline
Now, we are left with showing the last part of the statement.
Fix a log resolution $f \colon X'' \rar X$.
Notice that every log canonical center of $(X'',B'')$ is a log canonical center of the log smooth pair $(X'',\Delta'' + E'' + F'' + \epsilon P'')$, where $0 < \epsilon \ll 1$.
Let $Z_1,\ldots,Z_l$ be the log canonical centers of $(X'',B'')$ that are contained in $\mathrm{Ex}(f)$ and that are not contained in any of the $E''_i$.
Then, these are log canonical centers of the log smooth pair $(X'',\Delta'' + E'' + F'' + \epsilon P'' + \delta Q'')$, where $0 < \delta \ll 1$.
In particular, each $Z_i$ is smooth.
Therefore, we can blow up $Z_1$, then the strict transform of $Z_2$ on this first blow-up, etc., to obtain a new log smooth model.
By abuse of notation, we replace $X''$ with this model.
In particular, we can assume that all the log canonical centers of $(X'',B'')$ that are contained in $\mathrm{Ex}(f)$ are contained in one of the $E''_i$.
Now, let $X'$ be the $\qq$-factorial dlt model of $(X,B)$ constructed from $X''$ as above.
Then, as $X'' \drar X'$ is an isomorphism along the generic point of every log canonical center of $(X'',B'')$, the above property is preserved on $X'$.
More precisely, every log canonical center of $(X',B')$ that is contained in $\mathrm{Ex}(X' \rar X)$ is contained in $E'_i$ for some $i$.
\end{proof}

\begin{remark} \label{remark dlt model}
    Let $(X,B)$, $X''$, and $X'$ be as in Lemma \ref{lemma dlt model}.
    Further, assume that $X'' \rar X$ is obtained by blowing up centers of codimension at least 2.
    Then, there exists an effective divisor $\Gamma$ that is exceptional for $X'' \rar X$ and such that $-\Gamma$ is ample over $X$, see \cite{kol21}*{\S8}.
    Then, to resolve the rational map $X'' \drar X'$, we do not need to extract any log canonical place of $(X,B)$.
    Furthermore, we can resolve $X'' \drar X''$ by blowing up loci of codimension at least 2.
    Let $X'''$ be the obtained model.
    Then, we have that
    \begin{itemize}
        \item $X'''$ is obtained from $X$ blowing up loci of codimension at least 2;
        \item there exists an effective divisor $\Sigma$ that is exceptional for $X''' \rar X$ such that $-\Sigma$ is ample over $X$; and
        \item the log canonical places of $(X,B)$ extracted on $X'''$ are the same as the ones extracted on $X''$ and $X'$. Indeed, $X'' \drar X'$ is an isomorphism along the generic points of the log canonical centers; thus, the generic points of the blow-ups are contained in an open set that contains no log canonical center.
    \end{itemize}
    Furthermore, up to a further blow-up, we may assume that $X'''$ is a log resolution of $(X,B)$.
    Hence, replacing $X''$ with $X'''$, we may assume the three conditions above are satisfied.
\end{remark}

\begin{remark}
In Lemma \ref{lemma dlt model}, the assumption $\dim X \leq 3$ is required to have termination for an arbitrary dlt MMP.
The original proof of the existence of dlt models, originally due to Hacon, connects the log resolution $X''$ with the final model $X'$ with an MMP.
On the other hand, the presence of a correction term ($C$ in the notation of \cite{KK10}*{Theorem 3.1}) does not guarantee that no log canonical place that is present on $X''$ is contracted by the MMP.
By \cite{BZ16}*{Lemma 4.6}, one can replace $X'$ with a higher model $X'''$, such that the exceptional divisors of $X''' \rar X$ are exactly the log canonical places extracted on $X''$.
On the other hand, $X''$ and $X'''$ are not necessarily connected by an MMP.
\end{remark}

\subsection{Non-normal pairs}
Let $X$ be a non-normal variety.
We say that $X$ is \emph{demi-normal} if it satisfies Serre's condition $S_2$ and its codimension one points are either regular points or nodes.
Let $\pi \colon X^\nu \rar X$ denote the normalization of $X$.
The \emph{conductor ideal} $\mathcal{H}om_X(\pi_* \O X^\nu.,\O X.) \subset \O X.$ is the largest ideal sheaf on $X$ that is also an ideal sheaf on $X^\nu$.
Therefore, it defines two subschemes $D \subset X$ and $D^\nu \subset X^\nu$, which are called \emph{conductor subschemes}.
Notice that a rational involution $D^\nu \drar D^\nu$ is naturally induced; the rational involution extends to a regular involution on the normalization of $D^\nu$.
Let $X$ be a demi-normal scheme, and let $B$ be an effective divisor whose support does not contain any irreducible component of the conductor $D$.
Let $B^\nu$ denote the divisorial part of $\pi \sups -1. (B)$.
Then, we say that $(X,B)$ is a \emph{semi-log canonical pair} if $\K X. + B$ is $\qq$-Cartier and $(X^\nu,B^\nu + D^\nu)$ is log canonical.
We refer to \cite{Kol13}*{\S5.1} for the notion of a divisor on a demi-normal scheme and the notion of pull-back for $\K X. + B$.
In particular, we have that the notion of log discrepancy is well defined for semi-log canonical pairs.
Let $(X,B)$ a semi-log canonical pair.
We say that $(X,B)$ is \emph{semi-dlt} if every irreducible component of $X$ is normal, and $(X^\nu,B^\nu + D^\nu)$ is dlt.
Notice that this definition agrees with the one in \cite{Fuj00}, and it is stricter than the one in \cite{Kol13}.
By \cite{Fuj00}*{Remark 1.2}, if $(Y,\Delta)$ is a dlt pair, then $(\lfloor \Delta \rfloor, \mathrm{Diff}(\Delta - \lfloor \Delta \rfloor)$ is a semi-dlt pair.

\subsection{Semi-normal curves}
Let $X$ be a scheme, and let $\pi \colon X' \rar X$ be a finite morphism.
The morphism $\pi$ is a \emph{partial semi-normalization} if $X'$ is reduced, each point $x \in X$ has exactly one preimage $x' \coloneqq \pi \sups -1. (x)$, and $\pi ^* \colon k(x) \rar k(x')$ is an isomorphism.
A scheme $X$ is called \emph{semi-normal} if every partial semi-normalization $\pi \colon X' \rar X$ is an isomorphism.
In particular, a semi-normal scheme is reduced.
Over an algebraically closed field, a curve singularity $(0 \in C)$ is semi-normal if and only if it is analytically isomorphic to the union of $n$ coordinate axes in $\mathbb{A}^n$ \cite{Kol13}*{Example 10.12}.

\subsection{B-birational maps and B-representations}
Let $(X,B)$ and $(X',B')$ be not necessarily normal pairs, and let $f \colon X \drar X'$ be a birational map.
We say that $f \colon (X,B) \drar (X',B')$ is \emph{B-birational} if there exists a common resolution $Y$ admitting morphisms $p \colon Y \rar X$ and $q \colon Y \rar X'$ such that $p^*(\K X. + B)=q^*(\K X'. + B')$.
Notice that, if $f$ is a B-birational map, it induces a bijection between the irreducible components of $X$ and $X'$.
We refer to \cite{Fuj00} for more details about B-birational maps.

Let $(X,B)$ be a not necessarily normal pair. We define the group of self B-birational maps as
\[
\mathrm{Bir}(X,B) \coloneqq \lbrace f| f \colon (X,B) \drar (X,B) \; \text{is B-birational} \rbrace.
\]
Let $m$ be any integer such that $m(\K X. + B)$ is a Cartier divisor.
By definition of B-birational map, every $f \in \mathrm{Bir}(X,B)$ induces an automorphism of $H^0(X,\O X. (m(\K X. + B))$.
In particular, we have an induced representation
\[
\rho_m \colon \mathrm{Bir}(X,B) \rar \mathrm{Aut}(H^0(X,\O X. (m(\K X. + B))).
\]
Under suitable assumptions, the image of this representation is finite.
In particular, we have the following statement.
\begin{theorem}[{\cite{FG14b}*{Theorem 3.15}}] \label{Fujino Gongyo finite}
Let $(X,B)$ be a projective (not necessarily connected) log canonical pair such that $\K X. + B$ is semi-ample.
Let $m$ be a positive integers such that $m(\K X. + B)$ is Cartier.
Then, $\rho_m(\mathrm{Bir}(X,B))$ is a finite group.
\end{theorem}

\subsection{Koll\'ar's gluing theory}
Koll\'ar developed a theory of quotients by finite equivalence relations \cite{Kol13}*{Chapter 9}.
In particular, it is a powerful tool to study a semi-log canonical pair $(X,B)$ via its normalization $(X^\nu,B^\nu+D^\nu)$.
Here, we just recall some key facts that will be used in \S\ref{effective-Kollar-gluing}.
We refer to \cite{Kol13} for the terminology involved and to \cites{HX13,HX16} for examples of the interplay between Koll\'ar's gluing theory and semi-log canonical pairs.

\begin{theorem}[{\cite{HX16}*{Theorem 1.4}}] \label{thm hx16}
Let $(X,B)$ be a semi-log canonical pair, $f \colon X \rar S$ a projective morphism, $\nu \colon X^\nu \rar X$ the normalization, and write $\nu^*(\K X. + B)=\K X^\nu. + B^\nu + D^\nu$, where $D^\nu$ is the double locus.
If $\K X^\nu. + B^\nu + D^\nu$ is semi-ample over $S$, then $\K X. + B$ is semi-ample over $S$.
\end{theorem}

The idea of the proof of Theorem \ref{thm hx16} is the following.
Assume for simplicity that $S= \Spec(\cc)$, $B=0$, and $(X^\nu,B^\nu + D^\nu)=(X_1,D_1) \sqcup (X_2,D_2)$, where each $D_i$ is normal and irreducible.
A section of $\O X^\nu. (m(\K X^\nu. + D^\nu))$ descends to a section of $\O X.(m\K X.)$ if its restriction to $D^\nu$ is invariant under the involution $\tau$ that exchanges $D_1$ and $D_2$.
Koll\'ar's gluing theory guarantees that, in order to show that $|\O X.(m\K X.)|$ separates $x_1$ and $x_2$, it suffices to find two sections $s_1,s_2 \in \Gamma(\O X^\nu. (m (\K X^\nu. + D^\nu)))$ that separate the preimages of $x_1$ and $x_2$ and such that each $s_i| \subs D^\nu.$ is $\tau$-invariant.
The theory of B-representations and, in particular,  Theorem \ref{Fujino Gongyo finite}, guarantee that we can find all the needed $\tau$-invariant sections in $|\O X^\nu.(m(\K X^\nu. + D^\nu))|$ for some $m$.
As we are interested in finding $n$-complements for a bounded $n$, we need an effective version of this approach.
In particular, we need to show that we can find enough invariant sections in $|\O X^\nu.(k(\K X^\nu. + D^\nu))|$ for a bounded $k$.
We develop this approach in \S\ref{effective-Kollar-gluing}.

\subsection{B-divisors}
Let $X$ be a normal variety, and consider the set of all proper birational morphisms $\pi \colon  X_\pi \rightarrow X$, where $X_\pi$ is normal.
This is a partially ordered set, where $\pi' \geq \pi$ if $\pi'$ factors through $\pi$.
We define the space of {\it Weil b-divisors} as the inverse limit
\begin{equation}
\mathbf{Div}(X)\coloneqq \varprojlim_\pi \mathrm{Div}(X_\pi),
\end{equation}
where $\mathrm{Div}(X_\pi)$ denotes the space of Weil divisors on $X_\pi$.
Then, we define the space of {\it $\qq$-Weil b-divisors} $\mathbf{Div}_\qq(X) \coloneqq \mathbf{Div}(X) \otimes \qq$.
In the following, by b-divisor we will mean a $\qq$-Weil b-divisor.
Equivalently, a b-divisor $\mathbf{D}$ can be described as a (possibly infinite) sum of geometric valuations $V_i$ of $k(X)$ with coefficients in $\qq$,
\[
 \mathbf{D}= \sum_{i \in I} b_i V_i, \; b_i \in \mathbb{\qq},
\]
such that for every normal variety $X'$ birational to $X$, only a finite number of the $V_i$ can be realized by divisors on $X'$. 
The {\it trace} $\mathbf{D}_{X'}$ of $\mathbf{D}$ on $X'$ is defined as 
\[
\mathbf{D}_{X'} \coloneqq \sum_{%\substack
\{i \in I \; | \; c_{X'}(V_i)= D_i, \; 
\codim_{X'} (D_i)=1\}} b_i D_i
\]
where $c_{X'}(V_i)$ denotes the center of the valuation on $X'$.

Given a b-divisor $\mathbf{D}$ over $X$, we say that $\mathbf{D}$ is a {\it b-$\qq$-Cartier} b-divisor if there exists a birational model $X'$ of $X$ such that $\mathbf{D}_{X'}$ is $\qq$-Cartier on $X'$, and for any model $r \colon X''  \rar X'$, we have $\mathbf{D}_{X''} = r^\ast \mathbf{D}_{X'}$.
When this is the case, we will say that $\mathbf{D}$ descends to $X'$ and write $\mathbf{D}= \overline{\mathbf{D}_{X'}}$.
We say that $\mathbf{D}$ is {\it b-effective}, if $\mathbf{D}_{X'}$ is effective for any model $X'$.
We say that $\mathbf{D}$ is {\it b-nef} (respectively, {\it b-free}), if it is b-$\qq$-Cartier and, moreover, there exists a model $X'$ of $X$ such that $\mathbf{D}= \overline{\mathbf{D}_{X'}}$ and $\mathbf{D}_{X'}$ is nef (respectively, intregral and free) on $X'$.
The notions of b-nef and b-free b-divisor can be extended analogously to the relative case.

\begin{example}
\label{discr.div.ex}
Let $(X, B)$ be a sub-pair.
The \emph{discrepancy b-divisor} $\mathbf{A}(X,B)$ is defined as follows: on a birational model $\pi \colon X' \rar X$, its trace $\mathbf{A}(X,B)_{X'}$ is given by the identity  $\K X'. = \pi^*(\K X. + B) + \mathbf{A}(X,B)_{X'}$.
Then, the b-divisor $\mathbf{A}^*(X,B)$ is defined taking its trace $\mathbf{A}^*(X,B)_{X'}$ on $X'$ to be $\mathbf{A}^*(X,B)_{X'} \coloneqq \sum \subs a_i > -1. a_i D_i$, where $\mathbf{A}(X,B)_{X'} = \sum_i a_i D_i$.
\end{example}

\subsection{Generalized pairs}
A {\em generalized sub-pair} $(X,B, \mathbf{M})/Z$ over $Z$  is the datum of:
\begin{itemize}
\item a normal variety $X  \rar Z$ projective over $Z$;
\item a divisor $B$ on $X$;
\item a b-$\qq$-Cartier b-divisor $\mathbf{M}$ over $X$ which descends to a nef$/Z$ Cartier divisor $\mathbf{M}_{X'}$ on some birational model $X' \rightarrow X$.
\end{itemize}
Moreover, we require that $K_X +B+ \mathbf{M}_X$ is $\qq$-Cartier.
If $B$ is effective, we say that $(X,B, \mathbf M)/Z$ is a \emph{generalized pair}.
The divisor $B$ is called the \emph{boundary part} of $(X,B, \mathbf M)/Z$, and $\mathbf M$ is called the \emph{moduli part}.
In the definition, we can replace $X'$ with a higher birational model $X''$ and $\mathbf{M}_{X'}$ with $\mathbf{M}_{X''}$ without changing the generalized pair.
Whenever $\mathbf{M}$ descends on $X''$, then the datum of the rational map $X'' \drar X$, $B$, and $\mathbf{M}_{X''}$ encodes all the information of the generalized pair.

Let $(X,B, \mathbf{M})/Z$ be a generalized sub-pair and $\rho\colon Y \rar X$ a projective birational morphism. 
Then, we may write
\[
\K Y.+B_Y + \mathbf{M}_{Y}=\pi^*(K_X+B+\mathbf M _X).
\]
Given a prime divisor $E$ on $Y$, we define the {\em generalized log discrepancy} of $E$ with respect to $(X,B, \mathbf M)/Z$  to be $a_E(X,B,\mathbf M)\coloneqq 1-\mult_{E}(B')$.
If $a_E(X,B,\mathbf M) \geq 0$ for all divisors $E$ over $X$, we say that $(X,B, \mathbf M)/Z$ is \emph{generalized sub-log canonical}.
Similarly, if $a_E(X,B+M) > 0$ for all divisors $E$ over $X$ and $\lfloor B \rfloor \leq 0$, we say that $(X,B, \mathbf M)/Z$ is \emph{generalized sub-klt}.
When $B \geq 0$, we say that $(X,B, \mathbf M)/Z$ is \emph{generalized log canonical} or \emph{generalized klt}, respectively.

\subsection{Canonical bundle formula}

We recall the statement of the {\em canonical bundle formula}. 
We refer to \cites{FM00,FG14,Fil18} for the notation involved and a more detailed discussion about the topic.
Let $(X, B)$ be a sub-pair.
A contraction $f \colon X \rar T$ is an \emph{lc-trivial fibration} if
\begin{itemize}
    \item[(i)] $(X,B)$ is a sub-pair that is sub-log canonical over the generic point of $T$;
    \item[(ii)] $\mathrm{rank} f_* \O X. (\lceil \mathbf{A}^*(X,B)\rceil)=1$, where $\mathbf{A}^*(X,B)$ is the b-divisor defined in Example \ref{discr.div.ex}; and
    \item[(iii)] there exists a $\qq$-Cartier divisor $L_T$ on $T$ such that $\K X. + B \sim_\qq f^* L_T$.
\end{itemize}
Condition (ii) above is automatically satisfied if $B$ is effective over the generic point of $T$.
Given a sub-pair $(X,B)$ and an lc-trivial fibration $f \colon X \rar T$, there exist b-divisors $\mathbf{B}$ and $\mathbf{M}$ over $T$ such that the following linear equivalence relation, known as the {\it canonical bundle formula}, holds
\begin{equation}
    \label{cbf.eqn}
    K_X+B \sim_{\mathbb{Q}} f^\ast(K_T+\mathbf{B}_{T}+\mathbf{M}_{T}).
\end{equation}
The b-divisor $\mathbf{B}$ is often called the \emph{boundary part} in the canonical bundle formula; it is a canonically defined b-divisor.
Furthermore, if $B$ is effective, then so is $\mathbf{B}_T$.
The b-divisor $\mathbf{M}$, in turn, is often called the \emph{moduli part} in the canonical bundle formula, and it is in general defined only up to $\mathbb{Q}$-linear equivalence. 
The linear equivalence \eqref{cbf.eqn} holds at the level of b-divisor: namely, 
\[
\overline{(K_X+B)} \sim_\qq f^*(\mathbf{K}+\mathbf{B}+\mathbf{M}),
\] 
where $\mathbf{K}$ denotes the canonical b-divisor of $T$.
Let $I$ be a positive integer such that $I(\K X. + B) \sim 0$ along the generic fiber of $f$.
Then, by \cite{PS09}*{Construction 7.5}, we may choose $\mathbf M$ in its $\qq$-linear equivalence class such that
\begin{equation} \label{equation coeffs M}
I\overline{(K_X+B)} \sim If^*(\mathbf{K}+\mathbf{B}+\mathbf{M}).
\end{equation} 

The moduli b-divisor $\mathbf{M}$ is expected to detect the variation of the fibers of the morphism $f$.
In this direction, we have the following statement.

\begin{theorem}
\cite{FG14}*{Theorem 3.6} 
\label{classic cbf}
Let $f \colon (X,B) \rar T$ be an lc-trivial fibration and let $\pi \colon T \rar S$ be a projective morphism.
Let $\mathbf{B}$ and $\mathbf{M}$ be the b-divisors that give the boundary and the moduli part, respectively.
Then, $\mathbf K + \mathbf B$ and $\mathbf M$ are $\qq$-b-Cartier b-divisors.
Furthermore, $\mathbf{M}$ is b-nef over $S$.
\end{theorem}

\begin{remark}
In the setup of Theorem \ref{classic cbf}, let $T'$ be a model where the nef part $\mathbf M$ descends in the sense of b-divisors.
Then, $\mathbf M \subs T'.$ is nef over $S$. 
In particular, $(T, \bB T.,\mathbf{M})/S$ is a generalized sub-pair.
\end{remark}

\subsection{Fano-type pairs}
Let $(X,B)$ be a log canonical pair, and let $f \colon X \rar T$ be a contraction.
We say that $(X,B)$ is \emph{log Fano over $T$} if $-(\K X. + B)$ is ample over $T$.
If $-(\K X. + B)$ is nef and big over $T$, we say that $(X,B)$ is \emph{weak log Fano over $T$}.
If $B=0$, we say that $X$ is Fano (resp. weak Fano) over $T$.
If $T = \Spec (k)$, we omit it from the notation.
Finally, we say that $X$ is of \emph{Fano-type over $T$} if there exists a boundary $B$ such that $(X,B)$ is klt weak log Fano over $T$.

\subsection{Complements} \label{subsection_complements}
Let $(X,B)$ be a log canonical pair, $X\rightarrow T$ a contraction, and $n$ a positive integer.
We say that the divisor $B^+$ is a \emph{$\qq$-complement} over $t \in T$ if the following conditions hold over some neighborhood of $t \in T$:
\begin{itemize}
\item[(i)] $(X,B^{+})$ is a log canonical pair;
\item[(ii)] $K_X+B^{+} \sim_\qq 0$ over $t \in T$; and 
\item[(iii)] $B^+ \geq B$.
\end{itemize}
Furthermore, we say that $B^+$ is an \emph{$n$-complement} for $(X,B)$ over $t \in T$ if the following stronger version of condition (ii) holds:
\begin{itemize}
    \item [(ii)$'$] $n(K_X+B^{+}) \sim 0$ over $t \in T$.
\end{itemize}
In particular, if $B^+$ is an $n$-complement, $nB$ is an integral Weil divisor.

\begin{remark} \label{rmk_monotonic}
Notice that more general complements, where the above condition (iii) is weakened, are used in the literature.
See for example \cite{Bir16a}*{2.18}.
Since in this work condition (iii) is always satisfied, we decided to use this stronger definition of complement, in order to avoid redundant terminology and notation.
In case we need to reference works where the broader notion of complement is used, we will address why we can guarantee that the reference can be applied to obtain a complement in the sense of the above definition.
\end{remark}

Following the work of Birkar \cite{Bir16a}, we can extend the notion of complement to generalized pairs.
Let $(X,B,\mathbf M)/Z$ be a generalized log canonical pair, $X\rightarrow T$ a contraction over $Z$, and $n$ a positive integer.
We say that the divisor $B^+$ is a \emph{$\qq$-complement} over $t \in T$ if the following conditions hold over some neighborhood of $t \in T$:
\begin{itemize}
\item[(i)] $(X,B^{+},\mathbf M)$ is a generalized log canonical pair;
\item[(ii)] $K_X+B^{+} + \mathbf M _X \sim_\qq 0$ over $t \in T$; and 
\item[(iii)] $B^+ \geq B$.
\end{itemize}
As above, we say that $B^+$ is an \emph{$n$-complement} for $(X,B,\mathbf M)/Z$ over $t \in T$ if the following stronger version of condition (ii) holds:
\begin{itemize}
    \item [(ii)$'$] $n(K_X+B^{+} + \mathbf M _X) \sim 0$ over $t \in T$.
\end{itemize}
In particular, if $B^+$ is an $n$-complement, and $n \mathbf M$ is an integral b-divisor, then $nB$ is an integral Weil divisor.

The following is the relative version of \cite{Bir16a}*{6.1.(2)} and its proof is identical to the one of \cite{Bir16a}*{6.1.(2)}.

\begin{lemma}[cf. \cite{Bir16a}*{6.1.(2)}] \label{lemma birkar}
Let $(X,B, \mathbf M)/Z$ be a generalized pair, and let $\phi \colon X \drar X'$ be a birational map over $Z$ to a normal variety $X'$, projective over $Z$.
Let $X''$ be a resolution of $\phi$ where $\mathbf M$ descends.
Let $f \colon X'' \rar X$ and $g \colon X'' \rar X'$ be the corresponding morphisms.
Assume that there exist effective divisors $B'$ and $P''$ on $X'$ and $X''$, respectively, such that
\[
f^*(\K X. + B + \mathbf M _X) + P'' = g^*(\K X'. + B' + \mathbf M \subs X'.).
\]
Then, if $(X',B',\mathbf M)/Z$ has an $n$-complement, then so does $(X,B,\mathbf M)/Z$.
\end{lemma}

As special cases of Lemma \ref{lemma birkar}, we have the following immediate consequences.

\begin{lemma} \label{reduction complements}
Let $(X,B,\mathbf M)/Z$ be a generalized pair and let $(Y,\Delta)$ be a log canonical pair.
Then, the following facts hold:
\begin{itemize}
    \item[(1)] let $D \geq 0$ be a $\qq$-Cartier divisor on $X$.
    Then, an $n$-complement for $(X,B+D,\mathbf M)/Z$ is also an $n$-complement for $(X,B,\mathbf M)/Z$;
    \item[(2)] let $(Y',\Delta')$ be a dlt model for $(Y,\Delta)$.
    Then, if $(Y',\Delta')$ has an $n$-complement, then so does $(Y,\Delta)$;
    \item[(3)] let $X \drar X'$ be a partial MMP for $-(\K X. + B + \mathbf M _X)$ over $Z$.
    Let $B'$ denote the push-forward of $B$ on $X'$.
    Then, if $(X',B',\mathbf M)/Z$ has an $n$-complement, then so does $(X,B,\mathbf M)/Z$; and
    \item[(4)] let $X \drar X'$ be a sequence of rational birational contractions over $Z$ that are trivial for $\K X. + B + \mathbf M _X$.
    Let $B'$ denote the push-forward of $B$ on $X'$.
    Then, if $(X',B',\mathbf M)/Z$ has an $n$-complement, then so does $(X,B,\mathbf M)/Z$.
\end{itemize}
\end{lemma}

\begin{proof}
To avoid confusion with the notation, we proceed by cases.
\begin{itemize}
\item[(1)] This case follows immediately, as if $B^+ \geq B+D$, then $B^+ \geq B$, as we assumed $D \geq 0$.

\item[(2)] This case follows from Lemma \ref{lemma birkar}.
With the notation of Lemma \ref{lemma birkar}, we have $(X,B,\bM.)=(Y,\Delta)$, $X'=X''=Y'$, $B'=\Delta'$ and $P''=0$.

\item[(3)] This is the relative version of \cite{Bir16a}*{6.1.(3)} and its proof is identical to the absolute case.

\item[(4)] Let $X''$ be a common resolution of $X$ and $X'$, and let $f \colon X'' \rar X$ and $g \colon X'' \rar X'$ denote the corresponding morphisms.
Then, as $X \drar X'$ is $(K_X + B + \bM X.)$-trivial, we have
\[
f^*(\K X. + B + \mathbf M _X) = g^*(\K X'. + B' + \mathbf M \subs X'.).
\]
In particular, this case follows from Lemma \ref{lemma birkar} with $P''=0$.
\end{itemize}
\end{proof}

\begin{lemma} \label{rmk generic point}
Assume the setup of Theorem \ref{main-theorem} (resp. Theorem \ref{main-theorem-hyper}), where $t=\eta_T$ is the generic point of the base $T$.
Also, assume that Theorem \ref{main-theorem} (resp. Theorem \ref{main-theorem-hyper}) holds over closed points.
Then, Theorem \ref{main-theorem} (resp. Theorem \ref{main-theorem-hyper}) holds over $\eta_T$.
\end{lemma}

\begin{proof}
Assume the setup of Theorem \ref{main-theorem} (resp. Theorem \ref{main-theorem-hyper}) over the generic point $\eta_T$ of $T$.
Then, by assumption, $(X,B)$ is $\qq$-complemented over $\eta_T$.
That is, there exists $B' \geq 0$ such that, shrinking around $\eta_T \in T$, $K_X + B + B' \sim_\qq 0$.
Hence, there exist a non-empty open subset $U \subset T$ such that $K_X+B+B'$ is $\qq$-linearly trivial along $X \times_T U$.
In particular, the assumptions of Theorem \ref{main-theorem} (resp. Theorem \ref{main-theorem-hyper}) are satisfied for every $u \in U$.
Now, pick a closed point $u \in U$.
By assumptions, the conclusions of Theorem \ref{main-theorem} (resp. Theorem \ref{main-theorem-hyper}) hold for $u$.
But then, as $K_X+B$ has an $n$-complement over a neighborhood of $u \in T$, the same open set is a neighborhood of $\eta_T$, and the claim follows.
\end{proof}

\begin{lemma} \label{rmk lcc}
Assume that Theorem \ref{main-theorem} (resp. Theorem \ref{main-theorem-hyper}) holds under the additional assumption that there is a log canonical place of $(X,B)$ whose center on $T$ is $t$.
Then Theorem \ref{main-theorem} (resp. Theorem \ref{main-theorem-hyper}) holds.
\end{lemma}

\begin{proof}
Let $B' \geq 0$ be a $\qq$-complement for $(X,B)$ over $t \in T$.
By part (2) of Lemma \ref{reduction complements}, up to taking a dlt model for $(X,B+B')$, we may assume that $X$ is $\qq$-factorial and $(X,B+B')$ is dlt.
By Lemma \ref{rmk generic point}, we may assume that $t$ is not the generic point of $T$.
Since $T$ is quasi-projective, we may find a sufficiently ample Cartier divisor $D$ such that $t \in \Supp(D)$;
furthermore, up to shrinking $T$ around $t$, we may assume that $D$ is prime.
Set $c \coloneqq \mathrm{lct} \subs t.(X,B+B';f^*D)$, where by $\mathrm{lct} \subs t.$ we mean the log canonical threshold over the point $t$.
In particular, up to shrinking $T$ around $t$, $(X,B+B'+cf^*D)$ is strictly log canonical and has a log canonical place $E$ whose center in $T$ contains $t$.
If $c=1$, we have $E=f^*D$, and we replace $B$ by $B+f^*D$.
If $c <1$, let $\pi \colon X' \rar X$ be a dlt model for $(X,B+B'+cf^*D)$ where $E$ appears as a divisor.
Let $F'$ denote the reduced exceptional divisor of $\pi$.
Then, $(X',\pi \sups -1._*B + F')$ satisfies the assumptions of Theorem \ref{main-theorem} (resp. Theorem \ref{main-theorem-hyper}), with $\qq$-complement $\pi^*B'+\pi \sups -1._* (f^*D)$.
Furthermore, by Lemma \ref{lemma birkar}, the push-forward of an $n$-complement for $(X',\pi \sups -1._*B + F')$ provides an $n$-complement for $(X,B)$.
Therefore, we may replace $(X,B)$ with $(X',\pi \sups -1._*B + F')$ and assume that there is a prime component $P$ of $\lfloor B \rfloor$ such that $t \in f(P) \subsetneq T$.
If $t$ is the generic point of $f(P)$, we stop.
Otherwise, we repeat the above strategy picking a prime Cartier divisor $D$ on $T$ such that $t \in D$ and $f(P) \not \subset D$.
Since $\dim T \leq 3$, after finitely many iterations of this algorithm, we obtain the claimed reduction.
\end{proof}

\section{Examples}\label{sec:examples}

In this section, we give examples showing that the hypotheses of Theorem~\ref{main-theorem} are optimal.
In order, we provide the following examples:
\begin{itemize}
    \item Example \ref{not_q_comp} shows that Theorem \ref{main-theorem} fails if we weaken the assumption that $(X,B)$ is $\qq$-complemented.
    Indeed, we provide an example of a dlt 3-fold $(X,B)$ such that $-(K_X+B) \sim D \geq 0$ for some $D \geq 0$, but there is no $0\leq \Gamma \sim_\qq -(K_X+B)$ such that $(X,B+\Gamma)$ is log canonical;
    \item Example \ref{example_dcc} shows that Theorem \ref{main-theorem} fails if the set of coefficients $\Lambda$ does not satisfy the DCC property; and
    \item Example \ref{example_accumulation} shows that Theorem \ref{main-theorem} fails if the set of coefficients $\Lambda$ is not rational with rational accumulation points.
\end{itemize}
These examples are known to the experts, 
but we include them for the sake of completeness.

First, we show an example for which $-(K_X+B)$ is effective, but no divisor $0 \leq \Gamma \sim_\qq -(K_X+B)$ satisfies the condition that $(X,B+\Gamma)$ is log canonical.

\begin{example} \label{not_q_comp}
Let $X$ be the blow-up of $\pp^2$ at a point $p$.
Let $E$ be the exceptional divisor.
Let $H_1,H_2$ and $H_3$ be three lines on $\pp^2$ passing through $p$ with different tangent directions.
Let $L_1,L_2$ and $L_3$ be the strict transforms
of $H_1,H_2$ and $H_3$, respectively.
Observe that $L_1,L_2$ and $L_3$ are disjoint.
Hence, the pair $(X,L_1+L_2+L_3)$ is log canonical.
However, $-(K_X+L_1+L_2+L_3)\sim 2E$, and this divisor generates the ring of sections $\bigoplus \subs k \geq 0. H^0(X,\O X. (2kE))$.
Thus, the only effective divisor $\Gamma$ for which
$K_X+L_1+L_2+L_3+\Gamma \sim_\qq 0$ is $2E$,
and the pair $(X,L_1+L_2+L_3+2E)$ is not log canonical.
\end{example}

It is well-known that in order to find $n$-complements,
we need to impose some condition on the set of coefficients $\Lambda$.
First, $\Lambda$ has to satisfy the descending chain condition.

\begin{example} \label{example_dcc}
Consider the sequence of boundaries $B_i=\sum_{j=1}^{2k-1} \frac{1}{k} p_j$ on $\pp^1$,
where for each $i$, the points $p_1,\dots, p_i$ are 
distinct.
Then, the linear system $|-m(K_{\pp^1}+B_i)|$
is empty for $m\in \{1,\dots, k-1\}$.
Hence, there is no bounded $n$-complement for the sequence
of pairs $(\pp^1,B_i)$.
\end{example}

We recall an example in~\cite{FM18}, 
which shows that the statement of Theorem~\ref{main-theorem} does not hold
if the accumulation points of $\Lambda$ are not rational.
This already happens in the Fano case.
For more considerations on the conditions that $\Lambda$ has to satisfy see~\cite{FM18}*{\S2.6}.

\begin{example} \label{example_accumulation}
Let $( a_i ) \subs i \geq 1.$ be an increasing sequence of rational numbers converging to $ {1}/{\sqrt{2}}$.
Similarly, let $( b_i ) \subs i \geq 1.$ be an increasing sequence of rational numbers converging to $1-1/\sqrt{2}$.
Define $\Lambda \coloneqq \lbrace a_i \rbrace \subs i \geq 1. \cup \lbrace b_i \rbrace \subs i \geq 1.$. Note that $\Lambda$ satisfies the DCC, however its accumulation points are not rational.

Let $p_0,p_1,p_2,p_3$ be four distinct points of $\pp^1$.
Let $B_i = a_i p_0+a_i p_1+b_i p_2+b_i p_3$ be a sequence
of boundaries on $\pp^1$.
Observe that $(\pp^1,B_i)$ is klt
and $-(K_{\pp^1}+B_i)$ is ample.
We show that for no fixed $n$, all the pairs
$(\pp^1,B_i)$ admit an $n$-complement.

Fix a positive integer $n$.
Then, for $i$ large enough, we have $\frac{\lceil n a_i \rceil}{n} > \frac{\sqrt{2}}{2}$. Similarly, we have $\frac{\lceil n b_i \rceil}{n} > 1- \frac{\sqrt{2}}{2}$.
Therefore, there exists no $0 \leq \Gamma \sim_\qq (K_{\pp^1}+B_i)$ such that $n(K_{\pp^1}+B_i+\Gamma)$ is integral and $\deg (B_i+\Gamma) =2$. In particular, there exists no $n$-complement for $(\pr 1., B_i)$.
\end{example}

\section{Complements for surfaces}\label{log-canonical-surface-complements}

In this section, we prove the statement of Theorem~\ref{main-theorem-hyper} for surfaces.
In particular, we prove the following statement.

\begin{theorem}\label{thm:log-canonical-surface-complements}
Let $\mathcal{R} \subset [0,1]$ be a finite set of rational numbers.
There exists a natural number $n$ only depending on $\mathcal{R}$ that satisfies the following.
Let $X\rightarrow T$ be a contraction between normal quasi-projective varieties 
such that the log canonical surface $(X,B)$ is $\qq$-complemented over $t\in T$ and the coefficients of $B$ belong to $\Phi(\mathcal{R})$. 
Then up to shrinking $T$ around $t$ we can find
\[
\Gamma \sim_{T} -n(K_X+B)
\]
such that $(X,B+\Gamma/n)$ is a log canonical pair.
\end{theorem}

\begin{remark}
    By Lemma \ref{coefficients-perturbation}, one can recover a version of Theorem \ref{thm:log-canonical-surface-complements} where the set of coefficients is a DCC set $\Lambda \subset \qq$ with rational  accumulation points.
    On the other hand, Theorem \ref{thm:log-canonical-surface-complements} is sufficient for the structure of the proof of Theorem \ref{main-theorem}.
\end{remark}

The theory of complements for surfaces has been developed by Shokurov and Prokhorov \cites{Sho97,Pro99}.
On the other hand, some results are phrased only for the set of coefficients $\Phi(\lbrace 0,1 \rbrace)$ or they are phrased for complements that do not satisfy condition (iii) in \S\ref{subsection_complements}, cf. Remark \ref{rmk_monotonic}.
For these reasons, we perform some reductions to the known cases \cites{Bir16a,Pro99,Sho97} to extend the classic results to the setup of this paper.

\subsection{The log Calabi--Yau case}
As a first reduction, we focus on the case when the pair $(X,B)$ is of log Calabi--Yau type over the base of the contraction.
This is an important case of Theorem \ref{thm:log-canonical-surface-complements} since we can reduce more general situations to this one.

\begin{proposition} \label{prop logcy surface complement}
Theorem \ref{thm:log-canonical-surface-complements} holds true if $\K X. + B \sim_\qq 0 /T$ over some neighborhood of $t \in T$.
\end{proposition}
\begin{proof}
By part (2) of Lemma \ref{reduction complements}, up to taking a dlt model of $(X,B)$, we may assume that $(X,B)$ is $\qq$-factorial dlt.
Then, we subdivide the proof by cases, depending on $\dim T$.

{\bf Case 1:} Assume that $X \rar T$ is a birational morphism.
By \cite{Bir16a}*{Theorem 1.8}, we may assume that $\lfloor B \rfloor \neq 0$.
Indeed, if $\lfloor B \rfloor = 0$, as $(X,B)$ is dlt, it follows that $(X,B)$ is klt;
then, as $X \rar T$ is birational, $(X,B)$ is klt, and $K_X + B \sim_\qq 0/T$, it follows that $X$ is of Fano type over $T$, and thus \cite{Bir16a}*{Theorem 1.8} can be applied.
Thus, in treating Case 1, we may assume $\lfloor B \rfloor \neq 0$.
If $t \in T$ is not a closed point, we may assume that $X=T$ is a smooth surface. Therefore, we reduce to \cite{Bir16a}*{Theorem 1.8}.
Hence, we may assume that $t \in T$ is a closed point.
By \cite{Pro99}*{Proposition 4.4.3}, it suffices to show that the semi-log canonical pair $(\lfloor B \rfloor,\Diff \subs \lfloor B \rfloor.(B))$ has an $n$-semi-complement for a bounded $n$ (see \cite{Pro99}*{Definition 4.1.4}).
By \cite{Bir16a}*{Lemma 3.3}, there exists a finite set $\mathcal{S} \subset [0,1]$ only depending on $\mathcal{R}$ such that the coefficients of $\Diff \subs \lfloor B \rfloor.(B)$ belong to $\Phi(\mathcal{S})$.
Thus, by \cite{thesis}*{Theorem 1.8.7}, there exists $n$ depending only on the hypotheses of Theorem \ref{thm:log-canonical-surface-complements} such that $n(\K \lfloor B \rfloor. + \Diff \subs \lfloor B \rfloor.(B)) \sim 0$.
Thus, the birational case is settled.
Notice that \cite{Pro99} uses a weaker notion of (semi-)complement, where $B^+ \geq B$ may not hold.
On the other hand, as the complement produced by \cite{thesis}*{Theorem 1.8.7} satisfies this type of inequality and \cite{Pro99}*{Proposition 4.4.3} is used to lift this complement, we have $B^+ \geq B$.

{\bf Case 2:} Assume that $T$ is a curve.
In particular, the general fiber is either $\pr 1.$ or an elliptic curve.
The two cases correspond to $B^h \neq 0$ or $B^h = 0$.
First, assume that $B^h \neq 0$.
Then, by part (4) of Lemma \ref{reduction complements}, we may run a $\K X.$-MMP over $T$, as it is trivial for $K_X+B$.
This terminates with a Mori fiber space $\hat X \rar T$.
Let $\hat B$ denote the push-forward of $B$ to $\hat X$.
Then, it suffices to show the statement for $(\hat X,\hat B)$, and the latter follows by \cite{Bir16a}*{Theorem 1.8}.
Therefore, we may assume that $B^h = 0$.
Then, by Lemma \ref{rmk generic point} and \ref{rmk lcc}, we may assume that $(X,B)$ has a log canonical center dominating the closed point $t \in T$.
As $(X,B)$ is dlt, some irreducible component of the fiber $X_t$ has coefficient 1 in $B$; call this component $F$.
Since $X \rar T$ is an elliptic fibration and $K_X + B$ is trivial over $T$, we may contract all the components of $X_t$ but $F$.
Indeed, if the fiber is not irreducible, all the irreducible components of the fiber as rational curves with negative self-intersection.
By part (4) of Lemma \ref{reduction complements}, this is a legitimate reduction.
Then, we conclude by \cite{Sho97}*{Theorem 3.1}.
Notice that, in \cite{Sho97}, a weaker notion of complement where $B^+ \geq B$ may not hold is considered; on the other hand, as $B$ is reduced over $t \in T$, the notion in use in this paper and the notion in \cites{Sho97} agree and we can freely use \cite{Sho97}*{Theorem 3.1}.

{\bf Case 3:} Assume that $T= \Spec(k)$.
Then, this is the content of \cite{thesis}*{Theorem 1.8.7}.
\end{proof}

\subsection{An effective canonical bundle formula for fibrations in curves}
To construct complements for a pair $(X,B)$ that is relatively log Calabi--Yau over a base $T$, it is sometimes useful to decompose the structure morphism $X \rar T$ as a composition $X \rar S \rar T$.
This strategy allows for an inductive approach to the problem.
For this strategy to be successful, we need to be able to construct on $S$ a new pair $(S,B_S)$ such that the coefficients of $B_S$ are under control.
In order to proceed, we need to prove Theorem \ref{thm:log-canonical-surface-complements} in full generality for morphisms of relative dimension 1.

\begin{proposition} \label{full case rel dim 1}
Theorem \ref{thm:log-canonical-surface-complements} holds true if $\dim T =1$.
\end{proposition}

\begin{proof}
By assumption, there exists a $\qq$-divisor $B' \geq 0$ such that $\K X. + B + B' \sim_\qq 0 /T$ over a neighborhood of $t \in T$, which may be assumed to be a closed point of the curve $T$ by Lemma \ref{rmk generic point}.
By part (2) of Lemma \ref{reduction complements}, we may assume that $X$ is $\qq$-factorial.
If $\lfloor B' \rfloor \neq 0$, by part (1) of Lemma \ref{reduction complements}, we may add $\lfloor B' \rfloor$ to $B$ and replace $B'$ with $B'-\lfloor B' \rfloor$.
Thus, we may assume that $\lfloor B' \rfloor = 0$.
Furthermore, by Proposition \ref{prop logcy surface complement}, we may assume that $B' \neq 0$.
Let $\pi \colon Y \rar X$ be a $\qq$-factorial dlt model of $(X,B+B')$.
By construction, we may write
\[
\pi^*(\K X. + B + B') = \K Y. + \pi \sups -1._* B + \pi \sups -1. _* B' + E,
\]
where $E$ is a reduced divisor.
By parts (1) and (2) of Lemma \ref{reduction complements}, it suffices to find a bounded $n$-complement for $(Y,\pi \sups -1._* B + E)$.
Thus, up to relabelling $Y$ by $X$, we may assume that $(X,B+B')$ is $\qq$-factorial dlt, $B' \neq 0$, and $\lfloor B' \rfloor = 0$.
\newline
For $0 < \epsilon \ll 1$, the pair $(X,B + (1+\epsilon)B')$ is dlt.
Furthermore, we have $\epsilon B' \sim_\qq \K X. + B + (1+\epsilon)B'/T$.
Therefore, we may run a $B'$-MMP over $T$, which terminates with a good minimal model $Z \rar T$.
Let $B_Z$ and $B'_Z$ denote the push-forwards of $B$ and $B'$, respectively.
Since $B' \sim_\qq -(\K X. + B)$, by part (3) of Lemma \ref{reduction complements}, it suffices to produce a bounded $n$-complement for $(Z,B_Z)$ over $t \in T$.
\newline
Now, we distinguish two cases.

{\bf Case 1:} The divisor $B'$ is vertical over $T$.
\newline
By construction, $B'$ is supported on the fiber over $t$.
If it is a multiple of the fiber, we reduce to Proposition \ref{prop logcy surface complement}.
Therefore, we may assume that $B'$ is of insufficient fiber type.
Then, by \cite{Lai11}*{Lemma 2.10}, the MMP $X \rar Z$ contracts $B'$.
In particular, the pair $(Z,B_Z)$ satisfies the assumptions of Proposition \ref{prop logcy surface complement}.

{\bf Case 2:} The divisor $B'$ has a component that dominates $T$.
\newline
In this case, the general fiber is $\pr 1.$.
On the model $Z$, we have that $B'_Z$ is nef and big over $T$.
Therefore, we have that $-(\K Z. + B_Z)$ is nef and big over $T$.
Let $W$ be the relative ample model of $(Z,B_Z+(1+\epsilon)B'_Z)$, and let $\phi \colon Z \rar W$ be the corresponding morphism.
\newline
First, assume that $(W,B_W+(1+\epsilon)B'_W)$ is dlt, where $B_W$ and $B'_W$ denote the push-forwards of $B_Z$ and $B'_Z$, respectively.
Then, if $0 < \delta \ll 1$, $(W,(1-\delta)B_W)$ is klt and $-(\K W. + (1 -\delta)B_W)$ is ample over $T$.
Therefore, by \cite{Bir16a}*{Theorem 1.8}, $(W,B_W)$ admits a bounded complement.
As $\phi$ is $(\K Z. + B_Z)$-trivial, by part (4) of Lemma \ref{reduction complements}, this complement induces the sought complement for $(Z,B_Z)$.
\newline
Now, assume that $(W,B_W+(1+\epsilon)B'_W)$ is not dlt.
This means that $\phi$ contracts some divisor with coefficient 1 in $B_Z$.
Since $\phi$ is a morphism over $T$, such divisor has to be a component of the fiber over $t$.
Define $S \coloneqq \lfloor B_Z \rfloor$.
Then, $(S,\mathrm{Diff}_S(B_Z))$ is semi-dlt.
By \cite{thesis}*{Proposition 2.7.2}, there exists a bounded $n$-complement for $(S,\mathrm{Diff}_S(B_Z))$.
Then, the claim follows from \cite{Pro99}*{Proposition 4.4.3}.
Notice that \cite{Pro99} uses a weaker notion of (semi-)complement, where $B^+ \geq B$ may not hold.
On the other hand, as the complement produced by \cite{thesis}*{Theorem 1.8.7} satisfies this type of inequality and \cite{Pro99}*{Proposition 4.4.3} is used to lift this complement, we have $B^+ \geq B$.
\end{proof}

Now, we are ready to state an effective version of the canonical bundle formula for fibrations in curves.

\begin{theorem} \label{effective cbf}
Let $\mathcal{R} \subset [0,1]$ be a finite set of rational number, and let $(X,B)$ be a quasi-projective log canonical pair such that the coefficients of $B$ belong to $\Phi(\mathcal{R})$.
Let $f \colon X \rar T$ be a contraction to a normal quasi-projective variety $T$ with $\dim T = \dim X -1$ such that $\K X. + B \sim_\qq 0 /T$, and let $(T,B_T, \mathbf M)$ denote the generalized pair induced by the canonical bundle formula.
Then, there exists a finite set of rational numbers $\mathcal{S} \subset [0,1]$ such that the coefficients of $B_T$ belong to $\Phi(\mathcal{S})$.
Furthermore, there exists $q \in \nn$ only depending on $\mathcal R$ such that, if the b-divisor is $\bM.$ chosen as in \eqref{equation coeffs M}, $q \bM.$ is b-free and \[
q(\K X . + B) \sim q f^*(\K T. + B_T + \bM T.).
\]
Lastly, we may assume that $\frac{1}{q} \in \Phi(\mathcal{S})$.
\end{theorem}

\begin{proof}
We proceed in several steps.\\

{\bf Step 1:} We treat the existence of $q$.
\newline
The existence of $q$ follows from \cite{PS09}*{Theorem 8.1}.
Notice that \cite{PS09}*{Assumption 7.11}, which is necessary for \cite{PS09}*{Theorem 8.1}, is satisfied.
First, we notice that the condition can be checked over an open subset of the base.
Indeed, the condition on the singularities is for a general fiber.
Thus, if $\Theta$ is as in \cite{PS09}*{Assumption 7.11}, we can always perturb it with an $f$-vertical divisor.
Then, if a $\qq$-divisor $\Theta$ satisfies $\K X. + \Theta \sim \subs \qq. 0/U$ over an open subset $U$ of $T$, then we can find an $f$-vertical divisor $E$ such that $\K X. + \Theta + E \sim \subs \qq. 0/T$.
This shows that we can focus on the horizontal part of the divisor $\Theta$ in \cite{PS09}*{Assumption 7.11}.
Abusing notation, we still denote it by $\Theta$.
Then, as the morphism has relative dimension 1, \cite{PS09}*{Assumption 7.11} is vacuously satisfied:
if the general fiber is an elliptic curve, we can choose $\Theta = 0$, while if the general fiber is $\pr 1.$, we can choose $\Theta$ to be a suitable multiple of a general multisection of sufficiently high degree.
In particular, we are left with controlling the coefficients of $B_T$.

{\bf Step 2:} We reduce to the case when $T$ is a curve.
\newline
Notice that the computations needed to produce $B_T$ involve codimension 1 points.
Therefore, as $T$ is normal, we may assume that it is smooth.
Then, by \cite{Flo14}*{proof of Lemma 3.1}, we may assume that $T$ is a curve.
In particular, $X$ is a surface.

{\bf Step 3:} We reduce to the case when $X$ is projective and $(X,B)$ is $\qq$-factorial dlt.
\newline
Let $\overline{T}$ denote the compactification of $T$.
By \cite{HX13}*{Corollary 1.2}, there exist a projective log canonical pair $(\overline{X},\overline{B})$ with a contraction $\overline{f} \colon \overline{X} \rar \overline{T}$ such that $(\overline{X},\overline{B})\times_{\overline{T}} T = (X,B)$ and the restriction of $\overline{f}$ to $\overline{X} \times_{\overline{T}} T$ coincides with $T$.
Up to taking a dlt model, we may assume that $(\overline{X},\overline{B})$ is $\qq$-factorial dlt.
Then, up to removing the components of $\overline{B}$ that map to $\overline{T} \setminus T$, we may assume that the coefficients of $\overline{B}$ belong to $\Phi(\mathcal{R})$.
Finally, 
since $\overline{X}$ is a surface,
$(\overline{X},\overline{B})$ has a good minimal model over $\overline{T}$.
Thus, up to relabelling, we have $\K \overline{X}. + \overline{B} \sim_\qq 0/\overline{T}$.
As the relative minimal model may not be $\qq$-factorial and dlt, we replace it by a dlt model.
Now, let $(\overline{T},B \subs \overline{T}. + M \subs \overline{T}.)$ be the generalized pair induced by $(\overline{X},\overline{B})$ on $\overline{T}$.
By construction, we have $B_T = B \subs \overline{T}.|_T$ and $M_T \sim_\qq M \subs \overline{T}.|T$.
Therefore, it suffices to prove the statement for $B \subs \overline T.$ and $M \subs \overline{T}.$.

{\bf Step 4:} We reduce to the case when $X \rar T$ is an elliptic fibration.
\newline
Assume that $X \rar T$ is not an elliptic fibration.
Then, the general fiber is $\pr 1.$.
As $X$ is a $\qq$-factorial klt variety, we may run a $\K X.$-MMP relative to $T$.
As $\K X.$ is not pseudo-effective over $T$, this MMP ends with a Mori fiber space $X' \rar T$.
Let $B'$ denote the push-forward of $B$ to $X'$.
Then, we may apply \cite{Bir16a}*{Proposition 6.3}.
Indeed, $X'\rightarrow T$ is a Fano-type morphism, and
$K_{X'}+B'$ is $\qq$-trivial over the base.
In particular, there exists a finite set $\mathcal{S} \subset [0,1]$ such that the coefficients of $B_T$ belong to $\Phi(\mathcal S)$.

{\bf Step 5:} We conclude the proof by treating the case of an elliptic fibration.
\newline
We argue as in \cite{Bir16a}*{Step 3 in proof of Proposition 6.3}.
Fix a closed point $t \in T$, and let $c \coloneqq \lct(X,B;f^*(t))$.
Set $\Gamma \coloneqq B + c f^*(t)$.
Let $(X',\Gamma')$ be a dlt model for $(X,\Gamma)$ and write $\pi \colon X' \rar X$.
Then, there exists a boundary $B' \leq \Gamma'$ such that the coefficients of $B'$ belong to $\Phi(\mathcal{R})$, $\lfloor B' \rfloor$ has a component mapping to $t$, and $\pi \sups -1._* B \leq B'$.
Then, by Proposition \ref{full case rel dim 1}, $(X',B')$ admits a bounded $n$-complement $B'^+$ over $t \in T$.
Write $B^+ \coloneqq \pi_* B'$.
Then, $B^+$ is an $n$-complement for $(X,B)$ such that $(X,B^+)$ has a non-klt center mapping to $t$.
Since we have $\K X. + B \sim_\qq 0/T$, it follows that $B^+ - B \sim_\qq 0/T$ over $t$.
In particular, $B^+-B$ is a multiple of $f^*(t)$.
Since $(X,B^+)$ is log canonical but not klt over $t$, it follows that $B^+-B = c f^*(t)$.
\newline
Recall that the coefficient of $t$ in $B_T$ is $1-c$.
Fix a component $C$ of $f^*(t)$, and let $b$ and $b^+$ be the coefficients of $C$ in $B$ and $B^+$, respectively.
If $m \in \nn$ is the coefficient of $C$ in $f^*(t)$, we have $b^+=b+cm$.
Hence, we have $c = \frac{b^+ - b}{m}$.
Now, $b=1-\frac{r}{l}$ for some $r \in \mathcal{R}$ and $l \in \nn$.
Thus, we have $t = \frac{s}{m}$, where $s = b^+ - 1 + \frac{r}{l}$.
If $b^+ =1$, then $c = \frac{r}{lm}$, and $1 - c \in \Phi(\mathcal{R})$.
If $b^+ < 1$, as $r \leq 1$, and $nb^+$ is integral, we get
\[
1- \frac{1}{l} \leq b \leq b^+ \leq 1 - \frac{1}{n}.
\]
In particular, we have $l \leq n$, and there are finitely many possibilities for $s$.
Thus, we may find a finite set of rational numbers $\mathcal{S} \subset [0,1]$ as claimed.

{\bf Step 6:} We conclude arguing that we may assume $q \in \Phi(\mathcal{S})$.
\newline
This is immediate by replacing $\mathcal{S}$ with $\mathcal{S} \cup \lbrace \frac{q-1}{q}\rbrace$.
\end{proof}

\begin{corollary} \label{remark turning M into boundary}
Assume the setup of Theorem~\ref{effective cbf}, and choose $\epsilon$ with $0 \leq \epsilon \leq 1-\frac{1}{q}$.
Then, we may find $\Delta_T$ on $T$ such that the following facts hold:
\begin{itemize}
    \item[(1)] the coefficients of $\Delta_T$ belong to $\Phi(\mathcal{S})$;
    \item[(2)] $\Delta_T \sim_\qq \bM T.$; and
    \item[(3)] if $(T,B_T,\bM.)$ is generalized $\epsilon$-log canonical, then $(T,B_T+\Delta_T)$ is $\epsilon$-log canonical.
\end{itemize}
\end{corollary}

\begin{proof}
Let $T'$ be a birational model of $T$
where $\bM.$ descends, i.e.,
there exists a projective birational morphism
$\pi_T \colon T'\rightarrow T$ such that $\bM T'.$ is nef and $\bM T. = {\pi_T}_* \bM T'.$.
By Theorem \ref{effective cbf}, $q \bM T'.$ is integral and $|q \bM T'.|$ is a free linear series.
Thus, we may choose an integral divisor $0 \leq q \Delta \subs T'. \sim q \bM T'.$ such that the sub-pair $(T',B \subs T'. + \Delta \subs T'.)$ is sub $\epsilon$-log canonical,
whenever the pair $(T',B \subs T'.)$ is $\epsilon$-log canonical.
Notice that the generalized discrepancies of $(T,B_T,\bM.)$ are the same as the discrepancies of $(T',B \subs T'.)$.
Denote by $\Delta_T$ the push-forward of $\Delta \subs T'.$ to $T$.
Then, the pair $(T,B_T+\Delta_T)$ is generalized
$\epsilon$-log canonical, whenever $(T,B_T,\bM.)$ is generalized $\epsilon$-log canonical.
Furthermore, the generalized discrepancies of $(T,B_T,\bM.)$ are greater than or equal to the discrepancies of $(T,B_T+\Delta_T)$.
Finally, notice that, by Theorem \ref{effective cbf}, the coefficients of $B_T + \Delta_T$ belong to $\Phi(\mathcal{S})$.
\end{proof}

\subsection{The general case}
Now, we prove Theorem \ref{thm:log-canonical-surface-complements} in full generality.

\begin{proof}[{Proof of Theorem \ref{thm:log-canonical-surface-complements}}]
By assumption, there exists a $\qq$-divisor $B' \geq 0$ such that $\K X. + B + B' \sim_\qq 0 /T$ over a neighborhood of $t \in T$.
As argued in the proof of Proposition \ref{full case rel dim 1}, we may assume that $(X,B+B')$ is $\qq$-factorial dlt, $B' \neq 0$, and $\lfloor B' \rfloor = 0$.
\newline
For $0 < \epsilon \ll 1$, the pair $(X,B + (1+\epsilon)B')$ is dlt.
Furthermore, we have $\epsilon B' \sim_\qq \K X. + B + (1+\epsilon)B'/T$.
Therefore, we may run a $B'$-MMP over $T$, which terminates with a good minimal model $Z \rar T$.
Let $B_Z$ and $B'_Z$ denote the push-forward of $B$ and $B'$, respectively.
Since $B' \sim_\qq -(\K X. + B)$, by part (3) of Lemma \ref{reduction complements}, it suffices to produce a bounded $n$-complement for $(Z,B_Z)$ over $t \in T$.
\newline
By Proposition \ref{full case rel dim 1}, we may assume that $T=\Spec(k)$ or that $X \rar T$ is birational.
We will treat these cases separately.

{\bf Case 1:} We assume that $X \rar T$ is birational.
\newline
As argued in Case 1 of the proof of Proposition \ref{prop logcy surface complement}, we may assume that $t$ is a closed point and that $\lfloor B \rfloor \neq 0$.
By the above reduction, it suffices to find a bounded $n$-complement for $(Z,B_Z)$ over $t$.
Since $Z$ is a minimal model for $B'$ over $T$, it follows that $-(\K Z. + B_Z)$ is nef over $T$.
Since $Z \rar T$ is birational, $-(\K Z. + B_Z)$ is automatically big over $T$.
Since $X \rar Z$ is an MMP for $-(\K X. + B)$ and $(X,B)$ is not klt over $t$, it follows that $(Z,B_Z)$ is not klt over $t$.
Then, up to replacing $(Z,B_Z)$ with a dlt model, we can argue as in Case 1 in the proof of Proposition \ref{prop logcy surface complement}.
The only difference is the following: the input for \cite{Pro99}*{Proposition 4.4.3} is no longer \cite{thesis}*{Theorem 1.8.7}, but it is \cite{thesis}*{Proposition 2.7.2}.

{\bf Case 2:} We assume that $T = \Spec(k)$.
\newline
By construction, $-(\K Z. + B_Z)$ is semi-ample.
Let $\phi \colon Z \rar Z_0$ be the morphism induced by a sufficiently divisible multiple of $-(\K Z. + B_Z)$.
Then, $\K Z. + B_Z \sim_\qq \phi^*L$ for some $\qq$-Cartier divisor $L$ on $Z_0$, and $L$ is anti-ample.
If $Z_0$ is birational to $Z$, by part (4) of Lemma \ref{reduction complements}, we can replace $Z$ by $Z_0$, and assume that $-(\K Z. + B_Z)$ is ample.
Then, this case is covered by \cite{thesis}*{Theorem 1.8.5}.
If $Z_0 = \Spec(k)$, then it follows that $B'$ is contracted on $Z$. We deduce the statement by  \cite{thesis}*{Theorem 1.8.7}.
\newline
Therefore, we may assume that $Z_0$ is a smooth projective curve.
Let $(Z_0,B \subs Z_0. , \bM.)$ be the generalized pair induced by $(Z,B_Z)$ via the canonical bundle formula.
Notice that, by Theorem  \ref{effective cbf}, we have control of the coefficients of $B_{Z_0}$ and the Cartier index of $\bM.$.
Since $L \sim_\qq \K Z_0. + B \subs Z_0. + \bM  Z_0.$ is anti-ample, it follows that $Z_0 = \pr 1.$.
By \cite{Bir16a}*{Theorem 1.10}, there exists a bounded $n$ complement $B^+ \subs Z_0.$ for the generalized pair $(Z_0,B \subs Z_0. , \bM .)$.
By Theorem \ref{effective cbf}, there is a bounded $q$ such that
\begin{equation} \label{equation linear eq}
q(\K Z. + B_Z) \sim q \phi^*(\K Z_0. + B \subs Z_0. + \bM  Z_0.).
\end{equation}
Set $G \subs Z_0. \coloneqq B \subs Z_0. ^+ - B \subs Z_0.$, and define $G_Z \coloneqq \phi^* G \subs Z_0.$ and $B_Z^+ \coloneqq B_Z + G_Z$.
By equation \eqref{equation linear eq}, it follows that
\[
q(\K Z. + B_Z^+) \sim q \phi^*(\K Z_0. + B^+ \subs Z_0. + \bM  Z_0.).
\]
Up to taking a bounded multiple depending only on the setup of the problem, we may assume that $q(\K Z_0. + B^+ \subs Z_0. + \bM  Z_0.)$ is Cartier.
Thus, it follows that $q(\K Z. + B_Z^+)$ is Cartier.
Then, we are left with showing that $(Z,B_Z^+)$ is log canonical.
Since $B^+_Z-B_Z$ is vertical over $Z_0$, it follows that $(Z,B_Z^+)$ is log canonical over the generic point of $Z_0$.
Then, notice that $(Z_0,B_{Z_0}^+ , \bM.)$ is the generalized pair induced by $(Z,B_Z^+)$.
Then, by \cite{Amb99}*{Proposition 3.4}, as $(Z_0,B_{Z_0}^+ , \bM.)$ is generalized log canonical, then so is $(Z,B_Z^+)$.
\end{proof}

\section{Lifting complements from surfaces}\label{lifting-complements-surfaces}

In this section, we prove Theorem~\ref{main-theorem-hyper} under the assumptions that the contraction $X \rar T$ factors through a surface.
In the notation of \S\ref{sketch}, we consider the case when $\dim Z_0 = 2$.
In particular, we prove the following proposition.

\begin{proposition}\label{prop:lifting-complements-surfaces}
Let $\mathcal{R} \subset [0,1]$ be a finite set of rational numbers.
Then, there exists a natural number $n$ only depending on $\mathcal{R}$ such that the following holds.
Let $(Z,B_Z)$ be a log canonical pair such that $\dim Z = 3$ and the coefficients of $B_Z$ belong to $\Phi(\mathcal{R})$.
Let $Z \rightarrow T$ be a contraction between normal quasi-projective varieties  such that $(Z,B_Z)$ is $\qq$-complemented over $t\in T$.
Moreover, assume that $Z\rightarrow T$ factors as $\phi \colon Z \rar Z_0$ and $Z_0 \rar T$, where
\begin{itemize}
    \item $\dim Z_0 = 2$;
    \item $Z \rar Z_0$ and $Z_0 \rar T$ are contractions; and
    \item $\K Z. + B_Z \sim_\qq \phi^*L$, where $L$ is a $\qq$-Cartier divisor on $Z_0$.
\end{itemize}
Then, up to shrinking $T$ around $t$, we can find an effective divisor
\[
\Gamma \sim_{T} -n(K_Z+B_Z)
\]
such that $(Z,B_Z+\Gamma/n)$ is a log canonical pair.
\end{proposition}

\begin{proof}
Let $(Z_0,B \subs Z_0. , \bM  Z_0.)$ be the generalized pair induced by $(Z,B_Z)$ via $\phi$ and the canonical bundle formula.
Then, by Theorem \ref{effective cbf}, there exist a finite set of rational numbers $\mathcal{S} \subset [0,1]$ and a positive integer $q$ such that the coefficients of $B_{Z_0}$ belong to $\Phi(\mathcal{S})$ and $q\bM.$ is b-free.
Let $Z_0'$ be a higher model of $Z_0$ where the moduli b-divisor $\bM .$ descends.
Let $\Delta \subs Z_0.$ be as in Corollary \ref{remark turning M into boundary}.
In particular, we have $0 \leq \Delta \subs Z_0. \sim_\qq \bM Z_0.$ and the generalized discrepancies of the generalized pair $(Z_0,B \subs Z_0. , \bM  Z_0.)$ are less or equal to the discrepancies of the pair $(Z_0,B \subs Z_0. + \Delta \subs Z_0.)$.
\newline
Let $B'_Z$ be a $\qq$-complement for $(Z,B_Z)$ over $t \in T$.
In particular, up to shrinking $T$ around $t$, we may assume that $\K Z. + B_Z + B_Z' \sim_\qq 0/T$.
Since $\K Z. + B_Z \sim_\qq 0/Z_0$, it follows that $B'_Z \sim_\qq 0/Z_0$.
In particular, $B'_Z$ is vertical over $Z_0$.
Then, we have $B'_Z = \phi^* B' \subs Z_0.$ for some effective $\qq$-Cartier divisor $B' \subs Z_0.$ on $Z_0$.
It follows that the generalized pair $(Z_0,B \subs Z_0. + B' \subs Z_0. , \bM .)$ is induced by $(Z,B_Z+B'_Z)$ via the canonical bundle formula.
Since $(Z,B_Z+B'_Z)$ is log canonical, it follows that $(Z_0,B \subs Z_0. + B' \subs Z_0. , \bM .)$ is generalized log canonical.
Furthermore, as $\Delta \subs Z_0.$ is chosen generically, $(Z_0,B \subs Z_0. + B' \subs Z_0. + \Delta \subs Z_0.)$ is a log canonical pair.
In particular, $B \subs Z_0.'$ is a $\qq$-complement over $t \in T$ for the pair $(Z_0,B \subs Z_0. + \Delta \subs Z_0.)$.
\newline
By Corollary \ref{remark turning M into boundary}, the coefficients of $B \subs Z_0. + \Delta \subs Z_0.$ belong to $\Phi(\mathcal S)$.
Thus, by Theorem \ref{thm:log-canonical-surface-complements}, the pair $(Z_0,B \subs Z_0. + \Delta \subs Z_0.)$ admits a bounded $n$-complement $B \subs Z_0.^+$ over $t \in T$.
We may assume that the $q$ as in Theorem \ref{effective cbf} divides $n$.
Thus, it follows that $B \subs Z_0.^+$ is a complement also for the generalized pair $(Z_0,B \subs Z_0. , \bM .)$ over $t \in T$.
Set $G \subs Z_0. \coloneqq B \subs Z_0. ^+ - B \subs Z_0.$, and define $G_Z \coloneqq \phi^* G \subs Z_0.$ and $B_Z^+ \coloneqq B_Z + G_Z$.
Arguing as in Case 2 of the proof of Theorem \ref{thm:log-canonical-surface-complements}, it follows that $B_Z^+$ is an $n$-complement for $(Z,B_Z)$ over $t \in T$.
\end{proof}

\section{Lifting complements from curves}\label{lifting-complements-curves}

In this section, we prove Theorem~\ref{main-theorem-hyper} under the assumptions that the contraction $Z \rar T$
factors through a variety $Z_0$ of dimension at most one over which $K_Z+B_Z$ is $\qq$-trivial,
i.e., the morphism factors through a curve $Z_0$ and $K_Z+B_Z\sim_{\qq,Z_0} 0$, or $Z_0=T=\Spec(k)$
and $K_Z+B_Z\sim_{\qq}0$. In particular, we prove the following proposition.

\begin{proposition}\label{prop:lifting-complements-curves}
Let $\mathcal{R} \subset [0,1]$ be a finite set of rational numbers.
Then, there exists a natural number $n$ only depending on $\mathcal{R}$ that satisfies the following.
Let $(Z,B_Z)$ be a log canonical pair such that $\dim Z = 3$ and the coefficients of $B_Z$ belong to $\Phi(\mathcal{R})$.
Let $Z \rightarrow T$ be a contraction between normal quasi-projective varieties such that $(Z,B_Z)$ is $\qq$-complemented over $t\in T$.
Moreover, assume that $Z\rightarrow T$ factors as $Z \rar Z_0 \rar T$, where
\begin{itemize}
    \item $\dim Z_0 \leq 1$;
    \item $Z \rar Z_0$ and $Z_0 \rar T$ are contractions; and
    \item $\K X. + B \sim_{\qq,Z_0} 0$.
\end{itemize}
Then, up to shrinking $T$ around $t$, we can find
\[
\Gamma \sim_{T} -n(K_X+B)
\]
such that $(X,B+\Gamma/n)$ is a log canonical pair.
\end{proposition}

\begin{proof}
We will prove the statement in the three possible cases for $(\dim Z_0, \dim T)$ in $\{ (0,0), (1,1), (1,0)\}$. 
The above cases will be called case $1,2$ and $3$, respectively.
In what follows, according to the notation of \S\ref{sketch}, we will denote by $\phi\colon Z \rightarrow Z_0$ the
$(K_Z+B_Z)$-trivial morphism, and we will denote by $B'_Z$ the $\qq$-complement of $(Z,B_Z)$ over the point $t\in T$.

\textbf{Case 1:} We deal with the projective case of the statement.
\newline
In the case that $\dim Z_0=\dim T=0$ we have that 
$Z_0=T=\Spec(k)$ for some algebraically closed field $k$, and $K_Z+B_Z\sim_\qq 0$.
Bounding the index of $K_Z+B_Z$, in this case, is known as the projective index conjecture for $\qq$-trivial
log canonical $3$-folds, and this result is proved in~\cite{thesis}*{Theorem 1.8.7}.
Indeed, we know that there exists $n$, only depending on $\mathcal{R}$ such that 
$n(K_Z+B_Z)\sim 0$, proving the claim in the first case.

\textbf{Case 2:} We prove the boundedness of complements for $3$-folds locally over a curve.
\newline
Now we consider the case $\dim Z_0=\dim T=1$,
which means that the contraction $Z_0\rightarrow T$ between normal curves is an isomorphism.
Hence, we may assume that $Z_0=T$, and we are trying to complement the log canonical $3$-fold 
$(Z,B_Z)$ over the smooth point $t$ of the curve $T$.
By Lemma \ref{rmk generic point}, we may assume that $t$ is a closed point of $T$.
This case is the boundedness of complements for $\qq$-trivial $3$-folds locally over a curve.
Observe that, by the assumptions of the statement, in this case, $K_Z+B_Z$ is $\qq$-trivial over the point $t\in T$.
\newline
In order to prove the statement, we will make some reductions.
By Lemma~\ref{rmk lcc}, we may assume that the log canonical pair $(Z,B_Z)$ has a log canonical center that is mapped to $t$.
\newline
Therefore, when we apply the canonical bundle formula, up to shrinking around $t \in T$, on the base we obtain a generalized pair of the form
$(T,\{t\}, \bM.)$.
More precisely, we can write
\begin{equation}
\label{s6:can-bun}
q(K_Z+B_Z) \sim q\phi^*(K_T+\{t\} + \bM T.),   
\end{equation}
where $q$ is some natural number, $\{t\}$ is the boundary divisor, and $\bM T.$ is the moduli part.
\newline
Observe that, since $T$ is smooth at $t$, then the Cartier index of $\bM T.$ is equal to the Weil index of $\bM T.$.
We will run a minimal model program for $K_Z$ over $T$, which terminates with a model $Z'$.
Notice that, as $(Z,B_Z)$ is dlt, then $Z$ and $Z'$ are klt.
We either have a Mori fiber space $Z'\rightarrow Z_1$ over $T$,
or we have a semi-ample divisor $K_{Z'}$ over $T$,
which then induces a morphism $Z'\rightarrow Z_1$.
We will proceed depending on what kind of morphism $Z'\rightarrow Z_1$ is.
The case in which $Z'\rightarrow Z_1$ is a Mori fiber space will be called Case $2.1$, and the case in which $Z'\rightarrow Z_1$ is the ample model will be called case Case $2.2$.
Since each step of this minimal model program is $(K_Z+B_Z)$-trivial, 
by part (4) of Lemma \ref{reduction complements}, it suffices to produce an $n$-complement for the log canonical pair $(Z',B_{Z'})$ over $t\in T$.
Replacing $(Z,B_Z)$ with $(Z',B_{Z'})$, we may assume that $K_Z$ has either a good minimal model over $T$
or a Mori fiber space structure over $T$.

\textbf{Case 2.1.a:} The MMP terminates with a MFS to a curve.
\newline
Assume that $Z\rightarrow Z_1$ is a Mori fiber space and $\dim Z_1=1$.
In this case the contraction $Z_1\rightarrow T$ is an isomorphism.
Therefore, since $Z$ is klt, the morphism $Z\rightarrow T$ is of Fano-type.
Thus, we can apply~\cite{Bir16a}*{Proposition 6.3} to conclude the claim.

\textbf{Case 2.1.b:} The MMP terminates with a MFS to a surface.
\newline
Assume that $Z\rightarrow Z_1$ is a Mori fiber space and $\dim Z_1=2$.
This case follows from applying Proposition~\ref{prop:lifting-complements-surfaces}.

\textbf{Case 2.2.a:} The MMP terminates with a good minimal model mapping to a curve.
\newline
Assume that $K_Z$ is semi-ample over $T$ and the defined morphism $Z\rightarrow Z_1$ has $\dim Z_1=1$.
In this case, we have that $Z_1\rightarrow T$ is an isomorphism and $K_Z\sim_{\qq,T}0$.
First, we will reduce to the case in which the general fiber of the morphism $Z\rightarrow T$ is smooth.
\newline
Assume that the general fiber of $Z \rar T$ is not smooth.
Then, the pair $(Z,B_Z)$ has a horizontal non-terminal valuation over $T$.
Let $\pi \colon Z'' \rightarrow Z$ be a projective birational morphism that only extracts a divisor computing the minimal log discrepancy
of the general fiber of $Z \rightarrow T$. 
The divisor extracted on $Z''$ is horizontal over $T$.
Hence, we can write $\pi^*(K_Z+B_Z)=K_{Z''}+B_{Z''}$, where $B_{Z''}$ has a unique horizontal component $E$ whose
coefficient in $B_{Z''}$ is $1$ minus the minimal log discrepancy of a log canonical surface $Z_{\eta}$.
Hence, the coefficient of $E$ belongs to a fixed set satisfying the descending chain condition~\cites{Sho91,Ale93}.
By the global ascending chain condition~\cite{HMX14}*{Theorem 1.5},
we conclude that the coefficient of $B_{Z''}$ along $E$ belongs to a finite set 
$\mathcal{F} \subset [0,1]$ only depending on $\dim Z_\eta =2$.
\newline
By part (4) of Lemma \ref{reduction complements}, a log canonical $n$-complement for $(Z'',B_{Z''})$ over $t\in T$ 
pushes forward to a log canonical $n$-complement for $(Z,B_Z)$ over $t\in T$.
Thus, it suffices to produce an $n$-complement for $K_{Z''}+B_{Z''}$, whose coefficients belongs to $\Phi(\mathcal{R})\cup \mathcal{F}$.
Since the log canonical pair $(Z'',B_{Z''})$ is $\qq$-trivial over $T$, we conclude that 
$K_{Z''}+(1-\epsilon)B_{Z''}$ is not pseudo-effective over $T$ for $\epsilon \in (0,1)$.
Indeed, $B_{Z''}$ contains a horizontal component over $T$.
We run a minimal model program for $K_{Z''}+(1-\epsilon)B_{Z''}$ over $T$,
which terminates with a Mori fiber space $Z^{(3)}\rightarrow Z_2$ over $T$.
Observe that, by part (4) of Lemma \ref{reduction complements}, it is enough to find an $n$-complement for the divisor $K_{Z^{(3)}}+B_{Z^{(3)}}$ over $t\in T$, as the above minimal model program is $K_{Z''}+B_{Z''}$-trivial.
If $Z_2$ is a surface, then we conclude the existence of an $n$-complement by Proposition \ref{prop:lifting-complements-surfaces}.
If $Z_2$ is a curve, then $Z_2\rightarrow T$ is an isomorphism.
Hence, up to replacing $(Z,B)$ with $(Z^{(3)},B_{Z^{(3)}})$, we may assume that the morphism $Z\rightarrow T$ is of Fano-type, 
and we are in the situation of Case 2.1.a.
So, the claim holds if the general fiber of $Z \rar T$ is not smooth.
\newline
Now, we may assume that the generic fiber of $X\rightarrow Z_0$ is a smooth projective surface
with $K_{X_\eta} \sim_\qq 0$.
Thus, by~\cite{FM00}*{Theorem 4.5}, we know that $q$ and the Weil index of $\bM T.$ 
only depend on the index of $K_{X_\eta}$ and the second Betti number of the index one cover of $X_{\eta}$.
Observe that, by generic smoothness and invariance of plurigenera, the index of the generic fiber is equal to the index of a general fiber.
On the other hand, since this morphism is topologically trivial over a non-empty open subset (see, e.g.~\cite{Ver76}*{Corollarie 5.12.7}), we conclude that the Betti numbers of the index one cover of $X_{\eta}$ coincide with the Betti numbers of the index one cover of a general fiber. 
Thus, by the classification of smooth surfaces with $\qq$-trivial canonical divisor over an algebraically closed
field of characteristic zero, we know that both the index of the canonical divisor and the second Betti number of the index one cover
can only take finitely many possible values.
In particular, the natural number $q$ and the Weil index of $\bM T.$ can take finitely many possible values as well.
Hence, the Cartier index of $K_T+\{t\}+\bM T.$ only depends on $\mathcal{R}$.
By the linear equivalence~\eqref{s6:can-bun}, we conclude that the Cartier index of $K_Z+B_Z$ only depends on $\mathcal{R}$.

\textbf{Case 2.2.b:} The MMP terminates with a good minimal model mapping to a surface.
This case also follows from Proposition~\ref{prop:lifting-complements-surfaces}.

\textbf{Case 3:} We prove the boundedness of global complements for projective $3$-folds admitting a dominant morphism to a curve.
\newline
Now we will consider the case in which $\dim Z_0=1$ and $\dim T=0$, 
which means that $T={\rm Spec}(k)$ for some algebraically closed field $k$ of characteristic zero.
In this case, $Z_0$ is a projective curve over $T$, and we will aim to construct a
$\qq$-complemented generalized pair on $Z_0$.
Afterward, we will find a log canonical $n$-complement for this pair on $Z_0$ and lift it to a log canonical $n$-complement on $Z$.
The strategy of the third case is similar to the strategy of the second case, 
with the difference that now we need to produce a projective complement on the curve $Z_0$,
while in the second case we produced a local complement around the point $t$ on the curve $T$.
\newline
We run a minimal model program for $K_Z$ over $Z_0$, which terminates with $Z'$.
We either have a semi-ample divisor $K_{Z'}$ over $Z_0$ which induces the morphism $Z'\rightarrow Z_1$,
or we have a Mori fiber space $Z'\rightarrow Z_1$ over $Z_0$.
The former case will be called case 3.1, and the latter case 3.2.
Since each step of this minimal model program is $(K_Z+B_Z)$-trivial, by part (4) of Lemma \ref{reduction complements},
it suffices to produce a log canonical $n$-complement for $(K_{Z'}+B_{Z'})$ over $T={\rm Spec}(k)$.
Replacing $(Z,B_Z)$ with $(Z',B_{Z'})$, we may assume that either $Z$ is a good minimal model over $Z_0$
or $Z$ has a Mori fiber space structure $Z\rightarrow Z_1$ over $Z_0$.

\textbf{Case 3.1.a:} The MMP terminates with a MFS mapping to a curve.
\newline
Assume that $Z\rightarrow Z_1$ is a Mori fiber space and $\dim Z_1=1$.
In this case the morphism $Z_1\rightarrow Z_0$ is an isomorphism,
hence we may assume that the morphism $Z \rightarrow Z_0$ is of Fano-type. 
By~\cite{Bir16a}*{Proposition 6.3}, we may assume there exists a natural number $q$,
a finite set of rational numbers $\mathcal{S} \subset [0,1]$, and a generalized pair 
$(Z_0,B_{Z_0}, \bM .)$ on $Z_0$, such that 
\[
q(K_{Z}+B_Z) \sim q(K_{Z_0}+B_{Z_0}+\bM {Z_0}.),
\]
where $q \bM {Z_0}.$ is an integral b-Cartier b-divisor and the coefficients of $B_{Z_0}$ belong to $\Phi(\mathcal{S})$.
Observe that the projective generalized log canonical pair $(Z_0,B_{Z_0}, \bM.)$ is
$\qq$-complemented over ${\rm Spec}(k)$.
Indeed, since $(Z,B_Z)$ is $\qq$-complemented by an effective divisor $B'_Z$,
we can apply the canonical bundle formula for $(Z,B_Z+B'_Z)$ with respect to the morphism
$Z\rightarrow Z_0$ to obtain a $\qq$-trivial generalized pair 
$(Z_0,B_{Z_0}+B'_{Z_0},\bM.)$.
We conclude that either $B_{Z_0}=B'_{Z_0}=\bM {Z_0}.=0$ and $Z_0$ is a projective elliptic curve, or $Z_0\simeq \pp^1$.
In the former case, we have $q(K_Z+B_Z)\sim 0$ and we are done.
In the latter case,
since $qM_{Z_0}$ is integral Weil and the coefficients of $B_{Z_0}$ belong to $\Phi(\mathcal{S})$,
we conclude that there exists a natural number $n$, only depending on $q$ and $\mathcal{S}$,
such that there exists 
\[
\Gamma_{Z_0}\in |-n(K_{Z_0}+B_{Z_0}+\bM {Z_0}.)|,
\]
with $(Z_0, B_{Z_0}+\Gamma_{Z_0}/n,\bM {Z_0}.)$ generalized log canonical \cite{Bir16a}*{Theorem 1.10}.
Since $q$ and $\mathcal{S}$ only depend on $\mathcal{R}$, we conclude that $n$ itself only depends on $\mathcal{R}$.
Hence, it follows that $\Gamma_Z \coloneqq  \phi^*(\Gamma_{Z_0})/qn$ satisfies the property 
\[
qn( K_Z+B_Z+\Gamma_Z ) \sim 0.
\]
We claim that $(Z,B_Z+\Gamma_Z)$ is log canonical.
Indeed, applying the canonical bundle formula for $(Z,B_Z+\Gamma_Z)$ with respect to $Z\rightarrow Z_0$,
we obtain the generalized log canonical pair $(Z_0,B_{Z_0}+\Gamma_{Z_0}/n,\bM {Z_0}.)$.
Therefore, by \cite{Amb99}*{Proposition 3.4}, $(Z,B_Z+\Gamma_Z)$ is log canonical.

\textbf{Case 3.1.b:} The MMP terminates with a MFS mapping to a surface.
\newline
Assume that $Z\rightarrow Z_1$ is a Mori fiber space and $\dim Z_1=2$.
In this case, the existence of a log canonical $n$-complement for $n$ only depending on $\mathcal{R}$ follows from Proposition~\ref{prop:lifting-complements-surfaces}.

\textbf{Case 3.2.a:} The MMP terminates with a good minimal model mapping to a curve.
\newline
Assume that $Z\rightarrow Z_1$ is the morphism defined by the relatively semi-ample divisor $K_Z$ over $Z_0$
and $\dim Z_1=1$.
Then, we have that $Z_1\rightarrow Z_0$ is an isomorphism.
Moreover, we have that $K_Z+B_Z$ and $K_Z$ are $\qq$-trivial over $Z_0$.
First, we will reduce to the case in which the general fiber of the morphism $Z\rightarrow Z_0$ is smooth.
\newline
Assume that the general fiber fo $Z \rar Z_0$ is not smooth.
Then, the pair $(Z,B_Z)$ has a horizontal non-terminal valuation over $Z_0$.
Let $\pi \colon Z''\rightarrow Z$ be a projective birational morphism that only extracts the minimal log discrepancy
of the general fiber of $Z\rightarrow Z_0$. 
Hence, we can write $\pi^*(K_Z+B_Z)=K_{Z''}+B_{Z''}$, where $B_{Z''}$ has a unique horizontal component $E$ whose
coefficient in $B_{Z''}$ is 1 minus the minimal log discrepancy of the log canonical surface $Z_{\eta}$.
Thus, by \cites{Sho91,Ale93}, this coefficient belongs to a fixed set satisfying the descending chain condition.
By the global ascending chain condition~\cite{HMX14}*{Theorem 1.5},
we conclude that the coefficient of $B_{Z''}$ along $E$ belongs to a finite set 
$\mathcal{F} \subset [0,1]$ which only depends on $\dim Z_\eta=2$.
By part (4) of Lemma  \ref{reduction complements}, it suffices to produce an $n$-complement for $K_{Z''}+B_{Z''}$ whose coefficients belongs to $\Phi(\mathcal{R})\cup \mathcal{F}$.
We run a minimal model program for $K_{Z''}+(1-\epsilon)B_{Z''}$ over $Z_0$, 
which terminates with a Mori fiber space $Z^{(3)}\rightarrow Z_2$.
By part (4) of Lemma \ref{reduction complements}, it is enough to find an $n$-complement for the divisor $K_{Z^{(3)}}+B_{Z^{(3)}}$ over ${\rm Spec}(k)$.
If $Z_2$ is a surface, then we conclude by Proposition~\ref{prop:lifting-complements-surfaces}.
If $Z_2$ is a curve, then $Z_2\rightarrow Z_0$ is an isomorphism.
Hence, up to replacing $(Z,B_Z)$ with $(Z^{(3)},B^{(3)})$, we may assume that $Z\rightarrow Z_0$ is of Fano-type, 
and we are in the situation of Case 2.1.a.
So, the claim holds if the general fiber of $Z \rar Z_0$ is not smooth.
\newline
Now, we may assume that the general fiber of $Z\rightarrow Z_0$ is a smooth projective surface
with $K_{X_\eta} \sim_\qq 0$.
Thus, by~\cite{FM00}*{Theorem 4.5}, we know that we can write an effective canonical bundle formula
\[
q(K_Z+B_Z)\sim q(K_{Z_0}+B_{Z_0}+ \bM {Z_0}.),
\]
where $q$ and the Weil index of $\bM T.$ 
only depend on the index of $K_{X_\eta}$ and the second Betti number of the index one cover of $X_{\eta}$.
As in Case 2.2.a, we know that such numbers only depend on the general fiber.
By the classification of smooth surfaces with $\qq$-trivial canonical divisor over an algebraically closed
field of characteristic 0, we know that both the index of the canonical divisor and the second Betti number of the index one cover
can take finitely many possible values.
Hence, the natural number $q$ and the Weil index of $\bM T.$ can take finitely many possible values as well.
Moreover, by \cite{FM00}*{Theorem 4.5}, the coefficients of $B_{Z_0}$ belong to a set $\Phi$ satisfying the descending chain condition
with rational accumulation points.
The fact that the accumulation points of $\Phi$ are rational follows from~\cite{HMX14}*{Theorem 1.11}.
The set $\Phi$ only depends on $\mathcal{R}$.
If $Z_0$ is an elliptic curve, the claim is trivial.
Otherwise, it follows from~\cite{FM00}*{Theorem 4.5}.
By~\cite{FM18}*{Theorem 1.2}, we may find $n$ only depending on $\Phi$ and $q$   
such that there exists 
\[
\Gamma_{Z_0}\in |-n(K_{Z_0}+B_{Z_0}+\bM {Z_0}.)|,
\]
with $(Z_0, B_{Z_0}+\Gamma_{Z_0}/n,\bM.)$ generalized log canonical.
Since $q$ and $\Phi$ only depend on $\mathcal{R}$ we conclude that $n$ itself only depends on $\mathcal{R}$.
Hence, we conclude that $\Gamma_Z \coloneqq  \phi^*(\Gamma_{Z_0})/qn$ satisfies that 
\[
qn( K_Z+B_Z+\Gamma_Z ) \sim 0.
\]
We claim that $(Z,B_Z+\Gamma_Z)$ is log canonical.
Indeed, applying the canonical bundle formula for $(Z,B_Z+\Gamma_Z)$ with respect to $Z\rightarrow Z_0$
we obtain the generalized log canonical pair $(Z_0,B_{Z_0}+\Gamma_{Z_0}/n+M_{Z_0})$.
Hence, by \cite{Amb99}*{Proposition 3.4}, $(Z,B_Z+\Gamma_Z)$ is log canonical.
Thus, the $\qq$-complemented projective log canonical pair $(Z,B_Z)$ has
a log canonical $n$-complement for $n$ only depending on $\mathcal{R}$.

\textbf{Case 3.2.b:} The MMP terminates with a good minimal model mapping to a surface.
\newline
Assume that $Z\rightarrow Z_1$ is the morphism defined by the semi-ample divisor $K_X$ over $Z_0$
and $\dim Z_1=2$.
In this case, the existence of a log canonical $n$-complement for $n$ only depending on $\mathcal{R}$ follows from Proposition~\ref{prop:lifting-complements-surfaces}.
\end{proof}

\section{Effective Koll\'ar's gluing theory}\label{effective-Kollar-gluing}

In this section, we prove a generalization of Koll\'ar's gluing, which will be needed to descend complements onto semi-dlt pairs from their normalization.
In particular, the main statement is the following.

\begin{proposition}\label{prop:effective-kollar-gluing}
Let $(X,B)$ be a semi-dlt pair with $\dim X \leq 2$ and normalization $(X^\nu,B^\nu + D^\nu)$, and let $X\rightarrow T$ be a contraction.
Assume that, for some $n \in \nn$, we have 
$n(K_{X^\nu}+B^\nu + D^\nu)\sim_{T} 0$ and $n(K_{X^\nu}+B^\nu + D^\nu)$ is Cartier. 
Then, there exists $m$, only depending on $n$, such that 
$m(K_X+B)\sim_T 0$.
\end{proposition}

\begin{remark} \label{remark effectiveness missing}
In the setup of Proposition \ref{prop:effective-kollar-gluing}, the existence of $m$ depending on $(X,B)$ follows from \cite{HX16}*{Theorem 1.4}.
Thus, we are left with showing that we can find a uniform bound only depending on $n$.
\end{remark}

In order to address the $\dim X = 2$ case of Proposition \ref{prop:effective-kollar-gluing}, we first state a conjecture and prove it in dimension 1.

\begin{conjecture}\label{conj: bounded-B-birational-representation}
Let $n$ and $d$ be two positive integers.
Let $(X,B)$ be a connected projective log canonical pair of dimension $d$ such that $n(K_X+B)\sim 0$.
Then, there exists $m,N$ depending only on $n,d$ such that, for every positive integer $l$, the action $\rho_{ml} \colon \mathrm{Bir}(X,B) \rightarrow \mathrm{Aut} ( H^0(X, \O X. (m(K_X+B))))$ satisfies $|\rho_{lm}(\mathrm{Bir}(X,B))|\leq N$.
The action $\rho_{ml}$ is induced by pull-back as follows:
if $\phi \in \mathrm{Bir}(X,B)$ and $\gamma \in H^0(X, \O X. (m(K_X+B)))$, then $\rho_{ml} \colon \phi \mapsto (\gamma \mapsto \phi^* \gamma)$.
\end{conjecture}

\begin{remark}
Assume the setup of Conjecture \ref{conj: bounded-B-birational-representation}.
We note that $H^0(X, \O X. (m(K_X+B)))$ is 1-dimensional.
In particular, this implies that $\rho_m(\mathrm{Bir}(X,B))\subset \mu_N$, the group of $N$-th roots of unity. Therefore, Conjecture \ref{conj: bounded-B-birational-representation} is equivalent to the existence of a positive integer $k$ depending only on $n$ and $d$, such that $\rho_k(\mathrm{Bir}(X,B))$ is trivial.
\end{remark}

Now we will show the above in dimension 1.

\begin{proposition} \label{prop b-rep dim 1}
	Conjecture \ref{conj: bounded-B-birational-representation} hold in dimension 1.
\end{proposition}

\begin{proof}
In this case, $X$ is either a rational curve or an elliptic curve.
If $X$ is an elliptic curve, then $B=0$. 
In particular, we can take $m=1$, and it is well-known that $N\leq 12$, i.e. by \cite{ast}*{12.2.9.1 Complement}.
\newline
If $X=\mathbb{P}^1$, we have the following two cases:
\begin{itemize}
	\item[(1)] $|\Supp(B)|=2$.
	In this case, we have $B=P+Q$ for some points $P,Q$.
	Then, we can take $m=1$ and $N=2$ by the residue theorem; and
	\item[(2)] $|\Supp(B)|\geq 3$.
	Since $n(\K X. + B) \sim 0$ and $\deg n \K X. = -2n$, it follows that $|\Supp(B)| \leq 2n$.
	Since an automorphism of $\mathbb P ^1$ is determined by 3 points and we are forced to choose them in $|\Supp(B)|$, it follows that $|\mathrm{Bir}(X,B)|=|\mathrm{Aut}(X,B)| \leq \binom{2n}{3}$, which implies the result.
\end{itemize}
This concludes the proof.
\end{proof}

Now we are ready to prove Proposition \ref{prop:effective-kollar-gluing}.

\begin{proof}[Proof of Proposition \ref{prop:effective-kollar-gluing}]
We treat the cases $\dim X =1 $ and $\dim X = 2$ separately.

{\bf Case 1:} In this case, we assume that $\dim X =1$.
\newline
First assume that $\dim T = 0$.
In this case, $X$ is either an elliptic curve, an irreducible rational curve with exactly one nodal singularity, a cycle of smooth rational curves, or a tree of smooth rational curves.
In the former three cases, the claim is trivial, as these are curves of arithmetic genus 1, while in the latter case we may take $n=m$, since $\K X.$ is Cartier at the nodal points and a tree of rational curves does not have torsion line bundles.
\newline
Now, assume that $\dim T =1$.
Since $X$ is not irreducible, some components of $X$ may be contracted by $X \rar T$;
those are addressed as in the projective case $\dim T =0$.
Thus, we may assume that $X \rar T$ is an isomorphism.
As before, we may take $n=m$, as it suffices to make $K_X+B$ Cartier.

{\bf Case 2:} In this case, we assume that $\dim X=2$.
\newline
First, we recall that, by Remark \ref{remark effectiveness missing}, we only need to show that there exists an effective bound for $m$ only depending on $n$.
The effectiveness of the choice will be guaranteed by Proposition \ref{prop b-rep dim 1}, while the structure of the proof remains identical to the one of \cite{HX16}*{Theorem 1.4} (see \cite{HX13}*{\S3}).
Thus, we will address the effectiveness of the choices, while we will only provide a sketch of the use of Koll\'ar's gluing theory and refer to \cite{HX13}*{\S3} for its details. 
\newline
For dimensional reasons, the double locus $D^\nu$ consists of curves.
By Proposition \ref{prop b-rep dim 1}, we may choose $m$ depending only on $n$ such that $\rho_m(\mathrm{Bir}(Z,B_Z))$ is trivial for all $Z$ irreducible components of $D^\nu$ that are vertical over $T$ and hence projective.
Now, define $M \coloneqq m(\K X^\nu. + B^\nu + D^\nu)$.
As we may assume $n|m$, $M$ is Cartier and $M\sim_T 0$.
\newline
Now, we follow the proof in \cite{HX13}*{Proposition 3.1}.
We consider the morphism $f \colon X^\nu \rightarrow T$, and we have that $M  =  m(K_{X^\nu}+B^\nu + D^\nu)\sim_T 0$ is Cartier and relatively free.
Hence, $f$ is the morphism induced by $m(K_{{X}^\nu}+B^\nu + D^\nu)$ over $T$.
Let $H$ be a line bundle on $T$ such that $f^*H \sim m(K_{X^\nu}+B^\nu + D^\nu)$.  
Let $p_X  \colon  X_M\rightarrow X$ and $p_T  \colon  T_H\rightarrow T$ be the total spaces of the line bundles of $M$ and $H$, respectively.
Define $D_M  \coloneqq  p_X^{-1}(D^\nu)$, $Y  \coloneqq  f(D^\nu)$, and $Y_H  \coloneqq  p_T^{-1} (Y)$.
We see that the involution $\tau \colon  D^\nu \rightarrow D^\nu$ induces a set relation on $Y_H\rightarrow T_H$.
Repeating the proof of \cite{HX13}*{Proposition 3.1}, we obtain that the above set relation is a finite set-theoretic equivalence.
The only difference from the proof of \cite{HX13}*{Proposition 3.1} is that we do not need to replace $m$ with a multiple, as it is already guaranteed to trivialize the B-canonical representations of the log canonical centers that are projective over $T$ (the $Z_i$'s in the notation of \cite{HX13}).
Thus, again following the proof of \cite{HX13}*{Proposition 3.1}, by applying \cite{HX13}*{Theorem 3.8},  we see that the quotient $A$ with respect to $Y_H\rightarrow T_H$ exists. 
This implies that there is a line bundle $A$ on $T$ whose pull-back to $X$ is linearly equivalent to $m(K_X+B)$.
In particular, this implies that $m(K_X+B)\sim_{T} 0$.
\end{proof}

\section{Complements for semi-dlt surfaces}\label{sec:complements-sldt}

This section aims to prove some key results about complements for semi-dlt surfaces.
These results will play a crucial role in \S\ref{log-canonical-Fano}.
We start with some remarks that will be useful in the proof of Proposition \ref{prop:log-canonical-dlt-surface-complements-smooth-curves}.

\begin{remark} \label{remark points on curves}
Let $C$ be a smooth curve.
Let $P_1, \ldots ,P_n$ be $n$ closed points on $C$.
Then, for any Cartier divisor $D$ on C, we have $D\sim 0$ in a common neighborhood of $P_1,\ldots,P_n$.
Indeed, let $Q$ be a closed point on $C$ with $Q \neq P_i$ for all $i$.
Then, for sufficiently large $m$, $D+mQ$ is very ample.
Hence, we can find $0\leq R\sim D+mQ$ such that $P_i$ is not in $\Supp(R)$ for all $i$.
Hence, we get $D\sim R-mQ\sim 0$, in a common neighborhood of $P_1,\ldots,P_n$.
Similarly, the same observation holds in the relative case when considering a finite map $f \colon C\rightarrow E $ over another curve $E$, where $E$ is irreducible but not necessarily smooth.
In this case, we choose $Q$ and $R$ away from the fibers of $f$ containing $P_1,\ldots,P_n$.
\end{remark}

\begin{remark} \label{remark Galois}
Let $(Y,D)$ be a log canonical 3-fold and let $(Y',D')$ be a dlt model.
Let $C$ be a 1-dimensional log canonical center on $Y$, and let $C'$ be its normalization, which is a smooth curve.
Assume that $(Y',D')$ has a 1-dimensional log canonical center $E$ mapping onto $C$.
Notice that $E$ is a normal and hence smooth curve.
Then, by part (5) of \cite{Kol13}*{Theorem-Definition 4.45}, the morphism $f \colon E \rightarrow C'$ is Galois and finite. 
Now, define $K_E+D_E \coloneqq (K_{Y'}+D')|_E$ by adjunction.
By part (5) of \cite{Kol13}*{Theorem-Definition 4.45}, the pair $(E,D_E)$ is $\mathrm{Gal}(E/C')$ invariant.
In particular, if we assume that coefficients of $D_E$ belong to $\Phi(\mathcal{R})$, where $\mathcal{R} \subset [0,1]$ is a finite set of rational numbers, then, by the Riemann--Hurwitz formula, there exists $(C',D_{C'})$ such that  $D_{C'}\in \Phi(\mathcal{R})$ and $K_E+D_E = f^*(K_{C'}+D_{C'}) $.
For details on the computation with the Riemann--Hurwitz formula, see \cite{thesis}*{Lemma 4.1.1}.
\end{remark}

By Remark \ref{remark Galois}, we need to prove the existence of semi-dlt relative complements in the following setting, which we will call Condition A.

\begin{definition}[\textbf{Condition A}] \label{definition condition A}
Let $(X,B) \rightarrow S \rightarrow  T$ be surjective morphisms between (not necessarily normal) quasi-projective varieties, and let $\mathcal R \subset [0,1]$ be a finite set of rational numbers.
Assume that $X\rightarrow T$ is a contraction, and let $t\in T$ be a closed point. 
We say that the contraction satisfies {\em Condition A} if the following holds: 
\begin{itemize}
    \item $(X,B)$ is a semi-dlt surface that is $\qq$-complemented over the closed point $t\in T$;
    \item the coefficients of $B$ belong to $\Phi (\mathcal R)$;
    \item $S$ is a possibly reducible semi-normal curve; and
    \item $T$ is either a possibly reducible semi-normal curve, or $T = \lbrace t \rbrace$.
\end{itemize}
Moreover, given any irreducible component $X_1$ of $X$ with the induced pair structure $(X_1,B_1)$, we assume that one of the following occurs:
\begin{enumerate}
	\item $X_1$ is mapped to a closed point $s\in S$ such that $s$ maps to $t$, and $K_{X_1}+B_1 \sim_{\qq} 0$;
	\item $X_1$ is mapped onto a curve $S_1$ in $S$, $S_1$ is mapped to $t$, and $K_{X_1}+B_1\sim_{\qq} f^*A$, where we write $f \colon X_1\rightarrow S_1$, and $-A$ is ample on the projective curve $S_1$; or
	\item $X_1$ is mapped onto a curve $S_1$ in $S$, and $S_1$ is mapped onto T, and $K_{X_1}+B_1\sim_{\qq,S_1} 0$.
\end{enumerate}  
Furthermore, if (2) or (3) occur (that is, $X_1$ is mapped onto a curve $S_1 \subset S$), we assume the following condition:
\begin{itemize}
    \item if $E$ is a component of $\rddown{B_1}$ that dominates $S_1$, then the following holds.
	The morphism $E \rar C$ is a Galois finite morphism, where $C$ is the normalization of $S_1$, $K_E+B_E  \coloneqq  (K_{X_1}+B_1)|_E$ is $\mathrm{Gal}(E/C)$ invariant, and the pair $(E,B_E)$ is only dependent on the choice of such $S_1$ and independent of the choice of $X_1$.
\end{itemize}
\end{definition}

Now we are ready the state the main proposition of the section.
\begin{proposition}\label{prop:log-canonical-dlt-surface-complements-smooth-curves}
	Let $\mathcal{R} \subset [0,1]$ be a finite set of rational numbers.
	Then, there exists a natural number $n$ only depending on $\mathcal{R}$ that satisfies the following.
	Let $X\rightarrow S \rightarrow T$ be a contraction between quasi-projective varieties that satisfies Condition A.
	Then, up to shrinking $T$ around the closed point $t$, we can find
	\[
	\Gamma \sim_{T} -n(K_X+B)
	\]
	such that $(X,B+\Gamma/n)$ is a log canonical pair.
\end{proposition}

\begin{proof}
We only prove the case in which $T$ is a semi-normal curve.
The case in which $T$ is a point is analogous. 
\newline
We split the proof into two steps.
First, we show how to create complements on each component of $X_i$, and then show that they can be glued together to form a global complement.
We refer to the notation introduced in Definition \ref{definition condition A}.
In particular, we consider each case of Condition A separately.
To refer to each case, we use the same numbering as in Condition A.

\textbf{Step 1:}
In this step, we prove the existence of an $n$-complement on each component.
\begin{enumerate}
	\item Assuming $X_1$ is mapped to $s\in S$, then we have $K_{X_1}+B_1\sim_{\qq} 0$.
	Hence, by Theorem \ref{thm:log-canonical-surface-complements}, there exists a bounded $n$, only depending on $\mathcal{R}$, such that $n(K_{X_1}+B_1)\sim 0$.
	In particular, the complement is trivial.
	Also, we note that any complement of $X_1$ will be trivial on any irreducible component of $\lfloor B_1 \rfloor$.

		\item In this case, we apply the canonical bundle formula.
	    Notice that, in this case, the curve $S_1$ is projective.
		Therefore, we will consider global complements.
		We split into two further subcases for gluing.
		Indeed, the construction of the complement is specific to each subcase, and the different constructions need to be distinguished when the gluing is performed in Step 2.
		\begin{enumerate}
			\item First, assume that $\rddown{B_1}$ does not contain any horizontal component mapping onto $S_1$.
			In general, the morphism $f \colon X_1 \rar S_1$ does not have connected fibers.
			Thus, we replace $S_1$ with the image of the Stein factorization of $f$.
			Notice that this process preserves the properties in (2) of Condition A.
			By abusing notation, we still denote the image of the Stein factorization by $S_1$.
			Thus, we may assume that $f \colon X_1 \rar S_1$ is a contraction.
			Notice that this also implies that $S_1$ is normal.
			Then, we may apply Theorem \ref{effective cbf}, which guarantees the existence of a positive integer $q$, depending only on $\mathcal{R}$, such that
			\[
			q(K_{X_1}+B_1)\sim qf^*(K_{S_1}+B_{S_1}+M_{S_1}).
			\]
			Furthermore, again by Theorem \ref{effective cbf}, $qM_{S_1}$ is Cartier and the coefficients of $B_{S_1}$ belong to $\Phi(\mathcal{S})$, where $\mathcal{S} \subset [0,1]$ is a finite set of rational numbers only depending on $\mathcal{R}$.
			Now, since $-(K_{S_1}+B_{S_1}+M_{S_1})$ is ample and $M_{S_1}$ is nef, we conclude that $S_1$ is rational curve.
			Hence, there exists a $R_{S_1}\geq 0$ such that \[
		q(K_{S_1}+B_{S_1}+R_{S_1}+M_{S_1})\sim 0,
			\]
			possibly after replacing $q$ with a bounded multiple.
			We pull back $R_{S_1}$ via $f$ and define $R_1  \coloneqq  f^* R_{S_1}$.
			Then, we get
			\[
			q(K_{X_1}+B_1+R_1)\sim 0.
			\]
			Furthermore, it is clear that  $(X_1,B_1+R_1)$ is log canonical from  the canonical bundle formula.\newline
			
			\item Now, assume that $D$ is a component in $\rddown{B_1}$ mapping onto $S_1$.
			
			Let $C$ be the normalization of $S_1$.
			By assumption, the morphism $D\rightarrow C$ is Galois.
			Notice that both $D$ and $C$ are smooth curves.
			Let $K_D+B_D  \coloneqq  (K_{X_1}+B_1)|_D$, where $B_D\in \Phi(\mathcal{S})$ and $\mathcal{S} \subset [0,1]$ is a finite subset of the rational numbers depending only on $\mathcal{R}$.
			By Condition A and Remark \ref{remark Galois}, we see that there exists $B_C\in \Phi(\mathcal{S})$, such that $K_D+B_D = (K_C+B_C)|_D$.

			Furthermore, we claim that there exists a bounded $q$ such that $q(K_{X_1}+B_1)\sim q(K_C+B_C)|_{X_1}$.
			Here, as usual, the restriction to $X_1$ means the pull-back of the divisor to the corresponding component.
			Indeed, let $X_1\rightarrow E\rightarrow C$ be the Stein factorization.
			Then, by Theorem \ref{effective cbf}, there exists a bounded $q$ such that
			\[
			q(K_{X_1}+B_1)\sim qL= q(K_E+B_E+M_E)|_{X_1},
			\]
			where $B_E$ and $M_E$ are the discriminant and the moduli parts of the canonical bundle formula, respectively, and $L$ is a vertical divisor over $E$.
			Note that $L$ is $\qq$-equivalent to the pull-back of $A$.
			Furthermore, by replacing $\mathcal{S}$ and $q$, we may assume that $B_E\in \Phi(\mathcal{S})$, and $qM_E$ is Cartier.
			Notice that, by applying \cite{Fuj00}*{Proposition 2.1}, we see that $D\rightarrow E$ has either degree 1 or degree 2.
			Then we have that $q(K_D+B_D)\sim qL|_S = q(K_E+B_E+M_E)|_D$.
			Hence, we derive that $q(K_E+B_E+M_E)|_D \sim q(K_C+B_C)|_D$.
			Now, all curves here are rational curves.
			Therefore, we see that by replacing $q$ by $2q$, we have $q(K_E+B_E+M_E) 
		\sim q(K_C+B_C)|_E$.
		This will prove the claim.
		
			We note that such $K_C+B_C$ is determined independently of the choice of $S_1$ by \cite{Kol13}*{Theorem 4.45 (5)}.
			Hence, since $K_C+B_C$ is anti-ample, there exists a $R_C\geq 0$ such that $q(K_C+B_C+R_C)\sim 0$.
			Letting $R_1  \coloneqq  R_C|_{X_1}$, we see that $q(K_{X_1}+B_1+R_1)\sim 0$, possibly after replacing $q$ by a bounded multiple.
			However, we still need to show that $(X_1,B_1+R_1)$ is log canonical.
			By Lemma \ref{lemma Riemann--Hurwitz}, we can show that, possibly by replacing $q$, we can assume that $(D,B_D+R_D)$ is log canonical, where $R_D  \coloneqq  R_C|_D$.
			Then, we are done by Lemma \ref{lemma connectedness}.
		\end{enumerate} 
	\item This case is similar to the previous one.
	As in (2), we split the discussion into two further subcases.
	\begin{enumerate}
		\item First, assume that $\rddown{B_1}$ does not contain any horizontal component mapping onto $S_1$.
		
		Since the irreducible component of $T$ dominated by $S_1$ is not necessarily normal, we cannot apply directly Theorem \ref{thm:log-canonical-surface-complements}.
		By Theorem \ref{effective cbf}, there exists a positive integer $q$, depending only on $\mathcal{R}$, such that $q(K_{X_1}+B_1)\sim q(K_{S_1}+B_{S_1}+M_{S_1})|_{X_1}$, where we possibly replace $S_1$ by its normalization and the Stein factorization of $X_1 \rightarrow S_1$.
		Furthermore, $qM_{S_1}$ is Cartier and the coefficients of $B_{S_1}$ belong to $\Phi(\mathcal{S})$, where $\mathcal{S} \subset [0,1]$ is a finite set depending only on $\mathcal{R}$.
		Now, let $\lbrace s_1, \ldots, s_n \rbrace$ be the preimage of $t\in T$ in $S_1$.
		We can define
		\[
		R_{S_1}  \coloneqq  (1-\mult_{s_1}(B_{S_1}))s_1+\dots+(1-\mult_{s_n}(B_{S_1}))s_n.
		\]
		By Remark \ref{remark points on curves}, we see that $q(K_{S_1}+B_{S_1}+R_{S_1}+M_{S_1})\sim 0$ over a neighborhood of $t$.
		Hence, if we let $R_1  \coloneqq  R_{S_1}|_{X_1}$, we get $q(K_{X_1}+B_1+R_1)\sim 0$ over a neighborhood of $t$.
		That is, $\mathcal{O}_{X_1}(q(K_{X_1}+B_1+R_1))\cong f^*\mathcal{O}_T $ over a neighborhood of $t$.
		Furthermore, as $(S_1,B_1+R_1)$ is log canonical, by \cite{Amb99}*{Proposition 3.4}, then so is $(X_1,B_1+R_1)$.
		
		\item Now, assume that $D$ is a component in $\rddown{B_1}$ mapping onto $S_1$.
		
		Let $C$ be the normalization of $S_1$.
		By assumption, we have that $D\rightarrow C$ is Galois.
		Notice that both $D$ and $C$ are smooth curves.
		Let $K_D+B_D  \coloneqq  (K_{X_1}+B_1)|_D$, where $B_D\in \Phi(\mathcal{S})$ and $\mathcal{S} \subset [0,1]$ is a finite subset of the rational numbers depending only on $\mathcal{R}$.
		By Condition A and Remark \ref{remark Galois}, we see that there exists $B_C\in \Phi(\mathcal{S})$, such that $K_D+B_D = (K_C+B_C)|_D$.
		Arguing as in (2.b), there exists a bounded $q$ such that
		\[
		q(K_{X_1}+B_1)\sim q(K_C+B_C)|_{X_1}.
		\]
		We note that, by the assumptions,
		
		$K_D+B_D$ is $\mathrm{Gal}(D/C)$ invariant, and the pair $(D,B_D)$ is (up to B-birational automorphism) only dependent on the choice of such $S_1$ and independent of the choice of $X_1$.
		As B-birational automorphisms of a curve preserve $B_D$ (and not only its support) the indeterminacy up to B-birational automorphism does not affect the Riemann--Hurwitz formula. Moreover, it follows that $K_C+B_C$ is determined independently of the choice of $X_1$.
		Hence, we can define $R_C  \coloneqq  (1-\mult_{c_1}(B_{C}))c_1+\dots+(1-\mult_{c_n}(B_{C}))c_k \geq 0$, where $\lbrace c_1,\ldots c_k \rbrace$ is the preimage of $t$ on $C$.
		Then, possibly shrinking around $t$, by Remark \ref{remark points on curves}, it follows that $q(K_C+B_C+R_C)\sim 0$, where $R_1  \coloneqq  R_C|_{X_1}$.
		Then, we see that $q(K_{X_1}+B_1+R_1)\sim 0$, possibly after replacing $q$ by a bounded multiple and shrinking around $t$. 
		Furthermore, by Lemma \ref{lemma connectedness} and Lemma \ref{lemma Riemann--Hurwitz}, we see that $(X_1,B_1+R_1)$ is log canonical. 
	\end{enumerate}
	\end{enumerate}
	
\textbf{Step 2:} In this step, we glue the complements constructed in Step 1.

By definition of complements, up to shrinking around $t$, each complement $R_1$ constructed in Step 1 is such that $\mathcal{O}_{X_1}(n(K_{X_1}+B_1+R_1))\sim f^*\mathcal{O}_T$, i.e., each $n(\K X_1. + B_1 + R_1)$ is linearly equivalent to the pull-back of the structure sheaf on $T$.
\newline
Now, we claim that, if $X_1$ and $X_2$ are two irreducible components and $E$ is a component of $X_1 \cap X_2$, the complements $R_1$ and $R_2$ constructed on $X_1$ and $X_2$ respectively agree along $E$.
First, notice that, as $X$ is semi-dlt, $X_1$ and $X_2$ are necessarily distinct irreducible components.
Then, as the question is local over $t \in T$, we may assume that either $E$ is projective and is mapped to $t$ or it dominates an irreducible component of $T$.
Now, if $E$ is mapped to a closed point $s \in S$, Condition A forces any complement to be trivial along $E$, and the claim follows trivially.
On the other hand, if $E$ is mapped onto $S_1$, an irreducible component of $S$, it follows from the construction that $R_1|_E = R_2|_E$, since they are both pull-backs of the same well-defined divisor on $C$, independent of $X_i$, by considering the finite Galois map to $E\rightarrow C$, where $C$ is the normalization of $S_1$.
\newline
Let $(X^\nu,B^\nu+D^\nu)$ be the normalization of $(X,B)$, where $D^\nu$ denotes the conductor, and define $R^\nu$ on $X^\nu$ to be such that $R^\nu |_{X_i} = B_i+R_i$, where $B_i+R_i$ is the complement constructed on $X_i$ in Step 1.
By definition, we have $n(K_{X^\nu}+B^\nu+D^\nu+R^\nu)\sim 0$ over a neighborhood of $t \in T$.
By \cite{Kol13}*{Theorem 5.39}, it follows that $R$ is $\qq$-Cartier, where $R$ denotes the push-forward of $R^\nu$ to $X$. Therefore, $(X,B+R)$ is a semi-log canonical pair.
Now, by Proposition \ref{prop:effective-kollar-gluing}, there is a bounded $m$, depending only on $\mathcal{R}$ (via $n$), such that $m(K_X+B+R)\sim 0$ over $t$. 
\end{proof}

\begin{lemma} \label{lemma connectedness}
Let $X\rightarrow S$ be a contraction from a normal surface to a smooth curve $S$.
Let $(X,B)$ be a dlt pair such that $K_X+B\sim_{\qq,S} 0$.
Let $E$ be an irreducible component of $B^{=1}$ that is horizontal over $S$.
Let $(E,B_E)$ be defined by $K_E+B_E \coloneqq (K_X+B)|_E$.
Let $R_S\geq 0$ be a $\qq$-divisor on $S$, $R  \coloneqq  R_S|_X$ be its pull-back on $X$, and $R_E  \coloneqq  R_S|_E$.
Further assume that $(E,B_E+R_E)$ is log canonical.
Then, $(X,B+R)$ is log canonical.
\end{lemma}
\begin{proof}
This is a local question, hence we can work locally over $s\in S$.
Furthermore, we may assume that $\mult_s R_S >0$, as the conclusion is trivial over $s$ otherwise.
To derive a contradiction, we can assume that $(X,B+R)$ is not log canonical near fiber over $s$, i.e., there exists a vertical non-klt center $Z \subset X_s$ mapping to $s$ that is not a log canonical center.
Since $(E,B_E+R_E)$ is log canonical, by inversion of adjunction, $(X,B+R)$ is log canonical near a neighborhood of $E$.
Now, since $\mult_s (R_S) >0$ and $(E,B_E+R_E)$ is log canonical, $(E,B_E)$ is klt near $X_s \cap E$.
Therefore, $(X,B)$ is plt near $X_s \cap E$.
Thus, by considering $(X,B+aR)$ for some $a<1$ very close to 1, we see that $(X,B+aR)$ is plt near $X_s \cap E$ while $Z$ is still a non-klt center of $(X,B+aR)$.
Therefore, $\mathrm{Nklt} (X,B+aR)$ is disconnected over $s$.
Indeed, $Z$ is disjoint from $E$ over $s$.
Hence, by \cite{FS20}*{Theorem 1.2}, we see that $(X,B+aR)$ is plt near the fiber over $s$, which is a contradiction.
\end{proof}

\begin{remark}
We note that the above lemma also works in the local case near $s\in S$ since the proof is local.
\end{remark}

\begin{lemma} \label{lemma Riemann--Hurwitz}
Let $f \colon X\rightarrow Y$ be a finite Galois morphism between smooth rational curves.
Let $\mathcal{R} \subset [0,1]$ be a finite set of rational numbers.
Assume that $(X,B)$ is log canonical, $B\in \Phi(\mathcal{R})$, $-(K_X+B)$ is ample, and $(X,B)$ is $\mathrm{Gal}(X/Y)$ invariant.
Then, by the Riemann--Hurwitz formula, we may write $K_X+B=f^*(K_Y+B_Y)$, where the coefficients of $B_Y$ belong to $\Phi(\mathcal{R})$.
Furthermore, for any $\qq$-divisor $R\geq 0$ on $Y$ such that $(Y,B_Y+R)$ is log canonical, then so is $(X,B+f^*R)$.
\end{lemma}

\begin{proof}
See \cite{thesis}*{proof of Lemma 4.1.1}.
\end{proof}

\section{Complements for relative log Fano 3-folds}\label{log-canonical-Fano}

In this section, we prove the statement of Theorem~\ref{main-theorem-hyper}
in the relative log Fano case.
In particular, we prove the following statement.

\begin{proposition}\label{prop:log-canonical-Fano}
Let $\mathcal{R} \subset [0,1]$ be a finite set of rational numbers.
There exists a natural number $n$ only depending on $\mathcal{R}$ that satisfies the following.
Let $X\rightarrow T$ be a contraction between normal quasi-projective varieties 
such that the log canonical $3$-fold $(X,B)$ is log Fano over $t\in T$ and the coefficients of $B$ belong to $\Phi(\mathcal{R})$. 
Then, up to shrinking $T$ around $t$, we can find
\[
\Gamma \sim_{T} -n(K_X+B),
\]
such that $(X,B+\Gamma/n)$ is a log canonical pair.
\end{proposition}

In order to prove Proposition \ref{prop:log-canonical-Fano}, we will apply a version of Koll\'ar's injectivity theorem.
For the reader's convenience, we recall its statement, due to Fujino \cite{Fuj17}.

\begin{theorem}[see {\cite{Fuj17}*{Theorem 2.12}}] \label{kollar-injectivity}
	Let $(X,\Gamma)$ be a log smooth pair with ${\rm coeff}(\Gamma)\subset [0,1]$. 
	Let $\phi \colon X\rightarrow T$ be a proper morphism between schemes. 
	Let $\epsilon$ be a positive rational number.
	Let $L$ be a Cartier divisor on $X$.
	Let $S$ be an effective Cartier divisor on $X$ that does not contain any log canonical center of $(X,\Gamma)$.
	Assume that \begin{itemize}
		\item $L\sim_{\qq,T}K_X+\Gamma+N$;
		\item $N$ is a $\qq$-divisor that is semi-ample over $T$; and
		\item $\epsilon N\sim_{\qq,T}S+\overline{S}$, where $\overline{S}$ is an effective $\qq$-Cartier $\qq$-divisor that does not contain any log canonical center of $(X,\Gamma)$ in its support.
	\end{itemize}
Then the natural map $$R^q\phi_*(\mathcal{O}_X(L))\rightarrow R^q\phi_*(\mathcal{O}_X(L+S))$$ is injective for every $q$.
\end{theorem}

In order to prove the main proposition of this subsection, we will need the following lemma.

\begin{lemma}\label{decomposition-of-N'}
    Let $\phi\colon X\rightarrow T$ be a projective morphism of normal quasi-projective varieties.
    Let $(X,B)$ be a log canonical pair, with $-(K_X+B)$ ample over $T$.
    Let $\pi \colon X'\rightarrow X$ be a $\qq$-factorial dlt modification of $(X,B)$, 
    and define $N' \coloneqq -\pi^*(K_X+B)$. Then, we can write
    \[
    N'\sim_{\qq,T} A+D, 
    \]
    where $A$ is ample over $T$, and $D$ is an effective divisor that is semi-ample over $T$ outside ${\rm Ex}(\pi)$.
\end{lemma}

\begin{proof}
    First, we prove that the relative augmented base locus of $N'$ is contained in ${\rm Ex}(\pi)$.
    Let $A$ be an ample divisor on $X'$, $H$ be a very ample divisor on $T$.
    Fix a rational number $0 < \epsilon \ll 1$ such that $\mathbb{B}_+(N'/T) = \mathbb{B}(N' - \epsilon A/T)$.
    Then, we have 
    \[
    \mathbb{B}_{+}(N'/T) = \bigcap_{m\in \mathbb{N}} \bigcap_{n \in \mathbb{N}} {\rm Bs}\left| 
    m(N'-\epsilon A)+n\phi^*H
    \right|.
    \]
    On the other hand, we may choose $m$ and $n$ such that the Cartier divisor
    $mN'+n\phi^*H$ is big and nef on $X'$.
    Moreover, we may further assume that $|mN'+n\phi^*H|$ defines an isomorphism on the complement of ${\rm Ex}(\pi)$.
    By~\cite{BCL14}*{Theorem A}, we conclude that
    \[
    \mathbb{B}_{+}(mN'+n\phi^*H) \subset {\rm Ex}(\pi).
    \]
    This latter inclusion implies that for $\epsilon$ small enough, we have
    \[
    {\rm Bs}
    \left|
    mN'+n\phi^* H -m\epsilon A 
    \right|
    \subset {\rm Ex}(\pi).
    \]
    Thus, we conclude that $\mathbb{B}_{+}(N'/T)\subset {\rm Ex}(\pi)$.
    By the above inclusion, we conclude that we may write
    \[
    N'\sim_{\qq,T} A+D,
    \]
    where $A$ is ample over $T$, and the base locus of $D$ is contained in ${\rm Ex}(\pi)$.
    We conclude the claim by replacing $D$ with some general element in its relative $\qq$-linear system.
\end{proof}

The following lemma will be used to lift complements from the semi-dlt locus of a dlt pair.

\begin{lemma}\label{lem:lifting}
Let $\mathcal{R}\subset [0,1]$ be a finite set of rational numbers.
Let $X\rightarrow T$ be a contraction
with $\dim X \leq 3$
and $t\in T$.
Let $(X',B')$ be a dlt pair so that
$-(K_{X'}+B')$ is big and semiample over $T$ and 
the coefficients of $B'$ belong to $\Phi(\mathcal{R})$.
Let $X$ be the ample model of $-(K_{X'}+B')$ over $T$ and $X''\rightarrow X'$ a log resolution of $(X',B')$.
Assume that the following conditions hold: 
\begin{enumerate}
    \item  $X''\rightarrow X$ admits an exceptional and relatively anti-ample divisor; and
    \item $X''\rightarrow X'$ extracts no log canonical place of $(X',B')$.
\end{enumerate}
Let $(S',B_{S'})$ be the semi-dlt pair obtained by adjunction of $(X',B')$ to $S'=\lfloor B'\rfloor$.
Assume that $(S',B_{S'})$ is $n$-complemented over $t\in T$ for some $n$ divisible by $I(\mathcal{R})$.
Then, $(X',B')$ is $n$-complemented over $t\in T$.
\end{lemma}

\begin{proof}
The strategy follows the proof of \cite{Bir16a}*{Proposition 8.1}.
We proceed in several steps.

{\bf Step 1:} In this step, we introduce some notation.
\newline
Let $(X'',B'')$ denote the trace of $(X',B')$ on $X''$ 
and let $B$ be the push-forward of $B'$ to $X$.
By assumption, $N \coloneqq -(\K X. + B)$ is ample over $T$.
Therefore, $N'' \coloneqq -(K_{X''}+B'')$ is nef and big over $T$.
Define $W'' \coloneqq \lfloor {B''}^{\geq 0}\rfloor$, $\Delta'' \coloneqq B''-W''$, and $S''\coloneqq W'' - \pi^{-1}_* \lfloor B \rfloor$.
Observe $S''$ is a Weil divisor on $X''$, and therefore it is Cartier.
For any divisor $\Omega ''$ on $X''$, let $\Omega'$ and $\Omega$ denote the push-forwards on $X'$ and $X$, respectively.

{\bf Step 2:} In this step, we introduce some line bundles on $X''$ that are suitable for the use of vanishing theorems.
\newline
On $X''$, consider the Cartier divisor 
\[
L'' \coloneqq -nK_{X''}-nW''-\lfloor (n+1)\Delta'' \rfloor.
\]
The choice is motivated as follows: our goal is to lift the complement $B \subs S'. ^+$ from $S'$ to $X'$.
Since $X'$ may be singular, we need to work on the smooth model $X''$ to use the appropriate vanishing theorems.
Observe that we may write
\begin{align}
\nonumber    L'' = & -nK_{X''}-nW''-\lfloor (n+1)\Delta'' \rfloor \\
\nonumber     =& K_{X''}+W''+ (n+1)\Delta''- \lfloor (n+1)\Delta'' \rfloor  -(n+1)K_{X''}-(n+1)W''-(n+1)\Delta'' \\
\nonumber =& K_{X''}+W''+(n+1)\Delta''- \lfloor (n+1)\Delta'' \rfloor+(n+1)N''\\
\nonumber =& K_{X''}+B''+n\Delta''- \lfloor (n+1)\Delta'' \rfloor+(n+1)N''\\
\nonumber =& n \Delta'' - \lfloor (n+1)\Delta'' \rfloor + nN''.
\end{align}
Hence, we can write
\[
L''-S''=K_{X''}+(W''-S'')+(n+1)\Delta''- \lfloor (n+1)\Delta'' \rfloor+(n+1)N''.
\]

{\bf Step 3:} In this step, we introduce divisors $\Phi''$ and $\Lambda''$ on $X''$ and study their properties.
\newline
Let $\Phi''$ be the unique integral divisor on $X''$ such that
\[
\Lambda'' \coloneqq (W''-S'')+(n+1)\Delta''- \lfloor (n+1)\Delta'' \rfloor + \Phi''
\]
is a boundary, $(X'',\Lambda'')$ is dlt, and $\lfloor \Lambda'' \rfloor = W''-S''$.
It follows that $\Phi''$ is supported on $\mathrm{Ex}(X'' \rar X')$ and shares no components with $W''$.

{\bf Step 4:} In this step, we apply Theorem \ref{kollar-injectivity} to $L''-S''+\Phi''$.
\newline
Recall that $N''$ is semi-ample over $T$, being the pull-back of an ample divisor over $T$.
By assumption, we may find an effective divisor $F''$ on $X''$ that is exceptional and anti-ample for $X'' \rar X$.
Hence, $N''-\epsilon F''$ is ample over $T$ for $0 < \epsilon \ll 1$.
Observe that $(X'',\Lambda'')$ is a log smooth pair.
The pair $(X'',\Supp(\Lambda''+S''+F''))$ is a log smooth pair.
Furthermore, we have $\lfloor \Lambda'' \rfloor=W''-S''$.
Therefore, $\Supp(S''+F'')$ contains no log canonical center of $(X'',\Lambda'')$.
Since we have
\[
L''-S''+\Phi''=K_{X''}+\Lambda''+ (n+1) N'',
\]
in order to apply Theorem~\ref{kollar-injectivity}, we are left with checking that the third condition of the statement holds.
Fix $0 < \delta \ll 1$, 
such that $N''-\epsilon F'' -\delta S''$ is ample over $T$.
Then, we may write $N''-\epsilon F'' -\delta S'' \sim_{\qq,T} G''\geq 0$, 
where $G''$ contains no log canonical center of $(X'',\Lambda'')$.
Hence, we have a $\qq$-linear relation
\[
N''\sim_{\qq,T} G''+\epsilon F'' + \delta S'',
\]
where $G''+\epsilon F''$ is an effective divisor
that does not contain any log canonical center of $(X'',\Lambda'')$.
Thus, by Theorem~\ref{kollar-injectivity}, we deduce that there is an injection
\[
R^1\psi_*\mathcal{O}_{X''}(L''-S''+\Phi'') \rightarrow R^1\psi_*\mathcal{O}_{X''}(L''+\Phi''),
\]
as desired.
Here, $\psi$ denotes the morphism $X'' \rar T$.
Then we have a surjection
\[
\psi_*\mathcal{O}_{X''}(L''+\Phi'')\rightarrow \psi_*\mathcal{O}_{X''}((L''+\Phi'')|_{S''}).
\]
Shrinking $T$ around $t\in T$, the above surjection identifies with
\begin{equation} \label{eqtn surjection 1}
H^0(\mathcal{O}_{X''}(L''+\Phi'')) \rightarrow H^0(\mathcal{O}_{S''}((L''+\Phi'')|_{S''}).
\end{equation}

{\bf Step 5:} In this step, we introduce some divisors on $S''$.
\newline
By assumption, we have an $n$-complement $B \subs S'.^+= B \subs S'. + R \subs S'.$ for $(S',B \subs S'.)$ over $t \in T$.
Notice that $R \subs S'.$ is a $\qq$-Cartier divisor not containing any irreducible component of the conductor of $(S',B \subs S'.)$.
We have a birational morphism of possibly reducible algebraic varieties $S''\rightarrow S'$.
Furthermore, by the second assumption on $X''\rightarrow X'$, every irreducible component of $S''$ maps birationally onto its image in $S'$.
Therefore, $R_{S'}$ does not contain the image of any component of $S''$ on $S'$, and 
its pull-back $R_{S''}$ on $S''$ is well-defined.
Now, we have
\[
n (\K S''. + B \subs S''. + R \subs S''.) \sim_T 0.
\]
By construction, we have $B \subs S''. = (B''-S'')| \subs S''.$, and the restriction preserves the coefficients, as we are in a log smooth setting.
Removing the contribution of $(W''-S'')| \subs S''.$, which is integral, we realize that $n(\Delta \subs S''. + R \subs S''.)$ is integral, where we have $\Delta \subs S''. \coloneqq \Delta'' | \subs S''.$.
We define
\[
G_{S''} \coloneqq nR_{S''}+n\Delta_{S''}-\lfloor (n+1) \Delta_{S''}\rfloor + \Phi \subs S''.,
\]
where we have $\Phi \subs S''. \coloneqq \Phi''| \subs S''.$.
By definition, $G_{S''}$ is an integral divisor, and $nR \subs S''. + \Phi \subs S''.$ is effective.
Therefore, to show that $G \subs S''.$ is effective, it suffices to show that the coefficients of $n\Delta_{S''}-\lfloor (n+1) \Delta_{S''}\rfloor$ are strictly greater than $-1$.
Then, as rounding and restricting commutes in a log smooth setup, we may write
\[
n\Delta_{S''}-\lfloor (n+1) \Delta_{S''}\rfloor = ((n+1)\Delta'' - \lfloor (n+1) \Delta'' \rfloor - \Delta'')| \subs S''.,
\]
where the summand $(n+1)\Delta'' - \lfloor (n+1) \Delta'' \rfloor$ is effective.
As the coefficients of $\Delta''$ are strictly less than $1$, it follows that the coefficients of $- \Delta''$ are strictly greater than $-1$.
In particular, $G \subs S''.$ is effective.

{\bf Step 6:} In this step, we lift $G \subs S''.$ to $X''$.
\newline
We have $N \subs S''. \coloneqq N''| \subs S''. = -(\K X''. + B'')| \subs S''. = -(\K S''. + B \subs S''.)$.
Then, it follows that $nR_{S''}\sim_T nN_{S''}$.
Up to further shrinking $T$ around $t$, in the following we may drop $T$ in the linear equivalence.
In particular, we may write $nR_{S''}\sim nN_{S''}$.
By the previous considerations we have
\[
0\leq G_{S''} \sim nN_{S''}+n\Delta_{S''}-\lfloor (n+1) \Delta_{S''}\rfloor + \Phi \subs S''..
\]
Then, observe that
\begin{align}
\nonumber    L \subs S''. \coloneqq L'' | \subs S''. = &  (K_{X''}+B''+n\Delta''- \lfloor (n+1)\Delta'' \rfloor+(n+1)N'')| \subs S''.\\
\nonumber = & \K S''. + B \subs S''. + n \Delta \subs S''. - \lfloor (n+1) \Delta \subs S''. \rfloor + (n+1) N \subs S''.\\
\nonumber = & n \Delta \subs S''. - \lfloor (n+1) \Delta \subs S''. \rfloor + n N \subs S''..
\end{align}
Hence, we conclude that
\[
0\leq G_{S''}\sim L_{S''} + \Phi \subs S''..
\]
Thus, by the surjectivity of \eqref{eqtn surjection 1}, there exists $0 \leq G'' \sim L''+\Phi''$ on $X''$ such that $G''|_{S''} = G_{S''}$.

{\bf Step 7:} In this step, we study $G'$, the push-forward of $G''$ to $X'$, and we introduce $(B')^+$, the candidate to be a complement for $(X',B')$.
\newline
By definition of $L''$, we get
\[
0\leq G'' \sim -nK_{X''}-nW''-\lfloor (n+1)\Delta'' \rfloor + \Phi''.
\]
Let $G'$ be the push-forward of $G''$ to $X'$.
Then, as $\Phi''$ is exceptional for $X'' \rar X'$, we have
\begin{equation} \label{eqtn lin eq}
0\leq G' \sim -nK_{X'}-nW'-\lfloor (n+1)\Delta' \rfloor.
\end{equation}
Then, we can define
\[
nR' \coloneqq G' + \lfloor (n+1)\Delta' \rfloor - n \Delta ' \sim -n(\K X'. + B'),
\]
where the linear equivalence follows from \eqref{eqtn lin eq}.
By Proposition \ref{prop arithmetic} and the fact that the coefficients of $\Delta'$ are in $\Phi(\mathcal R)$, it follows that $nR'$ is effective.
Then, we define $(B')^+ \coloneqq B' + R'$.
By construction, we have that $n(\K X'. + B') \sim 0$.

{\bf Step 8:} In this step, we show that $(B')^+$ is an $n$-complement for $(X',B')$ over $t \in T$.
\newline
To conclude, it suffices to show that $(X',(B')^+)$ is log canonical.
First, we show that $R'|_{S'}=R_{S'}$.
Let
\[
nR'' \coloneqq G'' - \Phi'' + \lfloor (n+1)\Delta'' \rfloor - n\Delta'' \sim
L''  + \lfloor (n+1)\Delta'' \rfloor - n\Delta'' = nN'' \sim_{\qq,X} 0.
\]
By construction, the push-forward of $R''$ onto $X'$ is $R'$.
Furthermore, $R''\sim_{\qq,X'}0$ holds, as we have $R''\sim_{\qq,X}0$.
As $R''$ is relatively trivial for the birational morphism $X'' \rar X'$, it follows that it agrees with the pull-back of its push-forward.
In particular, $R''$ is the pull-back of $R'$.
Observe that $R''|_{S''}=R_{S''}$.
Hence, we have $R'|_{S'}=R_{S'}$.
This implies the equality
\[
K_{S'}+B_{S'}+R_{S'}=(K_{X'}+B'+R')|_{S'}=(K_{X'}+(B')^+)|_{S'}.
\]
Therefore, by inversion of adjunction \cite{thesis}*{Lemma 2.3.1}, the pair $(X',(B')^+)$ is log canonical in a neighborhood of $S'$.
If $(X',B'+R')$ is not log canonical in a neighborhood of $\phi \sups -1.(t)$, then we can write $\mathrm{Nklt}(X',B'+R') = \Supp (\lfloor B' \rfloor) \cup Z'_1 \cup Z'_2$, where $Z'_1$ is a union of log canonical centers, and $\mathrm{Nlc}(X',B'+R')=Z'_2$.
As $(X',(B')^+)$ is log canonical in a neighborhood of $S'$, we have $\Supp (S') \cap Z_2' = \emptyset$.
Then, fix $0 < \alpha \ll 1$, such that $\mathrm{Nklt}(X',B'+(1-\alpha)R')= \Supp(\lfloor B' \rfloor) \cup Z'_2$, and $\mathrm{Nlc}(X',B'+(1-\alpha)R')=Z_2'$.
Furthermore, this choice of $\alpha$ guarantees that the log canonical centers of $(X',B')$ are the same as the ones of $(X',B' + (1-\alpha)R')$.
Now, let $A'$ and $D'$ be as in Lemma~\ref{decomposition-of-N'}.
In particular, we have $-(\K X'. + B') \sim \subs \qq,T. A'+ D'$, $A'$ is ample over $T$, and $D'$ is semi-ample over $T$ outside of $\mathrm{Ex}(X' \rar X)$.
Consider the following facts:
\begin{itemize}
    \item $D'$ is semi-ample over $T$ outside of $\mathrm{Ex}(X' \rar X)$;
    \item every log canonical center of $(X',B')$ that is contained in $\mathrm{Ex}(X' \rar X)$ is contained in $S'$; and
    \item the log canonical centers of $(X',B')$ are the same as the log canonical centers of $(X',B' + (1-\alpha)R')$.
\end{itemize}
Fix $0 < \beta \ll \alpha$.
Then, we have that adding $\beta D'$ to $(X',B' + (1-\alpha)R')$ does not create new log canonical centers, but it may create deeper singularities along $S'$ and $Z_2'$.
In particular, we have $\mathrm{Nklt}(X',B' + (1-\alpha)R' + \beta D') = \Supp (B') \cup Z_2'$, and $\mathrm{Nlc}(X',B' + (1-\alpha)R' + \beta D')=Z_2 ' \cup Z_3'$, where $Z_3' \subset \Supp (S')$.
Then, we have the following linear equivalences
\begin{align}
\nonumber    \K X'. + B' + (1-\alpha) R' + \beta D' \sim & \subs \qq,T. \alpha(\K X'. + B') + \beta D' \\
\nonumber     \sim& \subs \qq,T. -\alpha(A' + D') + \beta D' \\
\nonumber \sim & \subs \qq,T.  -(\alpha - \beta)(A' + D') - \beta A'.
\end{align}
Then, fix $0 < \gamma \ll \beta$, such that we have
\begin{itemize}
    \item $\beta A' + \gamma (B' - S')$ is ample over $T$;
    \item $\mathrm{Nlc}(X',B'+(1-\alpha)R' + \beta D' - \gamma (B'-S'))= \mathrm{Nlc}(X',B'+(1-\alpha)R' + \beta D')$; and
    \item the union of the log canonical centers of $(X',B'+(1-\alpha)R' + \beta D' - \gamma (B'-S'))$ is $\Supp(S')$.
\end{itemize}
In particular, it follows that $\mathrm{Nklt}(X',B'+(1-\alpha)R' + \beta D' - \gamma (B'-S'))=\Supp(S') \cup Z_2'$, which is disconnected along $\phi \sups -1.(t)$.
On the other hand, we may write
\[
\K X'. + B' + (1-\alpha) R' + \beta D' - \gamma (B'-S') \sim \subs \qq,T. -(\alpha - \beta)(A'+D') - (\beta A' + \gamma (B'-S')),
\]
where $A'+D'$ is big and semi-ample over $T$, and $\beta A' + \gamma (B'-S')$ is ample over $T$.
Therefore, by the connectedness principle \cite{Bir16a}*{Lemma 2.14}, $\mathrm{Nklt}(X',B'+(1-\alpha)R' + \beta D' - \gamma (B'-S'))=\Supp(S') \cup Z_2'$ is disconnected along $\phi \sups -1.(t)$.
This provides the required contradiction.
In particular, $(X',B'+R')$ is log canonical along $\phi \sups -1.(t)$.
This concludes the proof.
\end{proof}

\begin{lemma}\label{lem:complement-sdlt-surface}
Let $\mathcal{R}\subset [0,1]$ be a finite set of rational numbers.
There exists a number $n$ only depending on $\mathcal{R}$ satisfying the following.
Let $\phi\colon X'\rightarrow T$ be a contraction of normal quasi-projective varieties and $t\in T$ be a point. 
Assume that the pair $(X',B')$ satisfies the following conditions:
\begin{enumerate}
    \item $(X',B')$ is a dlt 3-fold;
    \item $-(K_{X'}+B')$ is big and semiample; and
    \item $(X',B')$ is $\qq$-complemented over a neighborhood of $t\in T$.
\end{enumerate}
Let $\pi \colon X'\rightarrow X$ be the ample model of $-(K_{X'}+B')$ over $T$.
Set $S'=\lfloor B'\rfloor$, and let $S'_1$ be the sum of the components of $S'$ that are contracted on $X$.
Assume that $S'_1$ contains all the log canonical centers of $(X',B')$ that are contained in ${\rm Ex}(X'\rightarrow X)$.
Let $(S',B_{S'})$ be the semi-dlt pair obtained by adjunction of $K_{X'}+B'$ to $S'$.
Then, up to shrinking $T$ around $t$, we can find 
\[
\Gamma_{S'} \sim_T -n(K_{S'}+B_{S'}),
\] 
such that $(S',B_{S'}+\Gamma_{S'}/n)$ is a semi-log canonical pair.
\end{lemma}

\begin{proof}
If $(X',B')$ is klt, then the claim follows from \cite{Bir16a}*{Theorem 1.8}.
Thus, we may assume that $(X',B')$ is strictly log canonical over $t \in T$.
Furthermore, by Remark \ref{rmk generic point} and \cite{thesis}*{Theorem 1.8.7}, we may assume that $\dim T >0$ and that $t$ is not the generic point of $T$.
We proceed in several steps.

{\bf Step 1:} In this step, we fix our notation.

We let $\phi$ denote the morphism $\phi \colon X' \rar T$ and we let $S_2'$ denote the sum of all the components of $S'$ that are not contained in $\mathrm{Ex}(X' \rar X)$.
Lastly, we denote by $S_1$ and $T_1$ the images of $S_1'$ in $X$ and $T$, respectively.
\newline
Now, let $B'+B^+$ denote a $\qq$-complement of $(X',B')$ over $t \in T$.
Then, $(X',B'+B^+)$ induces minimal quasi-log canonical structures on $X$ and $T$, see \cite{KK10}*{Definition 5.3}.
Then, as $S_1$ and $T_1$ are closures of unions of $\pi$-qlc strata and $\phi-$qlc strata, respectively (see \cite{KK10}*{Definition 5.4}), it follows from \cite{KK10}*{Proposition 5.7} that $S_1$ and $T_1$ are semi-normal.

{\bf Step 2:} In this step, we show that $S'_1\rightarrow T_1$ is a contraction.

From the exact sequence
\[
0\rightarrow \mathcal{O}_{X'}(-S_1')\rightarrow \mathcal{O}_{X'} \rightarrow \mathcal{O}_{S_1'} \rar 0,
\]
we obtain the sequence
\[
\phi_* \mathcal{O}_{X'} \rightarrow \phi_* \mathcal{O}_{S_1'} \rightarrow R^1\phi_*\mathcal{O}_{X'}(-S_1'),
\]
which is exact in the middle.
Set $N' = - (\K X'. + B')$.
Then, by Lemma~\ref{decomposition-of-N'}, we can write
\begin{align}
\nonumber -S_1' &= K_{X'}+B'-S_1'+N'\\
\nonumber &\sim_{\qq,T} K_{X'}+B'-S_1'+(1-\epsilon)N'+\epsilon A' +\epsilon D', 
\end{align}
where $A'$ is ample over $T$, and $D'$ is an effective divisor that is semi-ample over $T$ outside of ${\rm Ex}(X' \rar X)$.
By assumption, we have that all the log canonical centers of $(X',B')$ that are contained in $\mathrm{Ex}(X' \rar X)$ are contained in $S'_1$.
Therefore, if we pick $0 < \epsilon \ll 1$ and $D'$ generically in its relative $\qq$-linear series, by Lemma \ref{decomposition-of-N'}, the pair $(X',B'-S_1'+\epsilon D')$ is dlt.
Moreover, since $\epsilon A'$ is ample over $T$, we may pick $\delta$ small enough such that
$(X',B'-S_1'-\delta\lfloor \pi^{-1}_*B \rfloor +\epsilon D')$ is klt 
and $\epsilon A' +\delta \lfloor \pi^{-1}_*B \rfloor$ is ample over $T$.
Hence, we may write
\[
-S_1' \sim_{\qq,T} (K_{X'}+B'-S_1'-\delta\lfloor \pi^{-1}_* B \rfloor +\epsilon D') + ((1-\epsilon)N'+\epsilon A' +\delta \lfloor \pi^{-1}_*B\rfloor),
\]
where the first summand is the log canonical divisor of a klt pair, and the second one is a divisor that is ample over $T$.
By the relative version of the Kawamata--Viehweg vanishing theorem, we conclude that
\[
R^1\phi_*\mathcal{O}_{X'}(-S'_1)=0.
\]
Thus, $\phi_*\mathcal{O}_{X'}\rightarrow \phi_*\mathcal{O}_{S'_1}$ is surjective.
Let $S'_1\rightarrow S'_0\rightarrow T$ be the Stein factorization of $S'_1\rightarrow T$, and write $\phi_0 \colon S_0' \rar T$ for the induced morphism.
Then, we have that $\mathcal{O}_T= \phi_*\mathcal{O}_{X'}\rightarrow \phi_*\mathcal{O}_{S'_1}={\phi_0}_*\mathcal{O}_{S'_0}$ is surjective.
The morphism $\mathcal{O}_T\rightarrow {\phi_0}_*\mathcal{O}_{S'_0}$ factors as 
$\mathcal{O}_T \rightarrow \mathcal{O}_{T'}\rightarrow {\phi_0}_*\mathcal{O}_{S'_0}$.
Hence, we conclude that $\mathcal{O}_{T_1}\rightarrow {\phi_0}_*\mathcal{O}_{S'_0}$
is surjective.
Since $S'_0\rightarrow T_1$ is finite, then $\mathcal{O}_{T_1}\rightarrow {\phi_0}_*\mathcal{O}_{S'_0}$ is indeed an isomorphism. 
Hence, $S'_0\rightarrow T_1$ is an isomorphism, and we conclude that $S_1'\rightarrow T_1$ is a contraction.

{\bf Step 3:} In this step, we construct a complement on $S_1'$.

Let $(S'_1,B_{S'_1})$ be the semi-dlt pair obtained by adjunction of $(X',B')$ to $S'_1$.
Recall that both $S_1$ and $T_1$ are semi-normal curves.
Then, by Remark~\ref{remark Galois}, $(S'_1,B_{S'_1})$ satisfies the conditions of Proposition~\ref{prop:log-canonical-dlt-surface-complements-smooth-curves}.
Indeed, by \cite{thesis}*{\S2.3.1}, there exists a finite set of rational numbers $\mathcal S \subset [0,1]$, only depending on $\mathcal R$, such that the coefficients of $B \subs S'.$ belong to $\Phi (\mathcal S)$.
Furthermore, as $(X',B')$ is $\qq$-complemented over $t \in T$, then so is $(S'_1, B \subs S'_1.)$ over $t \in T'$.
Thus, the pair $(S'_1,B_{S'_1})$ and the morphisms
$S'_1\rightarrow S_1\rightarrow T_1$ satisfy Condition A.
We conclude that there exists 
\[
\Gamma_{S'_1} \sim_T -n(K_{S'_1}+ B_{S'_1}) 
\] 
such that $(S'_1,B_{S'_1}+\Gamma_{S'_1}/n)$ is a semi-log canonical pair.

{\bf Step 4:} In this step, we lift the complement constructed in Step 3 to a complement on $S'$.

Let $(S'_2,B_{S'_2})$ be the semi-dlt pair obtained by adjunction of $(X',B')$ to $S'_2$.
Note that $-(K_{S'_2}+B_{S'_2})$ is big and nef over $T$, since $-(K_X + B)$ is ample over $T$ and every component of $S_2'$ is mapped birationally to $X$.
Let $C$ be the semi-dlt curve that is the union of all the one-dimensional strata of $\lfloor B'\rfloor$.
We denote by $B_C$ the boundary obtained by adjunction of $(X',B')$ to $C$.
Set $C_i=C\cap S'_i$
and $B_{C_i}\coloneqq  B_C|_{C_i}$ for each $i\in \{1,2\}$.
We let 
\[
\Gamma_{C_1}\coloneqq \Gamma_{S'_1}|_{C_1}.
\]
The pair $(C_1,B_{C_1}+\Gamma_{C_1}/n)$ is semi-log canonical over $t\in T$.
Then, we can extend $B_{C_1}+\Gamma_{C_1}/n$ to an $n$-complement $B_C$ of $C$ over $t\in T$.
Indeed, we may use \cite{thesis}*{Proposition 2.7.2} to lift the complement to the components of $C_2$ that are contracted to $t \in T$.
If we consider a horizontal component of $C_2$, we may first shrink $T$ around $t$ so that the only non-trivial components of $B_{C_2}$ are over $t \in T$.
Then, by Remark \ref{remark points on curves}, we may complement $B_{C_2}$ along these points by increasing it to its round up.
We denote this $n$-complement by
\[
(C,B_C + \Gamma_C/n).
\] 
By Lemma~\ref{lem:lifting}, for each component of $S'_2$, we
can lift the $n$-complement on $C\cap S'_2$ to an $n$-complement for $S'_2$.
Thus, for each component $P$ of $S'$, we have produced an $n$-complement for
$(K_{X'}+B')|_P$ so that they agree on the intersections.
Hence, up to shrinking $T$ around $t$,
we can find 
\[
\Gamma_{S'} \sim_T -n(K_{S'}+B_{S'}),
\] 
such that $(S',B_{S'}+\Gamma_{S'}/n)$ is a semi log canonical pair.
This finishes the proof.
\end{proof}

\begin{proof}[{Proof of Proposition \ref{prop:log-canonical-Fano}}]

By lemma~\ref{lemma dlt model}
and Remark~\ref{remark dlt model}, we can find a log resolution $X''\rightarrow X$ of $(X,B)$
and a dlt modification
$X'\rightarrow X$ of $(X,B)$ that satisfy the following conditions: 
\begin{enumerate}
    \item the rational map $X''\dashrightarrow X'$ is a morphism;
    \item $X''\rightarrow X$ admits an exceptional and relatively anti-ample divisor;
    \item $X''\rightarrow X'$ extracts no log canonical place of $(X',B')$; and 
    \item the sum of the components of $\lfloor B'\rfloor$ that are contracted in $X$ contains all the log canonical centers of $(X',B')$ contained in ${\rm Ex}(X'\rightarrow X)$.
\end{enumerate}
Note that $(X',B')$ is $\qq$-complemented over $t \in T$. 
Let $(S',B_{S'})$ be the pair obtained by performing adjunction for $(X',B')$ to $S'=\lfloor B'\rfloor$.
By Lemma~\ref{lem:complement-sdlt-surface}, we can find an $n$-complement for $(S',B_{S'})$ over $t\in T$.
Here, $n$ only depends on $\mathcal{R}$.
We may assume that $I(\mathcal{R})$ divides $n$.
Then, by Lemma~\ref{lem:lifting}, we conclude that $(X',B')$ admits an $n$-complement over $t\in T$.
\end{proof}

\section{Proof of the theorems}\label{proof-theorems}

In this section, we prove the main theorem of this article.
First, we aim to prove the following version of the main theorem.

\begin{theorem}\label{main-theorem-hyper}
Let $\mathcal{R} \subset [0,1]$ be a finite set of rational numbers.
There exists a natural number $n$ only depending on $\mathcal{R}$ that satisfies the following.
Let $X\rightarrow T$ be a contraction between normal quasi-projective varieties
such that the log canonical $3$-fold $(X,B)$ is $\qq$-complemented over $t\in T$, and the coefficients of $B$ belong to $\Phi(\mathcal{R})$. 
Then, up to shrinking $T$ around $t$, we can find
\[
\Gamma \sim_{T} -n(K_X+B)
\]
such that $(X,B+\Gamma/n)$ is a log canonical pair.
\end{theorem}

\begin{proof}
By Lemma \ref{rmk generic point} and Lemma \ref{rmk lcc}, we may assume that there is a log canonical place of $(X,B)$ whose center on $T$ is $t$.
Let $B+B'$ be the $\qq$-complement of the log canonical $3$-fold $(X,B)$ around the point $t\in T$.
Over a neighborhood of $t\in T$, we have that $B' \sim_{\qq,T} -(K_X+B)$ and $(X,B+B')$ is log canonical. 
Let $\pi \colon  Y\rightarrow X$ be a $\qq$-factorial dlt modification of $(X,B+B')$ over $T$.
Write
\[
\pi^*(K_X+B+B')=K_Y+B_Y+B'_Y+E, 
\]
where $E$ is the reduced divisor that contains all the log canonical centers of $(X,B+B')$.
We also set
\[
B_Y \coloneqq \pi^{-1}_{*}(B) -\pi^{-1}_{*}(B) \wedge E,
\]
and
\[
B'_Y \coloneqq \pi^{-1}_{*}(B') -\pi^{-1}_{*}(B')\wedge E.
\]
Observe that $B_Y$ (resp. $B'_Y$) is the strict transform of $B$ (resp. $B'$)
with all the prime components contained in the support of $E$ removed.
By parts (1) and (2) of Lemma~\ref{reduction complements},
we know it suffices to produce an $n$-complement for
$(Y,B_Y+E)$.
Observe that $(Y,B_Y+E)$ is $\qq$-complemented over $t\in T$,
hence the assumptions of the theorem are preserved.
Observe that all the log canonical places of $(Y,B_Y+B'_Y+E)$ are contained in the support of $E$.
Therefore, for $\epsilon>0$ small enough, the pair 
\[
(Y,B_Y+(1+\epsilon)B'_Y+E)
\]
is a $\qq$-factorial dlt pair which is pseudo-effective over $T$.
We run a minimal model program for $K_Y+B_Y+(1+\epsilon)B'_Y+E$ 
over $T$, which terminates with a good minimal model
$(Z,B_Z+(1+\epsilon)B'_Z+E_Z)$ over $T$ (see, e.g.,~\cite{Fuj00}). 
Here, $B_Z$ (resp. $B'_Z$ and $E_Z)$ denotes the strict transform of $B_Y$ (resp. $B'_Y$ and $E$).
Observe that this minimal model program is also a minimal model program for 
\[
-\epsilon (K_Y+B_Y+E) \sim_{\qq,T} \epsilon B'_Y \sim_{\qq,T} K_Y+B_Y+(1+\epsilon)B'_Y+E.
\]
By part (3) of Lemma~\ref{reduction complements}, any $n$-complement over $t\in T$ for $(Z,B_Z+E_Z)$ pulls back to a
$n$-complement over $t\in T$ for $(Y,B_Y+E_Y)$.
Therefore, it suffices to produce an $n$-complement for the log canonical pair
$(Z,B_Z+E_Z)$, which is $\qq$-complemented over $t\in T$.
Since $Z$ is a good minimal model, we have that $-(K_Z+B_Z+E_Z)$ is semi-ample.
Hence, it induces a morphism $\phi \colon Z\rightarrow Z_0$ over $T$.
We obtain a diagram as follows
\[
 \xymatrix@C=50pt{
 Y\ar[d]_-{\pi} \ar[rdd]\ar@{-->}[rr] & & Z\ar[ldd]\ar[d]^-{\phi} \\
 X\ar[rd] & & Z_0\ar[ld]  \\
 & T & 
 }
\]
We will analyze the cases depending on the dimension of
$Z_0$.
If $\dim Z_0=3$, then the map $Z\rightarrow Z_0$ is a
$(K_Z+B_Z+E_Z)$-trivial birational map.
It suffices to find an $n$-complement 
for $K_{Z_0}+B_{Z_0}+E_{Z_0}$,
where $n$ only depends on $\mathcal{R}$.
Moreover, $-(K_{Z_0}+B_{Z_0}+E_{Z_0})$ is ample over 
$T$.
The existence of such $n$-complement follows from Proposition \ref{prop:log-canonical-Fano}.
If $\dim Z_0=2$, then the existence of an $n$-complement
for $K_Z+B_Z+E_Z$ over $T$ follows from Proposition~\ref{prop:lifting-complements-surfaces}.
If $\dim Z_0=1$, then the existence of an $n$-complement 
for $K_Z+B_Z+E_Z$ over $T$ follows from 
Proposition~\ref{prop:lifting-complements-curves}.
Finally, if $\dim Z_0=0$, then we have $\dim T=0$,
so we are in the projective case.
In this case, the existence of an $n$-complement
for $K_Z+B_Z+E_Z$ follows from~\cite{thesis}*{Theorem 1.8.7}.
Observe that in the above three cases $n$ only depends on $\mathcal{R}$.
This finishes the proof of the existence of $n$-complements
with $n$ only depending on $\mathcal{R}$.
\end{proof}

In order to prove Theorem~\ref{main-theorem}, 
we just need to perform a perturbation of the coefficients set in order to reduce to the hyper-standard case and apply Theorem~\ref{main-theorem-hyper}.
This statement is proved in~\cite{FM18}*{Lemma 3.2}
for Fano-type varieties.
The proof, in this case, is essentially the same.
We recall some notation.

\begin{notation} \label{notation}
Let $\Lambda \subset \qq \cap (0,1]$ be a set with $\overline{\Lambda}\subset \qq$ and satisfying the descending chain condition.
Given a natural number $m$, 
we will define an $m$-truncation of the elements of $\Lambda$ as follows.
Consider the partition
\[
\mathcal{P}_m \coloneqq \left\{ \left(0, \frac{1}{m}\right], \left( \frac{1}{m}, \frac{2}{m} \right], \dots, \left( \frac{m-1}{m}, 1 \right]\right\}
\]
of the interval $(0,1]$ into $m$ intervals of length $\frac{1}{m}$.
For each $b\in \Lambda$, we denote by $I(b,m)$ the interval in $\mathcal{P}_m$ such that $b\in I(b,m)$.
For each $b\in \Lambda$, we define
$b_m \coloneqq \sup\{ x \mid x\in I(b,m)\cap \Lambda\}$.
For every $b\in \Lambda$ and $m$ positive integer,
we have $b\leq b_m$.
On the other hand, for $m$ divisible enough we have
$b=b_m$.
We define the set 
$\mathcal{C}_m \coloneqq \{b_m \mid b\in \Lambda\}$.
This is the $m$-truncation of the set $\Lambda$.
The set $\mathcal{C}_m$ is finite.
Moreover, we have an equality
\[
\overline{\Lambda} = \bigcup_{m \in \mathbb{N}} \mathcal{C}_m.
\]
Let $B$ be a boundary divisor
with prime decomposition $B=\sum_j b^jB^j$ such that the $b^j$ belong to $\Lambda$.
We define its $m$-truncation to be
$B_m \coloneqq  \sum_j b^j_m B^j$.
By the above discussion on $m$-truncations of elements of $\Lambda$, it follows that $B\leq B_m$ for every $m$,
and $B=B_m$ for $m$ divisible enough.
\end{notation}

The following lemma is a version of~\cite{FM18}*{Lemma 3.2} for $\qq$-complemented $3$-folds.

\begin{lemma}\label{coefficients-perturbation}
Let $\Lambda \subset \qq$ be a set satisfying the descending chain condition with rational accumulation points. There exists a natural number $m$, only depending on $\Lambda$, satisfying the following.
Let $X\rightarrow T$ be a contraction between normal quasi-projective varieties and $t\in T$ a closed point,
where $(X,B)$ is a log canonical $3$-fold, such that
\begin{itemize}
\item $(X,B)$ is $\qq$-complemented over $T$; and
\item ${\rm coeff}(B)\subset \Lambda$.
\end{itemize}
Let $B_m$ be as in Notation \ref{notation}.
Assume $X$ is $\qq$-factorial.
Then, $(X,B_m)$ is log canonical and $\qq$-complemented over $t\in T$.
\end{lemma}

\begin{proof} We proceed in steps.

\textbf{Step 1:} In this step, we prove that for $m$ large enough the pairs $(X,B_m)$ are log canonical.
\newline
We proceed by contradiction.
Assume this is not true.
Then, there exists a sequence of pairs $(X_i,B_i)$ as in the statement, such that
$(X_i,B_{i,i})$ is not log canonical for all $i$.
Here, $B_{i,i}$ is the $i$-th truncation
of the boundary $B_i$ as in Notation~\ref{notation}.
We claim that we can find boundaries
$B_i \leq \Delta_i \leq B_{i,i}$ and
prime divisors $D_i$ such that 
\begin{equation}\label{ineqcoeff}
{\rm coeff}_{D_i}(B_i) \leq {\rm coeff}_{D_i}(\Delta_i) < {\rm coeff}_{D_i}(B_{i,i}),
\end{equation}
all the remaining coefficients of $\Delta_i$ belong to $\overline{\Lambda}$,
and 
\[
{\rm coeff}_{D_i}(\Delta_i) = \lct(X_i,B_i; D_i).
\]
In what follows, we will write 
\[
B_i = \sum_j b_i^{j} B_{i}^{j},
\]
where the $B_{i}^{j}$ are pairwise distinct
prime divisors and $b^{j}_i \in \Lambda$. 
We construct $\Delta_i$ by successively increasing the coefficients of $B_i$ 
that are different from the coefficients of $B_{i,i}$. 
Indeed, if 
\[
\lct (X_i,B_i; B_i^{1}) \geq b_{i,i}^{1}-b_i^{1},
\]
we can increase $b_i^{1}$ to  $b_{i,i}^{1}$ 
and the pair will remain log canonical.
By abusing notation, we will denote the new boundary by $B_i$.
We proceed inductively with the other coefficients.
Since $(X_i,B_{i,i})$ is not log canonical,
we eventually find $j_i$ such that 
\[
\beta_i^{j_i}=\lct (X_i,B_i ; B_i^{j_i}) < b_{i,i}^{j_i}-b_{i}^{j_i},
\]
so we may increase $b_i^{j_i}$ to $\beta_i^{j_i}$, and we obtain the desired $\Delta_i$ by setting $D_i= B_i^{j_i}$.
We denote by $\Delta'_i$ the divisor obtained from $\Delta_i$ by reducing the coefficient at the prime divisor $D_i$ to zero.
Observe that the coefficients of $\Delta'_i$ belong to the set $\overline{\Lambda}$, which satisfies the descending chain condition.
\newline
We claim that the log canonical thresholds of $(X_i, \Delta'_i)$ with respect to $D_i$
form an infinite increasing sequence. This will provide the required contradiction.
Let 
\[
c \coloneqq \limsup_i \left({\rm coeff}_{D_i}(B_{i,i}) \right).
\]
By the construction in Notation \ref{notation}, we have ${\rm coeff}_{D_i}(B_{i,i}-B_i) \leq \frac{1}{i}$.
Hence, by~\eqref{ineqcoeff}, for every $\delta>0$, we may find $i$ large enough such that
\[
{\rm coeff}_{D_i}(\Gamma_i) \in (c-\delta,c). 
\]
Thus, passing to a subsequence, we obtain an infinite increasing sequence
\[
{\rm coeff}_{D_i}(\Delta_i) = \lct( X_i,\Delta'_i ; D_i),
\]
contradicting the ascending chain condition for log canonical thresholds \cite{HMX14}.

\textbf{Step 2:} In this step, we pass to a $\qq$-factorial dlt model.
\newline
Let $(X,B)$ be a pair as in the statement.
Let $B+B'$ be a $\qq$-complement over $T$.
Write $(Y,B+B'+E)$ for a $\qq$-factorial dlt modification of $(X,B+B')$.
Here, as usual, we redefine $B$ and $B'$ to make $E$ contain all the log canonical centers of $(Y,B+B'+E)$.
In particular, for every $\epsilon$ small enough,
the pair $(Y,B+E+(1+\epsilon) B')$ is dlt and effective over $T$.
Write
\begin{equation}\label{mmp-on-negative-kx}
-\epsilon(K_Y+B_{Y,m}+E) \sim_{\qq,T}
K_Y+B_Y+\epsilon(B_Y-B_{Y,m})+(1+\epsilon)B'_Y +E.
\end{equation}
The pair on the right-hand side is dlt provided that $\epsilon$ is small enough since the support of $B_{Y,m}$ is equal to the support of $B_Y$.
Hence, we may run a minimal model program 
for $-(K_Y+B_{Y,m}+E)$.
Observe that, in order to prove that $(X,B_m)$ is $\qq$-complemented over $T$ for $m$ large enough, it suffices to prove that
$(Y,B_{Y,m}+E)$ is $\qq$-complemented over $T$ for $m$ large enough.
Indeed, the push-forward of a $\qq$-complement for
$(Y,B_{Y,m}+E)$ to $X$ will give the desired
$\qq$-complement for $(X,B_m)$.

\textbf{Step 3:} In this step, we prove that for $m$ large enough, the pair
$-(K_Y+B_{Y,m}+E)$ is pseudo-effective over $T$.
\newline
Assume this is not the case.
We can find a sequence of pairs $(Y_i,B_i+E_i)$
such that $-(K_{Y_i}+B_{i,i}+E_i)$ is not pseudo-effective over $T$.
By the $\qq$-linear equivalence of~\eqref{mmp-on-negative-kx}, 
we may run a minimal model program for $-(K_{Y_i}+B_{i,i}+E_i)$ over $T$,
which terminates in a Mori fiber space $Y'_i\rightarrow Z_i$.
Since $(Y_i,B_i+E_i)$ is $\qq$-complemented over $T$,
we deduce that the corresponding pair 
$(Y'_i,B'_i+E'_i)$ is $\qq$-complemented over $T$
and log canonical.
Observe that all the assumptions of the first step are preserved, hence we may assume that every
$(Y'_i,B'_{i,i}+E'_i)$ is log canonical,
up to passing to a sub-sequence.
Hence, perturbing the coefficients of $B_i'$ as in the first step, we can produce boundaries
$B_i'\leq \Delta_i' < B'_{i,i}$ and prime divisors $D'_i$ that are ample over $Z_i$ such that
\[
{\rm coeff}_{D'_i}(B'_i) \leq {\rm coeff}_{D'_i}(\Delta_i') < {\rm coeff}_{D'_i}(B_{i,i}'),
\]
all the remaining coefficients of $\Delta_i'$ belong to $\overline{\Lambda}$, and
\[
-(K_{X_i'}+\Delta_i'+E'_i) \equiv 0/Z_i'.
\]
Indeed, $\K X_i'. + B' \subs i,i. + E'_i$ is ample over $Z_i$, while $\K X_i'. + B' \subs i. + E'_i$ is anti-nef over $Z_i$.
Thus, as the relative Picard number is 1, we may increase the coefficients of $B'_i$ to the coefficients of $B' \subs i,i.$ one at the time, until we obtain a relatively ample pair; the prime divisor  determining this change will be denoted by $D_i'$.
Passing to a subsequence, we may assume that ${\rm coeff}_{D_i'}(\Delta_i')$ forms an infinite increasing sequence, so the coefficients of the divisors $\Delta_i'$ belong to an infinite set satisfying the descending chain condition. By restricting to a general fiber of $Y'_i \rightarrow Z_i$, we get a contradiction by the global ascending chain condition for generalized pairs (see ~\cite{HMX14}*{Theorem 1.6}).

\textbf{Step 4:} In this step, we prove that for $m$ large enough, the pair $(Y,B_{Y,m}+E)$ is $\qq$-complemented.
\newline
By the last claim, we may assume that 
$-(K_Y+B_{Y,m}+E)$ is pseudo-effective over $T$.
By the $\qq$-linear equivalence~\eqref{mmp-on-negative-kx},
we find a good minimal model 
$Y'$ for $-(K_Y+B_{Y,m}+E)$.
Observe that all the hypotheses of the first step are preserved. Hence, we may assume that 
$-(K_{Y'}+B_{Y',m}+E')$ is semi-ample over $T$ and 
log canonical.
Hence, the pair $(Y',B_{Y',m}+E')$ has a $\qq$-complement over $t\in T$.
Pulling back this $\qq$-complement to $Y$, we obtain a $\qq$-complement for $(Y,B_{Y,m}+E)$,
which pushes forward to a $\qq$-complement for $(X,B_m)$,
proving the lemma.
\end{proof}

\begin{proof}[Proof of Theorem~\ref{main-theorem}]
By Lemma~\ref{coefficients-perturbation}, 
it suffices to prove the case of a finite set of coefficients.
Indeed, if $(X,B)$ is as in the statement of the theorem, we may find $m$, only depending on $\Lambda$, such that $(X,B_m)$ is log canonical, $B\leq B_m$ and $(X,B_m)$ is $\qq$-complemented over $T$.
In particular, by monotonicity, it suffices to produce an $n$-complement for $(X,B_m)$, where $n$ only depends on $m$. 
By Theorem~\ref{main-theorem-hyper}, we may find $n$, only depending on $m$, such that
\[
\Gamma \sim -n(K_X+B_m),
\]
and $(X,B_m+\Gamma/n)$ is log canonical.
Since $n$ only depends on $m$ and
$m$ only depends on $\Lambda$,
then $n$ only depends on $\Lambda$.
This proves the main theorem.
\end{proof}

\begin{proof}[Proof of Corollary~\ref{introcor}]
By Theorem~\ref{main-theorem}, there is an $n$-complement
$(X,B+\Gamma/n)$ for $(X,B)$ over $T$, where
$n$ only depends on $\Lambda$.
However, since $(X,B)$ is $\qq$-trivial over $T$,
we conclude that $\Gamma/n$ is an effective divisor that is $\qq$-trivial over $T$.
In particular, it either contains the fiber over $t\in T$,
or its image on $T$ is disjoint from $t$.
In the former case, the log pair $(X,B+\Gamma/n)$ would not be log canonical over $t\in T$, because
$(X,B)$ already contains a log canonical center in the fiber over $t$.
Thus, we conclude that the image of $\Gamma$ on $T$ is disjoint from $t$.
Otherwise, the pair $(X,B+\Gamma/n)$ would not be log canonical as $x$ is alreadya  log canonical cneter of $(x\in X,B)$.
Shrinking around $t$, we may assume that $\Gamma=0$ and
then $n(K_X+B)\sim_{T} 0$ as claimed.
\end{proof}

\begin{proof}[Proof of Corollary~\ref{introcor1}]
  We will follow the notation as in \cite{Fuj01}.
  Notice that, in \cite{Fuj01}*{Theorem 0.2}, the conclusions of the statement are proved in arbitrary dimension $n$ conditionally to two conjectures, denoted by $(F'_i)$ and $(F_j)$ (see \cite{Fuj01}*{Conjecture 3.2}). 
  More precisely, if $(F'_3)$ holds, and $(F_l)$ holds for $l \leq 2$, then the statement follows.
  Since $(F_2)$ is proved in ~\cite{Fuj01}*{Proposition 3.6},  it suffices to prove $(F_3')$.
  For the reader's convenience, we recall its statement.
  
  \begin{itemize}
      \item[($F_3'$):] There exists $m$ such that order of $\rho(g)$ in $\mathrm{GL}(H^0(S,K_S))$ is bounded above by $m$ for every $3$-dimensional variety $S$ with only canonical singularities such that $K_S \sim 0$ and for every $g \in \mathrm{Aut}(S, 0) = \mathrm{Aut}(S)$ such that $g$ has a finite order.
  Here, $\rho \colon \mathrm{Aut}(S)\rightarrow \mathrm{GL}(H^0(S,K_S))$ is the standard canonical representation map.
  \end{itemize}
  Let $S$ and $g$ be as in the statement of $(F'_3)$. 
  Firstly, we remark that $S$ projective, since in the proof of \cite{Fuj01}*{Theorem 0.2} $S$ is taken to be a component of the exceptional divisor above the isolated log canonical singularity.
  More precisely, in the proof of \cite{Fuj01}*{Theorem 0.2}, $(F'_{n-1})$ is needed in \cite{Fuj01}*{Proposition 4.18, Proposition 4.19}.
  In \cite{Fuj01}*{Proposition 4.18, Proposition 4.19}, it is assumed $\mu=n-1$, where $\mu$ is defined in \cite{Fuj01}*{Definition 4.12}.
  By the definition of $\mu$, this exactly means that we are focusing on an exceptional divisor over an isolated singularity.
  \newline
  Then, let $X \coloneqq S/\langle g \rangle$ denote the quotient of $S$ by the group generated by $g$.
  Let $f \colon S \rar X$ denote the quotient morphism.
  Then, by the Riemann--Hurwitz formula, there exists a boundary $B$ with coefficients in $\Phi(\lbrace 0,1 \rbrace)$ such that $K_S = f^*(\K X. + B)$.
  Since $\K S. \sim 0$, it follows that $\K X. + B \sim_\qq 0$.
  Furthermore, by \cite{KM98}*{Proposition 5.20}, $(X,B)$ is klt.
  Then, we are done, since the order of $\rho(g)$ is bounded above by the index of $\K X. + B$.
  Indeed, by \cite{thesis}*{Theorem 1.8.7}, $m (\K X. + B)\sim 0$ for some $m \in \mathbb N$ independent of $S$. If $m(K_X+B)\sim 0$, then $g|_X =id_X$ acts trivially on $H^0(X,m(K_X+B))$.
  Since the sections of $H^0(X,m(K_X+B))$ and $H^0(S,m\K S.)$ are identified to each other by pull-back, $g$ acts trivially on $H^0(S,mK_S)$.
  Therefore, $\rho(g)^m=1\in \mathrm{GL}(H^0(S,K_S))$, as $K_S\sim 0$.
  This concludes the proof.
\end{proof}

\end{document}